\documentclass[12pt]{article}
\usepackage{mathrsfs}
\usepackage[active]{srcltx}
\usepackage[curve]{xypic}
\usepackage{amsthm,amsfonts,amscd,amsmath,amssymb,color,array,verbatim,graphicx,caption2}
\usepackage[usenames,dvipsnames]{pstricks}
\usepackage{epsfig}

\def\mystrut{{\rule[-2ex]{0ex}{4.5ex}{}}}
\input epsf.tex

\newtheorem{theorem}{Theorem}[section]
\newtheorem{lemma}[theorem]{Lemma}
\newtheorem{proposition}[theorem]{Proposition}
\newtheorem{corollary}[theorem]{Corollary}
\newtheorem{definition}{Definition}

\def\al{\alpha}
\def\be{\beta}
\def\g{\gamma}
\def\G{\Gamma}

\def\la{\lambda}
\def\La{\Lambda}

\def\de{\delta}
\def\wt{\widetilde}
\def\wh{\widehat}

\def\R{\mbox{$\mathbb R$}}

\def\C{\mbox{$\mathbb C$}}

\def\D{\mbox{$\mathbb D$}}

\def\N{\mbox{$\mathbb N$}}

\def\SSS{{\mathcal S}}
\def\TTT{{\mathcal T}}
\def\AAA{{\mathcal A}}
\def\BBB{{\mathcal B}}
\def\CCC{{\mathfrak C}}
\def\DDD{{\mathcal D}}
\def\EEE{{\mathcal E}}
\def\FFF{{\mathcal F}}
\def\GGG{{\mathcal G}}

\def\III{{\mathcal I}}
\def\KKK{{\mathcal K}}
\def\JJJ{{\mathcal J}}

\def\OOO{{\mathcal O}}

\def\PPP{{\mathcal P}}
\def\RRR{{\mathcal R}}
\def\UUU{{\mathcal U}}
\def\VVV{{\mathcal V}}

\def\cbar{\widehat{\C}}
\def\smm{{\backslash}}
\def\m{\text{mod}}
\def\ds{\displaystyle}
\usepackage[english]{babel}
\usepackage[utf8]{inputenc}
\usepackage[T1]{fontenc}
\begin{document}
\title{Renormalizations and wandering \\ Jordan curves of rational maps
\footnote{2010 Mathematics Subject Classification: 37F10, 37F20}}
\author{Guizhen Cui, Wenjuan Peng and Lei Tan}
\date{\today}
\maketitle
\begin{abstract}
We realize a dynamical decomposition for a post-critically finite rational map which admits a combinatorial decomposition. We split the Riemann sphere into two completely invariant subsets. One is a subset of the Julia set consisting of uncountably many Jordan curve components. Most of them are wandering. The other consists of components that are pullbacks of finitely many renormalizations, together with possibly uncountably many points. The quotient action on the decomposed pieces is encoded by a dendrite dynamical system. We also introduce a surgery procedure to produce post-critically finite rational maps with wandering Jordan curves and prescribed renormalizations.
\end{abstract}

\section{Introduction}

A rational map of one complex variable acts on the Riemann sphere and generates a dynamical system by iteration. One general principle in analyzing the dynamical system is to decompose it into invariant sub-systems with simpler dynamics. The Riemann sphere is canonically decomposed into the disjoint union of the Fatou set and the Julia set, defined by whether the iterated sequence forms a normal family in a neighborhood of the point. While the dynamics on the Fatou set is relatively tame, the dynamics on the Julia set is wild and chaotic.

It may happen that there exists a periodic continuum in the Julia set, where the first return map behaves like another rational map on its own Julia set. The induced dynamical system is usually called a {\it renormalization}. It may also happen that a continuum in the Julia set wanders around and forms a pairwise disjoint orbit.  The continuum is usually called a {\it wandering continuum}.

Many existing decomposition results can be described by stable multicurves. For a polynomial with a disconnected Julia set such that only finitely many filled Julia components contain post-critical points, one can extract a stable multicurve from a finite pullback of an equipotential curve in the basin of infinity. This stable multicurve forms a tableau under pullback. The pioneering work of Branner-Hubbard on cubic polynomials uses this tableau to show that the filled Julia set consists of components that are either pullbacks of renormalizations or points \cite{BH}. Later on DeMarco and Pilgrim also use this tableau and an infinite tree to encode the combination of these points and renormalization copies \cite{DP}.

For polynomials with connected Julia sets there are similar decomposition results. Often there are more than one periodic external rays landing at a common periodic point. These periodic external rays cut the Julia set into pieces. Together with equipotential curves, these rays play the role of a stable multicurve and form puzzles.  As before the dynamics on a periodic puzzle piece is a renormalization. In various expanding cases, the remaining part of the Julia set, after extracting the renormalized copies and their pullbacks, consists of uncountable many point components. See for example \cite{DH1, Ly,Y}, among others.

For non-polynomial rational maps, the situation is more complicated. There is an example of hyperbolic rational maps whose Julia set is a Cantor set of Jordan curves \cite{Mc1}. A Jordan curve contained essentially in an annular Fatou domain forms a stable multicurve. In general, if a sub-hyperbolic rational map has a disconnected Julia set, one can also extract a stable multicurve in the multiply-connected Fatou domains and obtain a canonical decomposition of the Julia set. It turns out that every wandering component of the Julia set is either a single point or a Jordan curve, and there are only finitely many cycles of non-trivial periodic components of the Julia set except the cycles of Jordan curves. A detailed analysis can be found in \cite{CT, PT}.

The situation of a rational map with a connected Julia set is actually much harder. In some particular cases such as the Newton's method for cubic polynomials, or quadratic rational maps with a $2$-periodic critical point, rays from distinct Fatou basins may join together at their landing points and form graphs to cut apart the Julia set (see e.g. \cite{Luo, Ro, Tan2}). But this is by far the general case. For instance there are Julia sets  homeomorphic to a Sierpinski carpet, in which any two distinct Fatou domains have disjoint closures.

The Julia set of a post-critically finite rational map is automatically connected. Conversely, any sub-hyperbolic rational map with a connected Julia set is quasi-conformally conjugate to a post-critically finite rational map in a neighborhood of the Julia set. There is already a well-known combinatorial decomposition for post-critically finite maps, in case that the map has a fully invariant Jordan curve up to homotopy. Cutting along this curve will decompose the rational map (or its second iteration if necessary) into two polynomials. The rational map is called the mating of the two polynomials. See for example \cite{ST, Tan}. However, the decomposed polynomials can not be considered as renormalizations in the usual sense. In general none of the small Julia sets is embedded in the original big Julia set. Distinct tips of a small Julia set are often identified under mating \cite{Shi}.

In this work we will establish a new type of decomposition procedure for post-critically finite rational maps. The first challenge is to find a natural class of multicurves to cut along.  Our key concept, {\it Cantor multicurves}, is introduced precisely for this purpose. Contrary to the equator of a mating, a Cantor multicurve is a multicurve whose consecutive pullbacks will generate a strictly increasing number of curves in each homotopy class.

The stability of multicurves is only measured up to homotopy. The crucial step in our study is to promote a Cantor multicurve to a multi-annulus such that it is exactly invariant in a sense defined later on. We will call such a dynamical system an {\it exact annular system}. This is to be compared to the monograph \cite{Pi}, where stable multicurves are promoted to invariant Jordan curves in the setting of branched coverings. It will play the role of multiply-connected Fatou components in the disconnected case, and will allow us to decompose the Julia set into pairwise disjoint pieces.

Cantor multicurves and exact annular systems appear naturally in the study of rational maps with disconnected Julia sets (see e.g. \cite{Go, Mc1, PT}). In the connected Julia set case, these concepts appear also in Haissinsky-Pilgrim's example \cite{HP}. However they have never been applied to general post-critically finite rational maps.

The dynamics of the exact annular sub-system is quite simple since it is expanding. Similar to the disconnected Julia set case we will prove that each component of its Julia set is a Jordan curve. There are uncountably many of such components.
All but countably many of them are wandering.

The next step is to analyze the complement of the grand orbit of the Julia set of the exact annular sub-system. It has countably many continuum components which are pullbacks of finitely many renormalizations and possibly uncountably many point components. Finally, we will encode the induced dynamics on these decomposed components by an expanding dynamical system on a dendrite.

We also present a new result about wandering continua. For a post-critically finite rational map, the existence of a non-simply connected wandering continuum is equivalent to the existence of a Cantor multicurve. Moreover, the wandering continuum must be a Jordan curve. It is known that for a polynomial with connected Julia set and without irrational indifferent cycles, there are no wandering continua if and only if its Julia set is locally connected \cite{BL, Ki1, Ki2, Le, Th}.

Our final and somewhat independent task is to construct post-critically finite rational maps with Cantor multicurves and prescribed renormalizations. For this purpose we will introduce a new surgery procedure which we call {\it folding}, to construct a branched covering from a polynomial. The resulting branched covering has a Cantor multicurve consisting of a single curve. We will show that under certain conditions the branched covering is Thurston equivalent to a rational map. Consequently, this rational map has a Cantor multicurve. Hence it has wandering Jordan curves and a renormalization whose straightening map is the polynomial we started with.

\vskip 0.24cm

{\bf Statements}.
Let $F$ be a branched covering of the Riemann sphere $\cbar$. We always assume $\deg F\ge 2$ in this paper. Denote by $\Omega_F$ the set of critical points of $F$. The {\bf post-critical set}\ of $F$ is defined by
$$
\PPP_F=\overline{\bigcup_{n\ge 1}F^n(\Omega_F)}.
$$
The map $F$ is called {\bf post-critically finite}\ if $\PPP_F$ is a finite set.

A Jordan curve $\g$ in $\cbar\smm\PPP_F$ is {\bf null-homotopic} if one of its complementary components contains no point of $\PPP_F$, or {\bf peripheral} if one of its complementary components contains exactly one point of $\PPP_F$, or {\bf essential} otherwise, i.e. if each of its two complementary components contains at least two points of $\PPP_F$.

A {\bf multicurve}\ $\Gamma$ is a non-empty and finite collection of disjoint Jordan curves in $\cbar\smm\PPP_F$, each essential and no two homotopic rel $\PPP_F$.

Given $\G$,  a collection of curves in $\cbar$, we also use $\G$ to denote the union of curves in $\G$ as a subset of $\cbar$ if there is no confusion.

Let $\G$ be a multicurve of $F$. Denote by $\G_1$ the collection of curves in $F^{-1}(\G)$ homotopic rel $\PPP_F$ to a curve in $\G$. Inductively, for $n\ge 2$, denote by $\G_n$ the collection of curves in $F^{-1}(\G_{n-1})$ homotopic rel $\PPP_F$ to a curve in $\G$. Note that $\G_n$ is contained in, but may not equal, the collection of curves in $F^{-n}(\G)$ homotopic rel $\PPP_F$ to curves in $\G$. For each $\g\in\G$, denote
$$
\kappa_n(\g)=\#\{\beta\in\G_n:\, \beta\text{ is homotopic to $\g$ rel $\PPP_F$}\}.
$$

\begin{definition}
A multicurve $\G$ is called a {\bf Cantor multicurve} if $\kappa_n(\g)\to\infty$ as $n\to\infty$ for all $\g\in\G$.
\end{definition}

A post-critically finite rational map $f$ need not have even a multicurve, e.g. $\#\PPP_f\le 3$. A topological polynomial, i.e. a branched covering with a totally invariant point, has no Cantor multicurves. In the classical construction of mating of polynomials, there is a Jordan curve whose pre-image is a single curve homotopic to itself rel the post-critical set. In this case the multicurve consisting of this single curve is not a Cantor multicurve. Concrete examples of rational maps with Cantor multicurves will be constructed in \S8.

\vskip 0.24cm

The following notations and definitions will be used throughout this paper.

\indent $\bullet$ Let $f$ be a rational map. We will denote by $\JJJ_f$ the {\bf Julia set} of $f$ and $\FFF_f$ the {\bf Fatou set} of $f$. We refer to \cite{Bea, CG, Mc2, Mi} for the definitions and basic properties. \\
\indent $\bullet$ Let $U,V\subset\cbar$ be open sets. We write $U\Subset V$ if $\overline{U}\subset V$.  \\
\indent $\bullet$ Let $A\subset\cbar$ be an annulus and $E\subset A$ be a connected open or closed subset. We say that $E$ is contained {\bf essentially} in $A$ if $E$ separates the boundary $\partial A$. \\
\indent $\bullet$ A {\bf multi-annulus}\ is a finite disjoint union of annuli in $\cbar$ with finite modulus. \\
\indent $\bullet$ Let $f$ be a post-critically finite rational map. An annulus $A\subset\cbar\smm\PPP_f$ is {\bf homotopic} rel $\PPP_f$ to a Jordan curve $\g$ (or an annulus $A'$) in $\cbar\smm\PPP_f$ if essential Jordan curves in $A$ are homotopic to $\g$ (or essential curves in $A'$) rel $\PPP_f$. A multi-annulus $\AAA\subset\cbar\smm\PPP_f$ is {\bf homotopic} rel $\PPP_f$ to a multicurve $\G$ (or a multi-annulus $\AAA'\subset\cbar\smm\PPP_f$) if each component of $\AAA$ is homotopic to a curve in $\G$ (or exactly one component of $\AAA'$)) rel $\PPP_f$ and each curve in $\G$ (or each component of $\AAA'$) is homotopic to exactly one component of $\AAA$.

\begin{definition} \label{defAS} Let $\AAA^1, \AAA\subset\cbar$ be two multi-annuli such that each component of $\AAA^1$ is contained essentially in a component of $\AAA$. A map $g: \AAA^1\to \AAA$ is called an {\bf annular system}\ if \\
\indent (1) for each component $A^1$ of $\AAA^1$, its image $g(A^1)$ is a component of $\AAA$ and the map $g:A^1\to g(A^1)$ is a holomorphic covering; \\
\indent (2) there is an integer $n\ge 1$ such that for each component $A$ of $\AAA$, the set $g^{-n}(\AAA)\cap A$ is non-empty and disconnected. \\
\indent The {\bf Julia set} of $g$ is defined by $\JJJ_g:=\bigcap_{n\ge 0}g^{-n}(\AAA)$. An annular system $g: \AAA^1\to \AAA$ is called {\bf proper}\ if $\AAA^1\Subset\AAA$; or {\bf exact}\ if for every component $A$ of $\AAA$, each of the two components of $\partial A$ is also a component of $\partial(A\cap \AAA^1)$.
\end{definition}

{\bf Remark}. Our definition of the Julia set $\JJJ_g$ of an annular system $g$ is perhaps slightly misleading. At first, $\JJJ_g$ need not be compact. Secondly, $g^{-1}(\JJJ_g)=\JJJ_g$ and $g(\JJJ_g)\subset\JJJ_g$ from the definition but $g(\JJJ_g)$ need not equal $\JJJ_g$ since the map $g$ may not be onto.

\begin{theorem}\label{as}{\bf (from a Cantor multicurve to an annular system)}
Let $f$ be a post-critically finite rational map with a Cantor multicurve $\Gamma$. Then there exists a unique multi-annulus $\AAA\subset\cbar\smm\PPP_f$ homotopic rel $\PPP_f$ to $\Gamma$ such that $g=f|_{\AAA^1}: \AAA^1\to \AAA$ is an exact annular system, where $\AAA^1$ is the union of components of $f^{-1}(\AAA)$ that are homotopic rel $\PPP_f$ to curves in $\Gamma$. Moreover, $\JJJ_g\subset\JJJ_f$. $\JJJ_g$ has uncountably many components. Each of these is a Jordan curve and all but countably many of them are wandering.
\end{theorem}

A {\bf continuum} $E\subset\cbar$ is a connected compact subset containing more than one point. It is called {\bf simply connected} if $\cbar\smm E$ has exactly one component, or {\bf $n$-connected} with $n\in\N\cup\{\infty\}$ if $\cbar\smm E$ has exactly $n$ components.

\begin{definition}
Let $f$ be a rational map. By a {\bf wandering continuum} we mean a continuum $K\subset\JJJ_f$ such that $f^n(K)\cap f^m(K)=\emptyset$ for any $n>m\ge 0$.
\end{definition}

\begin{theorem}\label{wc}{\bf (from a wandering continuum to a Cantor multicurve)}
Let $f$ be a post-critically finite rational map. Suppose that $K\subset\JJJ_f$ is a non-simply connected wandering continuum. Then $K$ and all its forward iterates are Jordan curves. For $n>0$ large enough, the homotopy classes of $\{f^n(K)\}$ rel $\PPP_f$ form a Cantor multicurve.
\end{theorem}

A connected subset $E\subset\cbar$ is called of {\bf simple type} (w.r.t. $\PPP_f$) if either there exists a simply connected domain $U\subset\cbar$ such that $E\subset U$ and $U$ contains at most one point in $\PPP_f$, or there exists an annulus $A\subset\cbar\smm\PPP_f$ such that $E\subset A$. It is called of {\bf complex type} (w.r.t. $\PPP_f$) otherwise.

Let $f$ be a post-critically rational map with a stable Cantor multicurve $\Gamma$. Denote by $\JJJ(\Gamma)$ the union of the grand orbit under $f$ of the Julia set of the annular sub-system derived from Theorem \ref{as}. Set $\KKK(\Gamma)=\cbar\smm\JJJ(\Gamma)$. We refer to \S5 and \S6 for the definitions of renormalization and finite dendrite map.

\begin{theorem}\label{decomposition}{\bf (renormalization)} Let $f$ be a post-critically rational map with a stable Cantor multicurve $\Gamma$. Then every component of $\KKK(\Gamma)$ is either a single point or a continuum which is eventually periodic. There are only finitely many periodic components of $\KKK(\G)$ which are continua. Each of them is either the closure of a quasi-disk or a complex type continuum. The former is the closure of a periodic Fatou domain,  while the latter is the filled Julia set of a renormalization.
\end{theorem}

\begin{theorem}\label{coding}{\bf (coding)}
There exist an induced expanding finite dendrite map $\tau:\,\TTT\to\TTT$ and a continuous semi-conjugacy $\Theta$ from $f$ to $\tau$ such that for each regular point $t\in\TTT$, the fiber $\Theta^{-1}(t)$ is a component of $\JJJ(\G)$, and for each vertex $t\in\TTT$, the fiber $\Theta^{-1}(t)$ is a component of $\KKK(\G)$.
\end{theorem}

\begin{theorem}\label{folding}{\bf (folding)}
Let $g$ be a post-critically finite polynomial with $\#\PPP_g\ge 3$. Then there exist a post-critically finite rational map $f$ and a Cantor multicurve $\G$ consisting of a single curve, such that $\KKK(\G)$ consists of the grand orbit of a unique fixed component $K$. Moreover, $K$ is the filled Julia set of a renormalization whose straightening map is the polynomial $g$.
\end{theorem}

{\bf Outline of the paper}. This paper is organized as follows. In \S2, we recall Thurston's theory. Some equivalent conditions of Cantor multicurves are given here. In \S3, we show that every component of the Julia set of an exact annular system is a Jordan curve if the annular system is expanding. Theorem \ref{as} is proved in \S4. In \S5, we will study the decomposition pieces and prove Theorem \ref{decomposition}. In \S6, we introduce the definition of finite dendrite maps and prove Theorem \ref{coding}. Theorem \ref{wc} is proved in \S7. In \S8, we will introduce the folding surgery to construct branched coverings from polynomials and provide several theorems, which are more precise than Theorem \ref{folding}.

\section{Thurston's Theorem and multicurves}

We will recall Thurston's characterization theorem at first. Let $F$ be a post-critically finite branched covering of $\cbar$. Let $\G$ be a multicurve. Its {\bf transition matrix} $M_{\G}=(a_{\g\beta})$ is defined as:
$$
a_{\gamma\beta}=\sum_{\alpha}\frac 1{\deg(F: \alpha\to\beta)},
$$
where the summation is taken over components $\alpha$ of $F^{-1}(\beta)$ which are homotopic to $\gamma$ rel $\PPP_F$. Denote by $\lambda_{\G}$ the leading eigenvalue of its transition matrix $M_{\G}$.

A multicurve $\G$ is called {\bf stable} if
for each $\g\in\G$
each essential curve in $F^{-1}(\g)$
	is homotopic rel $\PPP_F$ to a curve in $\G$. A stable multicurve $\Gamma$ is called a {\bf Thurston obstruction}\ if $\la_{\G}\ge 1$.

Two post-critically finite branched coverings $F$ and $G$ are called {\bf Thurston equivalent}\ if there is a pair of homeomorphisms $(\phi,\psi):\cbar\to\cbar$ such that $\phi$ is isotopic to $\psi$ rel $\PPP_F$ and $\phi\circ F\circ \psi^{-1}=G$. We refer to \cite{DH3} or \cite{Mc2} for the next theorem and the definition of a hyperbolic orbifold.

\begin{theorem}\label{Thurston} {\bf (Thurston)} Let $F$ be a post-critically finite branched covering of $\cbar$ with a hyperbolic orbifold. Then $F$ is Thurston equivalent to a rational map $f$ if and only if $F$ has no Thurston obstruction. Moreover, the rational map $f$ is unique up to holomorphic conjugation.
\end{theorem}

A multicurve $\G$ is called {\bf irreducible} if for each pair $(\g,\be)\in\G\times\G$, there exists a sequence $\{\delta_0=\g, \delta_1, \cdots, \delta_n=\beta\}$ of curves in $\G$ such that
for every $1\le k\le n$,
$F^{-1}(\delta_k)$ has a component homotopic to $\delta_{k-1}$ rel $\PPP_F$

A multicurve $\G$ is called {\bf pre-stable} if each curve in $\G$ is homotopic rel $\PPP_F$ to a curve in $F^{-1}(\G)$. It is easy to check that an irreducible multicurve is pre-stable.

\begin{lemma}\label{Mc} Let $F$ be a branched covering of $\cbar$. For any pre-stable multicurve $\G_0$, there is a stable and pre-stable multicurve $\G$ such that $\G\supset\G_0$ and hence $\lambda_{\G_0}\le\lambda_{\G}$. Conversely, for any stable multicurve $\G$ with $\lambda_{\G}>0$, there is an irreducible multicurve $\G_0\subset\G$ such that $\lambda_{\G_0}=\lambda_{\G}$.
\end{lemma}

Refer to \cite{Mc2} for the proof of the second part. We prove the first part here.

\begin{proof} For each $n\ge 1$, let $\wt\G_n$ be a multicurve consisting of curves in $F^{-n}(\G_0)$ such that any essential curve in $F^{-n}(\G_0)$ is homotopic rel $\PPP_F$ to a curve in $\wt\G_n$. Then $\wt\G_n$ is pre-stable and each curve in $\wt\G_n$ is homotopic to a curve in $\wt\G_{n+1}$. Thus, $\#\wt\G_n\le\#\wt\G_{n+1}$. Since any multicurve contains at most $\#\PPP_F-3$ curves, there is an integer $N\ge 0$ such that $\#\wt\G_N =\#\wt\G_{N+1}$. It follows that $\wt\G_N$ is a stable and pre-stable multicurve.
\end{proof}

\vskip 0.24cm

Let $\G$ be a multicurve of a post-critically finite branched covering $F$. Recall that $\G_1$ is the collection of curves in $F^{-1}(\G)$ homotopic rel $\PPP_F$ to a curve in $\G$, $\G_n$ is the collection of curves in $F^{-1}(\G_{n-1})$ homotopic rel $\PPP_F$ to a curve in $\G$ for $n\ge 2$, and for each $\g\in\G$,
$$
\kappa_n(\g)=\#\{\beta\in\G_n:\, \beta\text{ is homotopic to $\g$ rel $\PPP_F$}\}.
$$
The multicurve $\G$ is a Cantor multicurve if $\kappa_n(\g)\to\infty$ as $n\to\infty$ for all $\g\in\G$. It is easy to check that a Cantor multicurve is pre-stable.

For each $\g\in\G$, denote by $m(\g)$ the number of curves in $F^{-1}(\g)$ homotopic to a curve in $\G$. The following equations are easy to check:
$$
\#\G_1=\sum_{\g\in\G}\kappa_1(\g)=\sum_{\g\in\G}m(\g).
$$
For $n\ge 1$,
$$
\#\G_{n+1}=\sum_{\g\in\G}\kappa_{n+1}(\g)=\sum_{\g\in\G}\kappa_n(\g)\cdot m(\g) \ .
$$

\begin{lemma}\label{Cantor}
Let $\G$ be an irreducible multicurve. The following statements are equivalent: \\
\indent (1) $\#\G_1>\#\G$. \\
\indent (2) $m(\g)\ge 2$ for some $\g\in\G$. \\
\indent (3) $\kappa_1(\g)\ge 2$ for some $\g\in\G$. \\
\indent (4) $\#\G_{n+1}>\#\G_n$ for all $n\ge 1$. \\
\indent (5) $\kappa_n(\g)\to\infty$ for some $\g\in\G$. \\
\indent (6) $\kappa_n(\g)\to\infty$ for all $\g\in\G$, i.e., $\G$ is a Cantor multicurve.
\end{lemma}

\begin{proof} Since $\G$ is irreducible, both $m(\g)$ and $k_1(\g)$ are positive for each $\g\in\G$. From the first equation above, we have
$(1)\Longleftrightarrow(2)\Longleftrightarrow (3)$. From the second equation, we have $(2)\Longleftrightarrow (4)$.

If $\#\G_{n+1}>\#\G_n$ for all $n\ge 1$, then $\kappa_n(\g)\to\infty$ for some $\g\in\G$ by the second equation. Conversely, if $\#\G_{n+1}=\#\G_n$ for some $n\ge 1$, then $m(\g)=1$ for all $\g\in\G$. Thus, $\#\G_{n+1}=\#\G_n$ for all $n\ge 1$. It follows that $\kappa_n(\g)\le\#\G_n=\#\G_1$ for all $\g\in\G$. This implies that $(4)\Longleftrightarrow (5)$.

Assume that $\kappa_n(\g)\to\infty$ for some $\g\in\G$. Since $\G$ is irreducible, for each $\beta\in\G$, there is an integer $k\ge 1$ such that $F^{-k}(\g)$ has a component $\de$ homotopic to $\beta$ rel $\PPP_F$ and $F^i(\de)$ is homotopic to a curve in $\G$ for $1\le i<k$. Therefore, $\kappa_{n+k}(\beta)\ge\kappa_n(\g)$. So $\kappa_n(\beta)\to\infty$ as $n\to\infty$. This shows that $(5)\Longleftrightarrow (6)$.
\end{proof}

Let $\Gamma=\{\gamma_1, \cdots, \gamma_n\}$ be a multicurve of $F$. Its {\bf reduced transition matrix} $M_{r, \Gamma}=(b_{ij})$ is defined by $b_{ij}=k$ if there are $k$ components of $F^{-1}(\gamma_j)$ homotopic to $\gamma_i$ rel $\PPP_F$. This definition was introduced by Shishikura.

\begin{lemma}\label{eigenvalue}
Let $\Gamma$ be a pre-stable multicurve of $F$. Then the leading eigenvalue of its reduced transition matrix satisfies $\lambda(M_{r,\Gamma})\ge 1$. Moreover, $\lambda(M_{r,\Gamma})>1$ if $\Gamma$ is a Cantor multicurve.
Conversely, if $\Gamma$ is irreducible and $\lambda(M_{r,\Gamma})>1$, then $\Gamma$ is a Cantor multicurve.
\end{lemma}

\begin{proof} Note that $M_{r, \Gamma}{\bf v}\ge {\bf v}$ for the vector ${\bf v}=(1, \cdots, 1)$ since $\Gamma$ is pre-stable. Thus, $\lambda(M_{r,\Gamma})\ge 1$ by Lemma A.1 in \cite{CT}. If $\Gamma$ is a Cantor multicurve, then there exists an integer $n\ge 1$ such that $M_{r, \Gamma}^n{\bf v}\ge 2{\bf v}$. Thus, $\lambda(M_{r,\Gamma})^n=\lambda(M_{r,\Gamma}^n)>1$. Conversely, if $\Gamma$ is irreducible and $\lambda(M_{r,\Gamma})>1$, then there exists at least one column of the matrix such that the summation of the entries of this column is bigger than one. Thus, $\Gamma$ is a Cantor multicurve by Lemma \ref{Cantor} (2).
\end{proof}

\section{Annular systems}

In this section we will show that every component of the Julia set of an expanding exact annular system (see Definition \ref{defAS}) is a Jordan curve.

Let $\AAA^1, \AAA\subset\cbar$ be two multi-annuli such that each component of $\AAA^1$ is contained essentially in a component of $\AAA$. Recall that a map $g: \AAA^1\to \AAA$ is an annular system if \\
\indent (1) for each component $A^1$ of $\AAA^1$, its image $g(A^1)$ is a component of $\AAA$ and the map $g:A^1\to g(A^1)$ is a holomorphic covering; \\
\indent (2) there is an integer $n\ge 1$ such that for each component $A$ of $\AAA$, the set $g^{-n}(\AAA)\cap A$ is non-empty and disconnected.

The Julia set of $g$ is defined by $\JJJ_g:=\bigcap_{n\ge 0}g^{-n}(\AAA)$. An annular system $g: \AAA^1\to \AAA$ is proper if $\AAA^1\Subset\AAA$; or exact if for every component $A$ of $\AAA$, each of the two components of $\partial A$ is also a component of $\partial(A\cap \AAA^1)$.

\subsection{Basic properties}

\begin{proposition}\label{deg}
Let $g: \AAA^1\to \AAA$ be an annular system. There is an integer $N\ge 1$ such that $\deg (g^N|_A)\ge 2$ for each component $A$ of $g^{-N}(\AAA)$.
\end{proposition}

\begin{proof}
Let $m\ge 1$ be the number of components of $\AAA$. By contradiction we assume that there is a component $A$ of
$g^{-m}(\AAA)$ such that $\deg (g^{m}|_A)=1$. There exist integers $0\le k<k+p\le m$ such that both $g^{k}(A)$ and $g^{k+p}(A)$ are contained in the same component $A^0$ of $\AAA$. So $g^p(g^k(A))\subset A^0$. Let $A^p\subset A^0$ be the component of $g^{-p}(A^0)$ containing $g^{k}(A)$. Since $g^p: A^p\to A^0$ is a covering between annuli and $g^{k}(A)$ is contained essentially in $A^p$, we have
$$
\deg\Big(g^{p}:\, A^p\to A^0\Big)=\deg\Big(g^{p}:\, g^{k}(A)\to
g^{k+p}(A)\Big)\le\deg (g^{m}|_A)=1.
$$
Thus, the moduli of the annuli $A^p$ and $A^0$ are equal and hence $A^p=A^0$. It follows that $A^0=g^{p}(A^0)$. Therefore, $A^0\cap g^{-np}(\AAA)=A^0$ for all integers $n\ge 1$. Since $g^{-n}(\AAA)\subset g^{-n+1}(\AAA)$ for all $n\ge 1$, we conclude that $A^0\cap g^{-n}(\AAA)=A^0$ for all $n\ge 1$. This contradicts the condition that $A^0\cap g^{-n}(\AAA)$ is disconnected for some $n\ge 1$. So $\deg(g^{m}|_A)\ge 2$ for each component $A$ of $g^{-m}(\AAA)$.
\end{proof}

\begin{proposition}\label{compactly}
Let $g:\AAA^1\to \AAA$ be an exact annular system. Let $\{A^n\}$ be a nested sequence of annuli of $g^{-n}(\AAA)$, i.e. the annulus $A^n$ is a component of $g^{-n}(\AAA)$ and $A^{n+1}\subset A^n$. Then either $\bigcap_{n\ge 1}A^n=\emptyset$ or for every $n\ge 1$, there is an integer $m>n$ such that $A^m\Subset A^n$.
\end{proposition}

\begin{proof} By exactness, either for any $n\ge 1$, there is an integer $m>n$ such that $A^m\Subset A^n$, or there is an integer $N\ge 1$ and a component $L$ of $\partial A^N$ such that $L\subset\partial A^n$ for $n\ge N$.

We only need to show that $\bigcap_{n\ge 1}A^n=\emptyset$ in the latter case. Since $\AAA$ has only finitely many components, there exist a component $B^0$ of $\AAA$ and integers $i>j>k\ge N$ such that $g^i(A^i)=g^j(A^j)=g^k(A^k)=B^0$. Since $B^0$ has exactly two boundary components, there are a boundary component $L'$ of $B^0$, and two integers among $\{i,j,k\}$, say $i$ and $j$, such that $g^i(L)=g^j(L)=L'$ (This formula means that as $z$ tends to $L$ in $A^i$, both $g^i(z)$ and $g^j(z)$ tends to $L'$).

Denote by $B^n=g^j(A^{n+j})$ for $n\ge 0$, then $\{B^n\}$ is a nested sequence of annuli which has a common boundary component $L'$. Moreover, $g^p(B^p)=B^0$ and $g^p(L')=L'$ for $p=i-j$. It follows that $g^p(B^{np})=B^{(n-1)p}$ for $n\ge 1$.
Note that $B^p\neq B^0$. Otherwise $B^n=B^0$ for all $n\ge 1$ and this contradicts the condition that $B^0\cap g^{-n}(\AAA)$ is disconnected for some $n\ge 1$.

Let $U$ be the component of $\cbar\smm L'$ containing $B^0$ and $\phi$ be a conformal map from $U$ onto the unit disk $\D$. Then $h:=\phi\circ g^{p}\circ \phi^{-1}$ is a holomorphic covering from $\phi(B^{p})$ to $\phi(B^0)$, which can be extended continuously to the unit circle.

Denote by $V_1$ the union of $\phi(B^{p})$, its reflection in the unit circle and together with the unit circle itself. Similarly, let $V$ be the union of $\phi(B^0)$, its reflection in the unit circle and the unit circle. By the symmetric extension principle, $h$ can be extended to a holomorphic covering map from the annulus $V_1$ to $V$.  Since $V_1\Subset V$, $h$ is expanding under the hyperbolic metric of $V$. So $\bigcap_{n>0}h^{-n}(V)=\partial\D$ and hence $\bigcap_{n>0}h^{-n}(\phi(B^0))=\emptyset$. Note that $\phi(B^{np})=h^{-n}(\phi(B^0))$. Therefore $\bigcap_{n>0}B^{np}=\emptyset$ and hence $\bigcap_{n>0}A^n=\emptyset$.
\end{proof}

Let $g: \AAA^1\to \AAA$ be an annular system and $K$ a connected component of $\JJJ_g$. Then, for each $n\ge 0$ there is a unique component of $g^{-n}(\AAA)$, denoted by $A^n(K)$, such that $K\subset A^n(K)$. Consequently, $K\subset\bigcap_{n\ge 1}A^n(K)$.

\begin{proposition}\label{2-continuum}
Let $g:\AAA^1\to \AAA$ be an exact annular system. \\
\indent (1) For any component $K$ of $\JJJ_g$, $K$ is a $2$-connected continuum contained essentially in each $A^n(K)$ and $K=\bigcap_{n\ge 1}A^n(K)$. \\
\indent (2) For each component $A$ of $\AAA$ and any point $z\in A$, there exist components $K_1, K_2$ of $\JJJ_g\cap A$ such that the annulus bounded by $K_1$ and $K_2$ contains the point $z$.
\end{proposition}

\begin{proof} (1) For any $n\ge 0$, there is an integer $m>n$ such that $A^m(K)\Subset A^n(K)$ by Proposition \ref{compactly}. Since $A^{n+1}(K)$ is contained essentially in $A^n(K)$ for every $n\ge 0$, $\bigcap_{n\ge 0} A^n(K)$ is a $2$-connected continuum contained essentially in each $A^n(K)$. By definition it is contained in $\JJJ_g$ and hence is equal to $K$.

(2) Let $A^n_1,\,A^n_2\subset A$ be the components of $g^{-n}(\AAA)$ such that they share a common boundary component with $A$. Then $\bigcap_{n\ge 0}(A^n_1\cup A^n_2)=\emptyset$ by Proposition \ref{compactly}. Thus, there exists an integer $m\ge 1$ such that $z\notin (A^m_1\cup A^m_2)$. Notice that there exists a component $K_i$ of $\JJJ_g$ contained essentially in $A^m_i$ ($i=1,2$). Thus, the annulus bounded by $K_1$ and $K_2$ contains the point $z$.
\end{proof}

By Proposition \ref{2-continuum}, for each component $K$ of $\JJJ_g$, $g(K)$ is a component of $\JJJ_g$ and each component of $g^{-1}(K)$ is also a component of $\JJJ_g$. We will say that a component $K$ of $\JJJ_g$ is {\bf periodic}\ if there is an integer $p\ge 1$ such that $g^p(K)=K$; or {\bf pre-periodic}\ if $f^k(K)$ is periodic for some integer $k\ge 1$; or {\bf wandering}\ otherwise. Recall that a {\bf quasicircle} is the image of the unit circle under a quasiconformal map of $\cbar$.

\begin{proposition}\label{qc}
Let $g:\AAA^1\to \AAA$ be an exact annular system. Then any pre-periodic component $K$ of $\JJJ_g$ is a quasicircle.
\end{proposition}

\begin{proof} We only need to consider periodic components of $\JJJ_g$ since each component of their pre-images is also a quasicircle. Let $K$ be a periodic component of $\JJJ_g$ with period $p\ge 1$. Then $g^p(A^{p}(K))=A^0(K)$ and $A^p(K)\Subset A^0(K)$ by Proposition \ref{compactly}. Now applying quasiconformal surgery, we have a quasiconformal map $\phi$ of $\cbar$ such that $\phi\circ g^{p}\circ \phi^{-1}=z^d$ in a neighborhood of $\phi(K)$, where $|d|=\deg (g^{p}|_{A^{p}(K)})\ge 2$. Thus, $K$ is a quasicircle.
\end{proof}

\subsection{Semi-conjugacy to linear systems}

Let $g:\AAA^1\to \AAA$ be an exact annular system. In this sub-section, we want to characterize its dynamics by a linear system as follows: Denote by $A_1, \cdots, A_n$ the components of $\AAA$ and $A^1_1, \cdots, A^1_m$ the components of $\AAA^1$. Let
$$
\III=I_1\cup\cdots\cup I_n\quad\text{ and }\quad\III^1=I_1^1\cup\cdots\cup I_m^1
$$
be disjoint unions of open intervals on $\R^1$ such that \\
\indent (a) $\III^1\subset\III$ and $I_i^1\subset I_j$  whenever $A^1_i\subset A_j$, and \\
\indent (b) for each $I_i$, $\partial I_i\subset\partial (\III^1\cap I_i)$.

Define $\sigma:\III^1\to \III$ by $\sigma(I_i^1)=I_j$ if $g(A_i^1)=A_j$ and $\sigma$ is linear on each $I_i^1$. Set
$$
\III^k=\sigma^{-k}(\III)\text{ for }k>1\text{ and }
\JJJ_{\sigma}=\ds\bigcap_{k\ge 1}\III^k.
$$

\begin{proposition}\label{linear0}
The linear system $\sigma:\III^1\to\III$ is expanding and the closure of $\JJJ_{\sigma}$ in $\R$ is a Cantor set.
\end{proposition}

\begin{proof} To prove that $\sigma$ is expanding, we only need to show that there is an integer $k\ge 1$ such that for any $x\in\III^k$, $|(\sigma^k)'(x)|>1$. For each $k\ge 1$, let $l_k$ and $L_k$ be the minimum and maximum, respectively, of the length of the components of $\III^k$.  Then $L_{k+1}\le L_k$ for any $k\ge 1$. It suffices to show that there is an integer $k\ge 1$ such that $L_k<l_1$.

We will prove that $L_k\to L=0$ as $k\to\infty$. Assume on the contrary that $L=\lim_k L_k>0$. Then for each $k\ge 1$, there is a component of $\III^k$ whose length is at least $L$. Therefore, there exists a sequence $\{I^k\}_{k\ge 1}$ with $I^k$ a component of $\III^k$, such that $I^k\supset I^{k+1}$ and $|I^k|\ge L$.

Denote by $I^{\infty}=\bigcap_k I^k$. Then $|I^{\infty}|\ge L$ and $|I^k|\to |I^{\infty}|$ as $k\to\infty$. In particular, there exists an integer $k_0\ge 1$ such that as $k\ge k_0$,
$$
\frac{|I^k|}{|I^{\infty}|}<\frac{L_1+l_1}{L_1}.
$$

Since $g$ is an annular system, there exists an integer $k_1\ge k_0$ such that $I^{k_1}$ contains another component $I$ of $\III^{k_1+1}$ distinct from $I^{k_1+1}$. Thus,
$$
\frac{|I|}{|I^{k_1+1}|}\le\frac{|I^{k_1}|-|I^{k_1+1}|}{|I^{k_1+1}|}<\frac{l_1}{L_1}.
$$
Since $\sigma^{k_1}$ is linear on $I^{k_1}$, we have
$$
\frac{|\sigma^{k_1}(I)|}{|\sigma^{k_1}(I^{k_1+1})|}<\frac{l_1}{L_1}.
$$
By definition, $|\sigma^{k_1}(I)|\ge l_1$ and $|\sigma^{k_1}(I^{k_1+1})|\le L_1$. So
$$
\frac{|\sigma^{k_1}(I)|}{|\sigma^{k_1}(I^{k_1+1})|}\ge\frac{l_1}{L_1}.
$$
This is a contradiction.

Now each component of $\JJJ_{\sigma}$ is a single point since the linear system $\sigma$ is expanding. It is easy to check that the closure of $\JJJ_{\sigma}$ in $\R$ is
a perfect set and hence a Cantor set.
\end{proof}

For any point $x\in\JJJ_{\sigma}$ and each $k\ge 1$, if we denote by $I^k(x)$ the component of $\III^k$ containing the point $x$, then $\cap_{k\ge 1}I^k(x)=\{x\}$. For any two distinct points $x,y\in\JJJ_{\sigma}$, either they are contained in different component of $\III^1$, or there exists an integer $k_0\ge 2$ such that $I^{k_0}(x)\cap I^{k_0}(y)=\emptyset$ and $I^k(x)=I^k(y)$ for $1\le k<k_0$. In the latter case, $\sigma^{k_0-1}(I^{k_0}(x))$ and $\sigma^{k_0-1}(I^{k_0}(y))$ are different components of $\III^1$. Define the {\bf itinerary} of a point $x\in\JJJ_{\sigma}$ by $i(x)=(j_0,j_1,\cdots)$ if $\sigma^k(x)\in I_{j_k}^1$. Then $i(x)\neq i(y)$ if $x\neq y$.

Define the {\bf itinerary} for each point $z\in\JJJ_g$ by $i_*(z)=(j_0,j_1,\cdots)$ if $g^k(z)\in A_{j_k}^1$. Define a map $\Pi:\, \JJJ_g\to\JJJ_{\sigma}$ by $\Pi(z)=x$ if $i_*(z)=i(x)$. It is well-defined and surjective by Proposition \ref{compactly}.

\begin{proposition}\label{linear}
The map $\Pi:\JJJ_g\to\JJJ_{\sigma}$ is continuous and $\sigma\circ\Pi=\Pi\circ g$ on $\JJJ_g$. For each point $x\in\JJJ_{\sigma}$,
$\Pi^{-1}(x)$ is a component of $\JJJ_g$.
\end{proposition}

\begin{proof} It is easy to check that $\sigma\circ\Pi=\Pi\circ g$ on $\JJJ_g$, and $\Pi^{-1}(x)$ is a component of $\JJJ_g$ for each point $x\in\JJJ_{\sigma}$. Fix any point $x\in\JJJ_{\sigma}$. The collection $\{I^k(x)\cap\JJJ_{\sigma}\}_{k\ge 1}$ forms a basis of neighborhoods of the point $x$ in $\JJJ_{\sigma}$. Now $\Pi^{-1}(\{I^k(x)\cap\JJJ_{\sigma}\})=A^k(\Pi^{-1}(x))\cap\JJJ_g$ is open in $\JJJ_g$ for every $k\ge 1$. So $\Pi$ is continuous.
\end{proof}

Since the set of pre-periodic points is a countable set, we have:

\begin{corollary}\label{uncountable}
There are uncountably many wandering components in $\JJJ_g$.
\end{corollary}

For any point $x\in\JJJ_{\sigma}$, its {\bf $\omega$-limit set} $\omega(x)$ is defined to be the set of points
$y\in\JJJ_{\sigma}$ such that $\sigma^{k_n}(x)$ converges to $y$ as $n\to\infty$ for a subsequence $k_n\to\infty$.

\begin{proposition}\label{infinite}
Let $x\in\JJJ_{\sigma}$ be a wandering point. Then $\omega(x)$ is an infinite set.
\end{proposition}

\begin{proof} Assume that $\omega(x)$ is finite. Define $d(y_1, y_2)$ to be the Euclidean distance if $y_1, y_2$ are contained in the same component of $\III$, or infinity otherwise. Then there exists a constant $\delta>0$ such that for any two distinct points $y_1, y_2\in\III$, $d(y_1, y_2)>\delta$ if either the two points $y_1$ and $y_2$ are contained in different components of $\III^1$ or both of them are contained in $\omega(x)$.

Pick a constant $M\in (1,\infty)$ such that $|\sigma'(y)|<M$ for any point $y\in\III^1$. By the definition of $\omega(x)$, there exists a constant $N\ge 1$ such that for any $n\ge N$, there exists a unique point $y_n\in\omega(x)$ such that $d(\sigma^n(x), y_n)<\delta/(2M)$. Thus, $\sigma^n(x)$ and $y_n$ are contained in the same component of $\III_1$. So $d(\sigma^{n+1}(x), \sigma(y_n))<\delta/2$. It follows that $y_{n+1}=\sigma(y_n)$ for $n\ge N$. This contradicts the fact that $\sigma$ is expanding.
\end{proof}

\subsection{Common boundary}

Recall that each component of the Julia set of an exact annular system is a $2$-connected continuum by Proposition \ref{2-continuum}.

\begin{proposition}\label{common}
Let $g:\AAA^1\to\AAA$ be an exact annular system and $K$ a component of $\JJJ_g$. Let $U$ and $V$ be the two components of $\cbar\smm K$. Then $\partial U=\partial V=K$.
\end{proposition}

\begin{proof}
Assume that each component of $\AAA$ contains at least two components of $\AAA^1$ (otherwise we consider $g^n$ for some $n\ge 2$ by the definition). Then $\|g'\|>1$ under the hyperbolic metric of $\AAA$.

If $K$ is eventually periodic, then $K$ is a quasicircle by Proposition \ref{qc} and hence the theorem holds.
Now we suppose that $K$ is wandering. Let $\Pi$ be a semi-conjugacy from $g:\JJJ_g\to\JJJ_g$ to a linear system $\sigma:\III^1\to\III$ defined in Proposition \ref{linear}. Then $x=\Pi(K)$ is a wandering point. Thus, $\omega(x)$ is an infinite set by Proposition \ref{infinite}. In particular $\omega(x)$ contains a point $y\in\III$ such that $y\notin\partial\III$. It follows that there exists a component $I^m$ of $\sigma^{-m}(\III)$ for which $y\in I^m$ and $I^m$ is contained in the interior of $\III$. Hence there exists an increasing sequence $\{n_k\}_{k\ge 1}$ of positive integers such that $n_k\to\infty$ as $k\to\infty$ and $\sigma^{n_k}(x)\in I^m$.

Denote by $A^m$ the component of $g^{-m}(\AAA)$ corresponding to the interval $I^m$. Then $A^m\Subset\AAA$ and $g^{n_k}(K)\subset A^m$. For any component $J$ of $\JJJ_g$, denote by $A^n(J)$ the component of $g^{-n}(\AAA)$ that contains $J$ . Then we have
$$
g^{n_k}(A^{m+n_k}(K))=A^m(g^{n_k}(K))=A^m.
$$

For each annulus $W\Subset\AAA$, define
$$
\text{width}(W)=\sup_{z\in W}\Big\{d_W(z,\partial_+ W)+d_W(z,\partial_-W)\Big\},
$$
where $\partial_{\pm}W$ denotes the two boundary components of $W$ and $d_W(z, \partial_{\pm}W)$ denotes the infimum of the length of arcs connecting $z$ to $\partial_{\pm}W$ in $W$ under the hyperbolic metric of $\AAA$.

Pick an annulus $W_0$ bounded by smooth curves such that $W_0\Subset\AAA$ and $A^m\subset W_0$. Then $\text{width}(W_0)<\infty$ and there exists a constant $\lambda>1$ such that $\|g'(z)\|\ge\lambda>1$ for every point $z\in g^{-1}(W_0)$.

Denote by $W_k$ the component of $g^{-n_k}(W_0)$ that contains $K$. Then
$$
A^{m+n_k-n_j}(g^{n_j}(K))\subset g^{n_j}(W_k)\Subset A^{n_k-n_j}(g^{n_j}(K))
$$
for $0\le j\le k$ (set $n_0=0$). Note that $A^{n_k-n_j}(g^{n_j}(K))\subset W_0$ if $n_k-n_j\ge m$. Thus, $g^{n_j}(W_k)\subset W_0$ if $k-j\ge m$. Therefore, $\|(g^{n_k})'(z)\|\ge\la^{k-m}$ for any point $z\in W_k$ since the finite orbit $\{z, g(z), \cdots, g^{n_k-1}(z)\}$ passes at least $k-m$ times through the set $g^{-1}(W_0)$ where $\|g'\|\ge\lambda$ and $\|g'\|>1$ always. So
$$
\text{width}(W_k)\le \la^{m-k}\text{width}(W_0).
$$
Hence $\text{width}(W_k)\to 0$ as $k\to\infty$.

Clearly $\partial U\cup\partial V\subset K$. In order to prove that $\partial U=\partial V=K$ we only need to show that $K\subset\partial U$ by symmetry. Otherwise, assume $z\in K\smm\partial U$. Then the spherical distance $d(z,\partial U)>0$. Label the boundary components of $W_k$ by $\partial _\pm W_k$ such that $\partial_+ W_k\subset U$. Then $d(z, \partial_+W_k)>d(z, \partial U)>0$. Note that there exists a constant $M>0$ such that $d_{W_k}(z,\partial_+W_k)\ge M\cdot d(z, \partial_+W_k)$ for all $k\ge 0$. Therefore,
$$
\text{width}(W_k)\ge d_{W_k}(z,\partial_+W_k)\ge M\cdot d_U(z, \partial_{\pm} U)>0.
$$
This contradicts the fact that $\text{width}(W_k)\to 0$ as $k\to\infty$.
\end{proof}

\subsection{Local connectivity}

In the appendix we will give an example constructed by X. Buff showing that an exact annular system may have non-locally connected wandering Julia components. The next theorem gives a sufficient condition about the local connectivity of wandering components. The idea of the proof comes from \cite{PT}.

\begin{theorem}{\label{curve}}
Let $g:\AAA^1\to\AAA$ be an exact annular system. Suppose that $g$ is expanding, i.e. there exists a smooth metric $\rho$ on $\AAA$ and a constant $\lambda>1$ such that $\|g'\|\ge\lambda$. Then every component of $\JJJ_g$ is a Jordan curve.
\end{theorem}

\begin{proof}
Pick a pre-periodic component of $\JJJ_g$ in each component of $\AAA$ and denote by $\G_0$ the collection of them.  It is a multicurve consisting of quasicircles. Denote by $\G_n$ the collection of curves in $g^{-n}(\G_0)$. Then for any curve $\gamma\in\G_n$ and any curve $\beta\in\G_m$ with $m\ne n$, either they are disjoint or they coincide.

For each curve $\be\in\G_1$, there is a unique curve $\g\in\G_0$ such that $\be$ and $\g$ are contained in the same component of $\AAA$. If $\be\neq\g$, there is a homotopy $\Phi_{\be}: S^1\times [0,1]\to\AAA$ from $\g$ to $\be$ such that $\phi_t:=\Phi_{\be}(\cdot, t)$ is a homeomorphism for any $t\in [0,1]$. In particular, $\phi_0(S^1)=\g$,  $\phi_1(S^1)=\be$ and $\phi_t(S^1)$ is a curve between $\be$ and $\g$. If $\be=\g$, define
$\Phi_{\be}(\cdot,t): S^1\to\be$ to be a homeomorphism independent of $t$.

Define the homotopic length of a path $\de: [0,1]\to\AAA$ by
$$
\text{h-length}(\de)=\inf\Big\{\text{length of } \alpha\text{ under metric } \rho\Big\},
$$
where the infimum is taken over all the path $\alpha$ in $\AAA$ connecting $\de(0)$ to $\de(1)$ and homotopic to $\de$. Then
$$
\text{h-length}(\tilde\de)\le\frac{1}{\la}\cdot \text{h-length}(\de)
$$
for any lift $\tilde\de$ of $\de$ under the map $g$ since $\|g'\|\ge\la$.

For each $\be\in\G_1$ and any $s\in S^1$, $\Phi_{\be}(s,\cdot)$ maps the interval $[0,1]$ to a path $\de_{\beta,\, s}$ in the closed annulus $\Phi_{\be}(S^1\times [0,1])$ which connects two points in each of its boundary. So there is a constant $C>0$ such that $\text{h-length}(\de_{\beta,\,s})<C$ for each $\beta\in\G_1$ and any $s\in S^1$.

For each wandering component $K$ of $\JJJ_g$, let $\al_n$ be the unique curve of $\G_n$ with $\al_n\subset A^n(K)$. Then $g^n(\al_n)\in\G_0$ and $\be:=g^n(\al_{n+1})\in\G_1$ are contained in the same component of $\AAA$. Now the homotopy $\Phi_{\be}$ from $g^n(\al_n)$ to $\be$ defined above can be lifted to a homotopy,
denoted by $\Psi_n$,
from $\al_n$ to $\al_{n+1}$,
through the following commutative diagram:
$$
\diagram
S^1\times [0,1] \rto^{\Psi_n} \dto_{P_d} & A^n(K)\dto^{g^n}  \\
S^1\times [0,1] \rto_ {\Phi_{\be}}  &
g^n(A^n(K)).
\enddiagram
$$
Here $d=\deg (g|_{A^n(K)})$ and $P_d(s,t)=(s^d, t)$, i.e. $P(\cdot,t)$ is a covering of $S^1$ with degree $d$. Set
$\psi_{n,t}:=\Psi_n(\cdot, t)$. It is a homeomorphism for any $t\in [0,1]$, and in particular, $\psi_{n,0}(S^1)=\al_n$ and $\psi_{n,1}(S^1)=\al_{n+1}$. For any $s\in S^1$, $\Psi_n(s,t)(S^1)$
is a path connecting a point in $\al_n$ with a point in $\al_{n+1}$ whose homotopic length is less than $C\la^{-n}$.

These isotopies $\Psi_n$ can be pasted together to a continuous map $\Psi: S^1\times [0, \infty)\to\AAA$ as follows:
$$
\Psi(s,t)=\left\{
\begin{array}{ll}
\Psi_0(s,\, t)  & \text{on } S^1\times [0,1]  \\
\Psi_1\Big(\psi_{1,0}^{-1}\circ\psi_{0,1}(s),\, t-1\Big)  & \text{on } S^1\times [1,2] \mystrut\\
\qquad\vdots & \qquad\vdots\mystrut \\
\Psi_n\Big(\psi_{n,0}^{-1}\circ\psi_{n-1,1}\circ\cdots\circ
\psi_{1,0}^{-1}\circ\psi_{0,1}(s),\, t-n\Big)  & \text{on } S^1\times [n,n+1] \\
\qquad\vdots & \qquad\vdots
\end{array}
\right.
$$
Set $h_t=\Psi(\cdot, t)$. Then $h_n(S^1)=\al_n$. For each $s\in S^1$ and any integers $m>n\ge 0$, the homotopic length of the path $\zeta_s(n,m):=\{\Psi(s,t):\, n\le t\le m\}$ satisfies:
$$
\text{h-length }\Big(\zeta_s(n,m)\Big)\le C\la^{-n}+\cdots +C\la^{1-m}\le
\frac C{(\la-1)\la^{n-1}}.
$$
Note that the two endpoints of $\zeta_s(n,m)$ are $h_n(s)\in\al_n$ and $h_m(s)\in\al_m$. The above inequality shows that $\{h_n\}$ is a Cauchy sequence and hence converges uniformly to a continuous map $h$. Since $\al_n\subset A^n(K)$, we have $h(S^1)\subset\bigcap_{n>1}A^n(K)=K$. Note that $h(S^1)$ separates the two components of $\cbar\smm K$. Thus, $h(S^1)=K$ by Proposition \ref{common}. Therefore, $K$ is locally connected and hence is a Jordan curve.
\end{proof}

\section{From multicurves to annular systems}

We prove Theorem \ref{as} in this section. In order to do so, we take a detour into the space of branched coverings of the sphere. We  first modify topologically the rational map $f$ to a branched covering $F$ in its Thurston equivalence class such that $F$ has a topological exact annular system. We then apply a theorem of Rees and Shishikura (cf.\ Theorem \ref{shi} in the appendix) to obtain a semi-conjugacy from $F$  to $f$. Finally we show that the existence of an exact annular system is preserved under the semi-conjugacy due to the following purely topological result.

\begin{lemma}\label{topology}
Let $\G\subset\cbar$ be a finite disjoint union of Jordan curves and $L\subset\cbar$ a compact subset. Then for any Jordan domain $D$ containing $L$, there is an integer $N\ge 0$ such that for any two distinct points $z_1,z_2\in L$, there exists a Jordan arc $\de\subset D$ connecting $z_1$ with $z_2$ and such that $\#(\de\cap \G)\le N$.
\end{lemma}

\begin{proof}
Set
$$
\La=\{\al:\, \al \text{ is a component of }\G\cap D\text{ such that } \al\cap L\ne\emptyset\}.
$$
Then $N:=\#\La<\infty$. In fact, let $\gamma: S^1\times\{1,\cdots,n\}\to\G$ be a homeomorphism. Then $\gamma^{-1}(\G\cap L)$ is a compact subset, which is covered by the open intervals $\{\gamma^{-1}(\al),\, \al\in\La\}$. Thus, $\La$ is a finite set.

For any two distinct points $z_1,z_2\in L$, denote
$$
\La(z_1,z_2)=\{\al\in\La:\, \al\cup\partial D\text{ separates $z_1$ from $z_2$}\}.
$$
Then there exists a Jordan arc $\de\subset D$ connecting $z_1$ with $z_2$ such that $\de$ intersects each $\al\in\La(z_1,z_2)$ in a single point and is disjoint from other components of $\G\cap D$. So $\#(\de\cap\G)\le\#\La(z_1,z_2)\le N$.
\end{proof}

\vskip 0.24cm

\noindent{\it Proof of Theorem \ref{as}}. {\bf Topological modification}. Let $f$ be a post-critically finite rational map with a Cantor multicurve $\G$. There exists a multi-annulus $\CCC\subset\cbar\smm\PPP_f$ homotopic to $\G$ rel $\PPP_f$ such that its boundary $\partial\CCC$ is a disjoint union of Jordan curves in $\cbar\smm\PPP_f$. Let $\CCC^*$ be the union of all of the components of $f^{-1}(\CCC)$ which are homotopic to curves in $\G$ rel $\PPP_f$. Then for each $\g\in\G$, there is at least one component of $\CCC^*$ homotopic to $\g$ rel $\PPP_f$ since $\G$ is pre-stable.

For each $\g\in\G$, denote by $\CCC^*(\g)$ the smallest annulus containing all the components of $\CCC^*$ which are homotopic to $\g$ rel $\PPP_f$. Then its boundary are two Jordan curves in $\cbar\smm\PPP_f$ homotopic to $\g$ rel $\PPP_f$. Set $\CCC^*(\G)=\cup_{\g\in\G}\CCC^*(\g)$. There exist a neighborhood $U$ of $\PPP_f$ and a homeomorphism $\theta_0$ of $\cbar$ such that $\theta_0$ is isotopic to the identity rel $U$ and $\theta_0(\CCC)=\CCC^*(\G)$.

Set $F:=f\circ\theta_0$ and $\CCC^1:=\theta_0^{-1}(\CCC^*)$. Then $\PPP_F=\PPP_f$ and $F$ is Thurston equivalent to $f$ via the pair $(\theta_0, \text{id})$. Moreover, the restriction $F|_{\CCC^1}: \CCC^1\to\CCC$ is a
topological exact annular system.

\vskip 0.24cm

{\bf A semi-conjugacy}. By Theorem \ref{shi}, there exists a neighborhood $V$ of the critical cycles of $F$ and a sequence $\{\phi_n\}$ $(n\ge 1)$ of homeomorphisms of $\cbar$ isotopic to the identity rel $\PPP_F\cup V$ such that $f\circ\phi_{n}=\phi_{n-1}\circ F$, the sequence $\{\phi_n\}$ converges uniformly to a continuous onto map $h$ of
$\cbar$ and $f\circ h=h\circ F$.

\vskip 0.24cm

  {\bf Each fiber does not cross $\CCC$}. Let us define
  $$
  T=\{w\in\cbar:\, h^{-1}(w)\text{ crosses some component of }\CCC\}.
  $$
  Here,
  we say a continuum $E$ {\bf crosses} an annulus $C$ if $E$ intersects both of the two boundary components of $C$. In particular, for each point $w\in T$, $h^{-1}(w)$ is a continuum. Thus $T\subset\JJJ_f$ by Theorem \ref{shi} (3). Moreover, the set $T$ is a closed subset. In fact, let $\{w_n\}$ be a sequence of points in $T$ converging to a point $w_{\infty}\in\cbar$. Passing to a sub-sequence, we may assume that there exist a component $C$ of $\CCC$ and two points $z_n$ and $z'_n$ in $h^{-1}(w_n)$ such that they are contained in the two different boundary components of $C$.  Again we may assume that both the sequences $\{z_n\}$ and $\{z'_n\}$ converge to the points $z_{\infty}$ and $z'_{\infty}$. Then the two points $z_{\infty}$ and $z'_{\infty}$ are also contained in the two different boundary components of $C$.
The continuity of $h$ implies that $h(z_{\infty})=h(z'_{\infty})=w_{\infty}$. Thus, $w_{\infty}\in T$.

  The next claim is crucial. The property of Cantor multicurves is essential here. It is not true for the equator in a mating of polynomials.

\vskip 0.24cm

{\bf Claim 1}. The set $T$ is empty.

\vskip 0.24cm

{\it Proof}. Assume $T\ne\emptyset$ by contradiction. Then $f(T)\subset T$. In fact, suppose $w\in T$, i.e., $h^{-1}(w)$ crosses some component of $\CCC$, then $h^{-1}(w)$ crosses some component $C^1$ of $\CCC^1$. By Theorem \ref{shi} (4), $h^{-1}(f(w))=F(h^{-1}(w))$. So $h^{-1}(f(w))$ crosses $F(C^1)$ which is a component of $\CCC$, so $f(w)\in T$. Set $T_\infty=\bigcap_{n\ge 0}f^n(T)$. Then $T_\infty$ is a non-empty
closed set and $f(T_\infty)=T_\infty$.

Pick one point $w_0\in T_{\infty}$. Since $f(T_\infty)=T_\infty$, there exists a sequence of points $\{w_n\}_{n\ge 0}$ in $T_\infty$ such that $f(w_{n+1})=w_n$ (i.e. $T_{\infty}$ contains a backward orbit). Either $w_n$ is periodic for all $n\ge 0$ or there is an integer $n_0\ge 0$ such that $w_n$ is not periodic for all $n\ge n_0$. In the former case all the points $w_n$ are not critical points of $f$ since $w_n\in\JJJ_f$. In the latter case, there exists an integer $n_1\ge 0$ such that $w_n$ are non-critical points of $f$ for $n\ge n_1$. So in both cases, we have a sequence of points $\{w_n\}_{n\ge 0}$ in $T_\infty\smm\Omega_f$ such that $f(w_{n+1})=w_n$.

Set $L_n=h^{-1}(w_n)$. By Theorem \ref{shi} (3) and (4), $L_n$ is a component of $F^{-1}(L_{n-1})$ and there exists a Jordan domain $D_0\supset L_0$ such that $F^n:\, D_n\to D_{0}$ is a homeomorphism for $n\ge 1$, where $D_n$ is the component of $F^{-n}(D_{0})$ containing $L_n$.

Pick an essential Jordan curve in each component of $\CCC$. They form a Cantor multicurve $\G_0$. By Lemma \ref{topology}, there exists an integer $N\ge 0$ such that for any two distinct points $z_0,z'_0\in L_0$, there is a Jordan arc $\de\subset D_0$ connecting $z_0$ with $z'_0$ and
$$
\#(\de\cap\G_0)\le N.
$$

Since $\G_0$ is a Cantor multicurve, there is an integer $m>0$ such that for each component $C$ of $\CCC$, there are at least $N+1$ components of $F^{-m}(\CCC)$ contained essentially in $C$. Since $L_m$ crosses a component of $\CCC$, there exist two distinct points $z_m, z'_m\in L_m$ such that $F^{-m}(\G_0)$ has at least $N+1$ components separating $z_m$ from $z'_m$.

Now the two points $F^m(z_m)$ and $F^m(z'_m)$ are contained in $L_0$. So there exists a Jordan arc $\de\subset D_0$ connecting $F^m(z_m)$ with $F^m(z'_m)$ such that $\#(\de\cap\G_0)\le N$. Let $\de_m$ be the component of $F^{m}|_{D_m}(\de)$. Then $\delta_m$ connects $z_m$ with $z'_m$. Since $F^m: \de_m\to\de$ is a homeomorphism, we have
$$
\#(\de_m\cap F^{-m}(\G_0))\le N.
$$
This inequality contradicts the fact that $F^{-m}(\G_0)$ has at least $N+1$ components separating $z_m$ from $z'_m$. \qed

\vskip 0.24cm

{\bf Claim 2}. For any $n\ge 0$ and any distinct components $E_1, E_2$ of $\cbar\smm F^{-n}(\CCC)$, $h(E_1)$ is disjoint from $h(E_2)$.

\vskip 0.24cm

{\it Proof}. $E_1$ and $E_2$ are separated by a component $A$ of $F^{-n}(\CCC)$. If $h(E_1)\cap h(E_2)\neq\emptyset$, pick a point $w\in h(E_1)\cap h(E_2)$, then $h^{-1}(w)$ crosses $A$. So $F^n(h^{-1}(w))=h^{-1}(f^n(w))$ crosses $F^n(A)$ by Theorem \ref{shi} (4). This contradicts Claim 1.
\qed

\vskip 0.24cm

{\bf Construction of the multi-annulus $\AAA$}. Denote $\wh E=h^{-1}(h(E))$ for any subset $E\subset\cbar$. Then $\wh E$ is also a continuum if $E$ is a continuum by Theorem \ref{shi} (5) since $h(E)$ is a continuum.

Denote $\EEE=\cbar\smm\CCC$. Then $F^{-1}(\wh\EEE)=\wh{F^{-1}(\EEE)}$ by Theorem \ref{shi} (7). Thus, if $E^1$ is a component of $F^{-1}(\EEE)$, then $\wh{E^1}$ is a component of $F^{-1}(\wh\EEE)$ by Claim 2.

For any two disjoint continua $E_1, E_2\subset\cbar$, we denote by $A(E_1, E_2)$ the unique annular component of $\cbar\smm (E_1\cup E_2)$. For each component $C$ of $\CCC$, there are two distinct components $E_+, E_-$ of $\EEE$ such that $C=A(E_+, E_-)$. Define $\wt C:=A(\wh E_+, \wh E_-)$. It is an annulus contained essentially in $C$ by Claim 2. We claim that the following statements hold:

\vskip 0.24cm

\indent (a) $h^{-1}(h(\wt C))=\wt C$. \\
\indent (b) $\wt C\cap\wh E=\emptyset$ for any subset $E\subset\cbar$ with $E\cap\wt C=\emptyset$. \\
\indent (c) $h(\wt C)$ is an annulus homotopic to $C$ rel $\PPP_f$.

\vskip 0.24cm

{\it Proof}. (a) For any point $z\in\wt C$, if $h^{-1}(h(z))$ is not contained in $\wt C$, then it must intersect $E_+\cup E_-$. So $z\in\wh E_+\cup\wh E_-$. This is a contradiction.

(b) If $z\in\wt C\cap\wh E$, then $h^{-1}(h(z))\subset\wt C$ and hence is disjoint from $E$. It contradicts the condition that $z\in\wh E$.

(c) Let $Q_+, Q_-$ be the two components of $\cbar\smm\wt C$. Then both $\wh Q_+$ and $\wh Q_-$ are disjoint from $\wt C$ by (b). Moreover, they are also disjoint from each other since $h^{-1}(h(z))$ does not cross $C$ for any point $z\in\cbar$ by Claim 1. So $\cbar\smm h(\wt C)$ has exactly two components, $h(Q_+)$ and $h(Q_-)$. Therefore, $h(\wt C)$ is an annulus. Since $h$ is homotopic to the identity rel $\PPP_f$,  the annulus $h(\wt C)$ is
homotopic to $C$ rel $\PPP_f$.  \qed

\vskip 0.24cm

Now let $\wt{\CCC}$ be the union of $\wt C$ for all the components $C$ of $\CCC$. Then $\wt\CCC\subset\CCC$ and it is a multi-annulus homotopic to $\Gamma$ rel $\PPP_f$. Set $\AAA$ to be the union of $h(\wt C)$ for all the components $C$ of $\CCC$. Since $h(\wt C_1)$ is disjoint from $h(\wt C_2)$ for distinct components $C_1, C_2$ of $\CCC$ by (b),
$\AAA$ is a multi-annulus and homotopic to $\Gamma$ rel $\PPP_f$ by (c). Moreover, $\AAA$ is disjoint from a neighborhood of critical cycles of $f$ since $h$ is the identity in a neighborhood of critical cycles of $f$.

\vskip 0.24cm

{\bf Construction of $\AAA^1$}. For each component $C^1$ of $\CCC^1$, there are two distinct components $E^1_+, E^1_-$ of $F^{-1}(\EEE)$ such that $C^1=A(E^1_+, E^1_-)$. Define $\wt{C^1}:=A(\wh{E_+^1}, \wh{E_-^1})$ as above. It
is an annulus contained essentially in $C^1$. Moreover, the following statements hold:

\vskip 0.24cm

\indent $\bullet$ $h^{-1}(h(\wt{C^1}))=\wt{C^1}$. \\
\indent $\bullet$ $\wt{C^1}\cap\wh E=\emptyset$ for any subset $E\subset\cbar$ with $E\cap\wt{C^1}=\emptyset$. \\
\indent $\bullet$ $h(\wt{C^1})$ is an annulus homotopic to $C^1$ rel $\PPP_f$.

\vskip 0.24cm

Set $\wt{\CCC^1}$ to be the union of $\wt{C^1}$ for all the components $C^1$ of $\CCC^1$. Set $\AAA^1$ to be the union of $h(\wt{C^1})$ for all the components $C^1$ of $\CCC^1$. Then it is a multi-annulus contained essentially in $\AAA$.

\vskip 0.24cm

{\bf Invariance of $\AAA$}. Note that each component of $\wt\CCC$ is a component of $\cbar\smm\wh\EEE$ and each component of $\wt{\CCC^1}$ is a
component of $\cbar\smm F^{-1}(\wh\EEE)=\cbar\smm\wh{F^{-1}(\EEE)}$. So $F:\wt{\CCC^1}\to\wt\CCC$ is proper. Since $\wt\CCC=h^{-1}(\AAA)$ and $\wt{\CCC^1}=h^{-1}(\AAA^1)$, the map $f:\, \AAA^1\to\AAA$ is also proper.

For any component $E$ of $\EEE$, there is a unique component $E^1$ of $F^{-1}(\EEE)$ such that $\partial E\subset\partial E^1$. Moreover, $E^1\subset E$ and $E\smm E^1$ is a disjoint union of Jordan domains in $E$. We claim that $\wh E\smm E=\wh{E^1}\smm E$.

Since $E\supset E^1$, we have $\wh E\supset\wh{E^1}$. On the other hand, any component $D$ of $\cbar\smm E$ is a Jordan domain. Assume $z\in\wh E\cap D$, then $h^{-1}(h(z))$ is a full continuum intersecting $\partial E$ by Theorem \ref{shi} (3). Thus, $h^{-1}(h(z))$ intersects $\partial E^1$. Therefore, $z\in\wh{E^1}$ and hence $\wh E\smm E\subset\wh{E^1}\smm E$. The claim is proved.

From the claim, we see that
$\wt{\CCC^1}\subset\wt\CCC$ and each component of $\partial\wt\CCC$ is a component of $\partial\wt{\CCC^1}$. Hence $\AAA^1\subset\AAA$ and
for any component $A$ of $\AAA$,
each component of $\partial A$
	is a component of $\partial A^1$ for some component $A^1$ of $\AAA^1$ in $A$. So $f:\AAA^1\to\AAA$ is an exact annular system.

\vskip 0.24cm

{\bf Uniqueness of $\AAA$}. Suppose that $f: \AAA^1_1\to\AAA_1$ is another exact annular system such that $\AAA_1$ is homotopic to $\G$ rel $\PPP_f$. Pick an essential Jordan curve in every components of $\AAA$ and $\AAA_1$, respectively. They form two multicurves $\G_0\subset\AAA$ and $\G_1\subset\AAA_1$. Both of them are homotopic to $\Gamma$ rel $\PPP_f$. So there exist a neighborhood $U$ of the critical cycles of $f$ and a homeomorphism $\theta_0$ of $\cbar$ such that $\theta_0(\G_0)=\G_1$ and $\theta_0$ is isotopic to the identity rel $\PPP_f\cup U$. By Theorem \ref{shi}, there exist a neighborhood $V$ of the critical cycles of $f$ and a sequence $\{\theta_n\}$ ($n\ge 1$) of homeomorphisms of $\cbar$ isotopic to the identity rel $\PPP_f\cup V$, such that $f\circ\theta_n=\theta_{n-1}\circ f$. Moreover, $\{\theta_n\}$ converges uniformly to a continuous map $h$ of $\cbar$.

It is easy to see that $h$ is the identity in the Fatou set of $f$. On the other hand, $h$ is also the identity on the Julia set $\JJJ_f$ since the closure of $\cup_{n\ge 0}f^{-n}(\PPP_f)$ contains $\JJJ_f$ and $\theta_n$ is the identity on $f^{-n}(\PPP_f)$. So $\{\theta_n\}$ converges uniformly to the identity.

For each component $A$ of $\AAA$, set $A(n,\G_0)$ to be the closed annulus bounded by two curves in $A\cap (f|_{\AAA^1})^{-n}(\G_0)$ such that $A\cap (f|_{\AAA^1})^{-n}(\G_0)\subset A(n,\G_0)$. Then $\theta_n(A(n,\G_0))\subset\AAA_1$ since $\theta_n(f^{-n}(\G_0))=f^{-n}(\G_1)$. By Proposition \ref{compactly}, for any compact set $G\subset A$, $G\subset A(n,\G_0)$ when $n$ is large enough. So $A\subset\AAA_1$ since $\{\theta_n\}$ converges uniformly to the identity. It follows that $\AAA\subset\AAA_1$. By symmetry, we have $\AAA=\AAA_1$.

\vskip 0.24cm

{\bf Properties of $\JJJ_g$}. We want to prove that $\JJJ_g\subset\JJJ_f$. Assume by contradiction that there is a point $z\in \JJJ_g\smm\JJJ_f$. Then $\{f^n(z)\}_{n\ge 0}$ converges to a super-attracting cycle of $f$ as $n\to\infty$. But $f^n(z)\in g^n(\JJJ_g)\subset\JJJ_g$. Thus, $\overline{\AAA}$ contains a critical cycle. This is a contradiction since $\AAA$ is disjoint from a neighborhood of critical cycles.

There is a singular conformal metric $\rho$ on $\cbar$ where the singularities may occur at $\PPP_f$ such that $f$ is strictly expanding on $(\cbar,\rho)$ except in a neighborhood of super-attracting cycles (e.g., the hyperbolic metric on the orbifold of $f$, cf.\  \cite{DH1, TY, Th}). Applying Theorem \ref{curve}, we see that each component of $\JJJ_g$ is a Jordan curve. In particular, $\JJJ_g$ has uncountably  many components. All but countably many of them are wandering by Corollary \ref{uncountable}.
\qed

\section{Decompositions and renormalizations}

We will prove Theorem \ref{decomposition} in this section. First, let us recall the definition of rational-like maps and prove a straightening theorem.

\subsection{Rational-like maps}

\begin{definition} Let $U\Subset V$ be two finitely-connected domains in $\cbar$. A  map $g:U\to V$  is called a {\bf
rational-like map}\ if \\
\indent (1) $g$ is holomorphic and proper with $\deg g\ge 2$, \\
\indent (2) the orbit of every critical point of $g$ (if any) stays in $U$, and \\
\indent (3) each component of $\cbar\smm U$ contains at most one  component of $\cbar\smm V$.

The {\bf filled Julia set}\ of $g$ is defined by
$$
\KKK_g=\bigcap_{n>0}g^{-n}(V).
$$

We say that a rational-like map $g:U\to V$ is  a {\bf renormalization}\ of a rational map $f$ if $g=f^p|_U$ for some integer $p\ge 1$ and $\deg g<\deg f^{p}$.
\end{definition}

{\bf Remark.} A rational-like map here is actually a repelling system of constant complexity defined in \cite{CT}.

\begin{proposition}\label{connected}
Let $g: U_1\to U_0$ be a rational-like map. Then $g^{-n}(U_0)$ is connected for any $n\ge 1$. The filled Julia set $\KKK_g$ is a continuum.
\end{proposition}

\begin{proof} Pick a domain $V_0\Subset U_0$ such that every component of $\cbar\smm V_0$ contains exactly one component of $\cbar\smm U_0$, $U_1\Subset V_0$ and every component of $\partial V_0$ is a Jordan curve. Set
$V_1:=g^{-1}(V_0)$. Then $V_1\Subset V_0$, every component of $\cbar\smm V_1$ contains at most one component of $\cbar\smm V_0$ and each component of $\partial V_1$ is a Jordan curve.

Since every component of $\cbar\smm V_1$ contains at most one component of $\cbar\smm V_0$, each component $W$ of $V_0\smm\overline{V_1}$ is either a disk or an annulus. In the latter case, one of the component of the boundary $\partial W$ is a component of the boundary $\partial V_0$ and the other is a component of the boundary $\partial V_1$.

Denote by $V_n=g^{-n}(V_0)$ for $n>1$. Then $V_{n+1}\Subset V_n$ for $n\ge 1$. Since all the critical orbits of $g$ stay in $U_1$ and thus in $\KKK_g$, each component $W$ of $V_1\smm\overline{V_2}$ is also either a disk or an annulus. In the latter case, one of the component of $\partial W$ is a component of $\partial V_1$ and the other is a component of $\partial V_2$. Therefore, $V_2$ is also connected and every component of $\cbar\smm V_2$ contains at most one component of $\cbar\smm V_1$. Inductively, we have that $V_{n+1}$ is connected and every component of $\cbar\smm V_{n+1}$ contains at most one component of $\cbar\smm V_n$. It follows that $\KKK_g$ is a connected compact set.
\end{proof}

Similarly to Douady-Hubbard's polynomial-like map theory \cite{DH2}, we get a straightening theorem for rational-like maps with a slightly different proof.

\begin{theorem}\label{st}
Let $g: U\to V$ be a rational-like map. Then there is a rational map $f$ with $\deg f=\deg g$ and a quasiconformal map $\phi$ of $\cbar$ such that: \\
\indent (a) $f\circ\phi=\phi\circ g$ in a neighborhood of $\KKK_g$, \\
\indent (b) the complex dilatation $\mu_{\phi}$ of $\phi$ satisfies $\mu_{\phi}(z)=0$ for a.e. $z\in\KKK_g$, \\
\indent (c) $\JJJ_f=\partial\phi(\KKK_g)$, and \\
\indent (d) each component of $\cbar\smm\phi(\KKK_g)$ contains at most one point of $\PPP_f$. \\
\indent Moreover, the rational map $f$ is unique up to holomorphic conjugation.
\end{theorem}

\begin{proof} Pick a domain $V_1\Subset V$ such that every component of $\cbar\smm V_1$ contains exactly one component of $\cbar\smm V$, $U\Subset V_1$ and every component of $\partial V_1$ is a quasicircle. Then
$U_1:=g^{-1}(V_1)\Subset V_1$, every component of $\cbar\smm U_1$ contains at most one component of $\cbar\smm V_1$ and each component of $\partial U_1$ is a quasicircle.

Let $E_1,\cdots, E_m$ be the components of $\EEE:=\cbar\smm\overline{V_1}$. Let $B_1,\cdots, B_n$ be the components of $\BBB:=\cbar\smm\overline{U_1}$ such that $B_i\supset E_i$ for $1\le i\le m$. Then $\EEE\Subset\BBB$. Define a map
$\tau$ on the index set by $\tau(i)=j$ if $g(\partial B_i)=\partial E_j$.

Let $D_i\subset\C$ ($i=1,\cdots,n$) be round disks centered at $a_i$ with unit radius such that their closures are pairwise disjoint. Denote their union by $\DDD$. Define a map $Q$ on $\DDD$ by
$$
Q(z)=r(z-a_i)^{d_i}+a_{\tau(i)},\ z\in D_i,
$$
where $0<r<1$ is a constant and $d_i=\deg (g|_{\partial E_i})$. Then $Q(D_i)\Subset D_{\tau(i)}$. Denote by
$D_{\tau(i)}(r):=Q(D_i)$ and $\DDD(r)=Q(\DDD)$.

Let $\psi:\EEE\to\DDD(r)$ be a conformal map such that $\psi(E_i)=D_i(r)$. It can be extended to a quasiconformal map in a neighborhood of $\overline{\EEE}$ since the components of $\EEE$ are quasidisks with pairwise disjoint closures. Since $Q: \partial\DDD\to\partial\DDD(r)$ and $g:\, \partial\BBB\to\partial\EEE$ are coverings with same degrees on
corresponding components, there is a homeomorphism $\psi_1: \partial\BBB\to\partial\DDD$ such that $\psi\circ g=Q\circ\psi_1$.

Since each component of $\partial\BBB$ is a quasicircle, the conformal map $\psi:\EEE\to\DDD(r)$ can be extended to a
homeomorphism $\psi:\overline{\BBB}\to\overline{\DDD}$ such that $\psi|_{\partial\BBB}=\psi_1$ and $\psi$ is quasiconformal on $\BBB$. Define a map
$$
G=\left\{
\begin{array}{ll}
g & \text{on } U_1,\\
\psi^{-1}\circ Q\circ\psi & \text{on }\overline{\BBB}.
\end{array}\right.
$$
Then $G$ is a quasiregular branched covering of $\cbar$. Set $\OOO:=\psi^{-1}(\{a_1,\cdots,a_n\})$. Then
$$
G(\OOO)\subset\OOO\ \text{ and }\ \PPP_G\smm\KKK_g\subset\OOO
$$
since no critical point of $g$ escapes. Moreover, for each point $z_0\in\cbar\smm\KKK_g$, there exists a neighborhood $N$ of the point $z_0$ such that the forward orbit $\{G^n(z)\}$ converges to the invariant set $\OOO$ uniformly for $z\in N$.

By the Measurable Riemann Mapping Theorem, there is a quasiconformal map $\Phi$ of $\cbar$ whose complex dilatation satisfies $\mu_{\Phi}=0$ on $U_1$ and $\mu_{\Phi}=\mu_{\psi}$ on $\BBB$. Set $F:=\Phi\circ G\circ\Phi^{-1}$. Then $F$ is holomorphic in the interior of $\Phi(g^{-1}(U_1)\cup\EEE)$.

For any orbit $\{F^n(z)\}_{n\ge 0}$, if $z$ is not contained in the interior of $\Phi(g^{-1}(U_1)\cup\EEE)$, then $z$ is contained in either $\Phi(\overline{U_1})\smm\Phi(g^{-1}(U_1))$ or the closure of $\Phi(\BBB)\smm\Phi(\EEE)$. In the latter case, $F(z)\in\Phi(\EEE)$ and thus $F^2(z)$ is contained in the interior of $\Phi(\EEE)$. In the former case, $F(z)\in \Phi(\BBB)\smm\Phi(\EEE)$ and thus $F^3(z)$ is contained in the interior of $\Phi(\EEE)$. Thus, $F^n(z)$ is contained in the interior of $\Phi(\EEE)$ for $n\ge 3$ in both cases. This shows that every orbit of $F$ passes through the closure of $\cbar\smm \Phi(g^{-1}(U_1)\cup\EEE)$ at most three times. Applying Shishikura's Surgery Principle (see Lemma 15 in \cite{Ah}), there is quasiconformal map $\Phi_1:\cbar\to\cbar$ such that $f=\Phi_1\circ F\circ\Phi_1^{-1}$ is a rational map. Moreover,
$\mu_{\Phi_1}(z)=0$ for $a.e.\ z\in \Phi(\KKK_g)$. Set $\phi=\Phi_1\circ\Phi$. Then $f\circ\phi=\phi\circ g$ on $U_1$ and $\mu_{\phi}(z)=0$ for $a.e.\ z\in\KKK_g$.

For each point $w_0\in\cbar\smm\phi(\KKK_g)$, there exists a neighborhood $N$ of the point $w_0$ such that the forward orbit $\{f^n(w)\}$ converges to the invariant set $\phi(\OOO)$ uniformly for $w\in N$. Moreover, $\PPP_f\smm\phi(\KKK_g)\subset\phi(\OOO)$. So $\cbar\smm\phi(\KKK_g)\subset\FFF_f$. Since $\phi(\KKK_g)$ is completely invariant under $f$, we have $\partial\phi(\KKK_g)=\JJJ_f$.

If there is another rational map $f_1$ satisfying the conditions of the theorem, then there is a quasiconformal map $\theta$ of $\cbar$ such that $f_1\circ\theta=\theta\circ f$ in a neighborhood of
$\phi(\KKK_g)$ and $\mu_{\theta}(z)=0$ for $a.e.\ z\in\phi(\KKK_g)$.

Let $W$ be a periodic Fatou domain of $f$ in $\cbar\smm\phi(\KKK_g)$ with period $p\ge 1$. Then $W$ is simply connected and contains exactly one point $z_0\in\PPP_f$, which is a super-attracting periodic point. Therefore, there is a conformal map $\eta$ from $W$ onto the unit disc $\D$ such that $\eta(z_0)=0$ and $\eta\circ
f^p\circ\eta^{-1}(z)=z^d$ with $d=\deg_{z_0}f^p>1$. On the other hand, let $z_1\in\theta(W)$ be the super-attracting periodic point of $f_1$, then there is a conformal map $\eta_1: \theta(W)\to\D$ such that $\eta_1(z_1)=0$ and $\eta_1\circ f_1^p\circ\eta_1^{-1}(z)=z^d$. Therefore,
$$
\eta_1\circ\theta\circ f^p\circ\theta^{-1}\circ\eta_1^{-1}(z)=z^d
$$
in a neighborhood of $\partial\D$ in $\D$. This shows that $T=\eta_1\circ\theta\circ\eta^{-1}$ is a rotation on $\partial\D$ (see the commutative diagram below).
$$
\diagram \D\dto_{z\mapsto z^d} & W \lto_{\eta}\dto_{f^p} \rto^{\theta} & \theta(W)\dto^{f_1^p} \rto ^{\eta_1} & \D\dto^{z\mapsto z^d} \\ \D
&W \lto^{\eta} \rto_{\theta} & \theta(W)\rto_{\eta_1} & \D
 \enddiagram
$$
Let $\theta_W=\eta_1^{-1}\circ T\circ\eta$. Then $\theta_W:W\to\theta(W)$ is holomorphic, $\theta_W=\theta$ on the
boundary $\partial W$ and $f_1\circ\theta_W=\theta_W\circ f$.

Define $\Theta_0:\cbar\to\cbar$ by $\Theta_0=\theta_W$ on all the super-attracting Fatou domains of $f$ in $\cbar\smm\phi(\KKK_g)$, and $\Theta_0=\theta$ otherwise. Then $\Theta_0$ is a quasiconformal map and $\Theta_0\circ f=f_1\circ\Theta_0$ on the union of $\phi(\KKK_g)$ and all the super-attracting Fatou domains of $f$ in $\cbar\smm\phi(\KKK_g)$. Pulling back $\Theta_0$, we get a sequence of quasiconformal maps $\Theta_n:\cbar\to\cbar$ such that $\Theta_{n-1}\circ f=f_1\circ\Theta_n$. It is easy to check that $\Theta_n$ converges uniformly to a holomorphic conjugacy from $f$ to $f_1$.
\end{proof}

\subsection{Renormalizations}

Let $f$ be a post-critically rational map with a stable Cantor multicurve $\Gamma$. By Theorem \ref{as}, there is a unique multi-annulus $\AAA\subset\cbar\smm\PPP_f$ homotopic rel $\PPP_f$ to $\Gamma$ such that $g=f|_{\AAA^1}: \AAA^1\to \AAA$ is an exact annular system, where $\AAA^1$ is the union of the components of $f^{-1}(\AAA)$ homotopic rel $\PPP_f$ to curves in $\Gamma$. Moreover, $\JJJ_g\subset\JJJ_f$ and each component of $\JJJ_g$ is a Jordan curve.
Denote
$$
\JJJ(\Gamma)=\bigcup_{n\ge 1}f^{-n}(\JJJ_g).
$$
Since $g^{-1}(\JJJ_g)=\JJJ_g$, we have $\JJJ_g\subset f^{-1}(\JJJ_g)$, each component of $\JJJ_g$ is also a component of $f^{-1}(\JJJ_g)$ and each component of $f^{-1}(\JJJ_g)$ is a Jordan curve. Consequently, $f^{-1}(\JJJ(\Gamma))=\JJJ(\Gamma)$ and each component of $\JJJ(\Gamma)$ is a Jordan curve. By the definition of $\JJJ_g$, we have:
$$
\JJJ(\Gamma)=\bigcup_{n\ge 1}\bigcap_{m\ge n}f^{-m}(\AAA).
$$
Denote $\KKK(\Gamma)=\cbar\smm\JJJ(\Gamma)$. It is completely invariant and
$$
\KKK(\Gamma)=\bigcap_{n\ge 1}\bigcup_{m\ge n}(\cbar\smm f^{-m}(\AAA)).
$$
Since $\partial\AAA\subset\partial f^{-1}(\AAA)$, we have
$$
\partial f^{-n+1}(\AAA)\subset\partial f^{-n}(\AAA)\subset\KKK(\Gamma)\text{ for }n\ge 1.
$$

Recall that a connected subset $E\subset\cbar$ is of simple type (w.r.t. $\PPP_f$) if either there exists a simply connected domain $U\subset\cbar$ such that $E\subset U$ and $U$ contains at most one point in $\PPP_f$, or there exists an annulus $A\subset\cbar\smm\PPP_f$ such that $E\subset A$. It is of complex type (w.r.t. $\PPP_f$) otherwise. Since $f(\PPP_f)\subset\PPP_f$, for each simple type continuum $E\subset\cbar$, each component of $f^{-1}(E)$ is also simple type.

Set $\BBB^0=\cbar\smm\AAA$. It has $\#\Gamma+1$ components and each of them is of complex type. For each component $B$ of $\BBB^0$ and any component $A$ of $f^{-n}(\AAA)$ with $n\ge 1$,  either $A\cap B=\emptyset$ or $A\subset B$ since $\partial\AAA_{n}\subset\partial\AAA_{n+1}$. In the latter case, the essential Jordan curves in $A$ is either null-homotopic or peripheral since $\Gamma$ is stable. Thus, for each $n\ge 1$, $B\smm f^{-n}(\AAA)$ has exactly one complex type component.

Denote by $\BBB^n$ the union of complex type components of $\cbar\smm f^{-n}(\AAA)$ for $n\ge 1$. Then $\BBB^n\subset\BBB^{n-1}$ and it also has exactly  $\#\Gamma+1$ components since each component of $\AAA\smm f^{-n}(\AAA)$ is of simple type. Obviously, each component of $\BBB^n$ is also a component of $f^{-1}(\BBB^{n-1})$.  Denote by
$$
\KKK_c=\bigcap_{n\ge 0}\BBB^n.
$$
Then $\KKK_c$ is compact and has exactly $\#\Gamma+1$ components which are of complex type. Each of its components is also a component of $f^{-1}(\KKK_c)$.
It is easy to verify that $\KKK_c\subset\KKK(\Gamma)$ and each component of $\KKK_c$ is also a component of $\KKK(\Gamma)$.

Each component of $\cbar\smm\KKK_c$ is either a component of $\AAA$ or a simply connected domain containing at most one point of $\PPP_f$. Therefore, each component of $\KKK(\Gamma)\smm\KKK_c$ is of simple type. In summary we have:

\begin{proposition}\label{decomposition1}
The compact set $\KKK_c$ is the union of complex type components of $\KKK(\Gamma)$. It has exactly $\#\Gamma+1$ components and each of them is also a component of $f^{-1}(\KKK_c)$. Consequently, each component of $\KKK_c$ is eventually periodic.
\end{proposition}

\begin{theorem}\label{renorm}
Let $K$ be a periodic component of $\KKK_c$ with period $p\ge 1$. Then there exist domains $U\Subset V$ in $\cbar$ such that $K\subset U$ and $g=f^p|_U: U\to V$ is a renormalization of $f$ with filled Julia set $\KKK_g=K$.
\end{theorem}

\begin{proof}
Let $B_0, \cdots, B_{p-1}$ be the components of $\cbar\smm\AAA$ such that $K\subset B_0$ and $f^i(K)\subset B_i$ for $0<i<p$. Let $A_1, \cdots, A_n$ be the components of $\AAA$ whose boundary intersects $B_0$. Set
$$
W'=B_0\cup\bigcup_{i=1}^n A_i.
$$
It is a finitely-connected domain. Let $W'_1$ be the component of $f^{-p}(W')$ containing $K$. Then $W'_1\subset W'$ and each component $D$ of $W'\smm\overline{W'_1}$ is either an annulus disjoint from $\PPP_f$ or a simply-connected domain. For each $1\le i\le n$, an essential curve in each component of $f^{-p}(A_i)$ is either non-essential or homotopic to a curve in $\G$ since $\Gamma$ is stable. Thus in the latter case the simply connected domain $D$ contains at most one point of $\PPP_f$.

Each $A_i$ contains exactly one component of $W'_1\smm K$, denoted by $A_i^p$, which is a component of $f^{-p}(\AAA)$ and shares a common boundary component with $A_i$. Relabeling the indices of the $A_i$'s if necessary we may assume that there exists an integer $q\ge 1$ such that
$$
f^p(A_1^p)=A_2,\, \cdots, f^p(A_{q-1}^p)=A_q \text{ and } f^p(A_q^p)=A_1,\,
$$
i.e. $A_1$ is $q$-periodic. Then there is an integer $1\le i\le q$ such that $A^p_i\subsetneq A_i$. Otherwise $A^p_i=A_i$ for $1\le i\le q$ and hence $f^{qp}(A_1)=A_1$. It contradicts the fact that $f: \AAA^1\to\AAA$ is an annular system.

We claim that there exists a Jordan curve $\g_i$ contained essentially in each $A_i$ for $1\le i\le q$ such that $\g'_i$ separates $\g_i$ from $K$, where $\g'_i$ is the component of $f^{-p}(\g_{i+1})$ contained $A^p_i$ for $1\le i<q$ and $\g'_q$ is the component of $f^{-p}(\g_{1})$ contained $A^p_q$.

Assume that $q=1$. Then $A^p_1\subsetneq A_1$. Let $\g_1$ be a Jordan curve contained essentially in $A_1$ such that it is disjoint from $A^p_1$. Let $\g'_1$ be the component of $f^{-p}(\g_1)$ in $A^p_1$. Then $\g'_1$ is a Jordan curve contained essentially in $A^p_1$ and hence separates $\g_1$ from $K$.

Now we assume that $q\ge 2$ and $A^p_1\subsetneq A_1$. Then there exists a Jordan curve $\g_1$ contained essentially in $A_1$ such that it is disjoint from $A^p_1$. Let $\g'_q$ be the component of $f^{-p}(\g_1)$ in $A^p_q$. Then one may find a Jordan curve $\g_q$ contained essentially in $A_q$ such that $\g'_q$ separates $\g_q$ from $K$. Let $\g'_{q-1}$ be the component of $f^{-p}(\g_q)$ in $A^p_{q-1}$. Then one may also find a Jordan curve $\g_{q-1}$ contained essentially in $A_{q-1}$ such that $\g'_{q-1}$ separates $\g_{q-1}$ from $K$. Inductively, for $2\le i<q$, let $\g'_i$ be the component of $f^{-p}(\g_{i+1})$ in $A^p_i$, one may find a Jordan curve $\g_i$ contained essentially in $A_i$ such that $\g'_i$ separates $\g_i$ from $K$. Let $\g'_1$ be the component of $f^{-p}(\g_2)$ contained in $A^p_1$. Then $\g'_1$ separates $\g_1$ from $K$ since $\g_1$ is disjoint from $A^p_1$. Now the claim is proved.

Do this process for each cycle. We obtain a Jordan curve $\g_i$ contained essentially in each periodic annulus $A_i$ such that if $f^p(A^p_i)=A_j$, then the component of $f^{-p}(\g_j)$ in $A_i$ separates $\g_i$ from $K$.  If $A_i$ is not periodic but $A_j=f^p(A^p_i)$ is periodic, then there is always a Jordan curve $\g_i$ contained essentially in $A_i$ such that the component of $f^{-p}(\g_j)$ in $A_i$ separates $\g_i$ from $K$.

In summary, we have a Jordan curve $\g_i\subset A_i$ for each $A_i$ such that if $f^p(A_i^p)=A_j$, then the component of $f^{-p}(\g_j)$ in $A^p_i$ separates $\g_i$ from $K$. Let $W\subset W'$ be the domain bounded by the curves $\g_i$ defined above. Then $W_1\Subset W$, where $W_1$ is the component of $f^{-p}(W)$ containing $K$, and each component of $W\smm\overline{W_1}$ is either an annulus disjoint from $\PPP_f$ or a disk containing at most one point of $\PPP_f$.

Let $W_n$ be the component of $f^{-np}(W)$ containing $K$ for $n\ge 2$. Then $W_n\Subset W_{n-1}$ and each component of $W\smm\overline{W_1}$ is either an annulus disjoint from $\PPP_f$, or a disk which contains at most one point of $\PPP_f$.

Since $\PPP_f$ is finite, there is an integer $N\ge 1$ such that $W_n\cap\PPP_f=W_N\cap \PPP_f$ for $n\ge N$. Set $U=W_{N+1}$, $V=W_N$ and $g:=f^p|_U: U\to V$. Then every critical point of $g$ stays in $U$. By Proposition \ref{deg}, there is an integer $n\ge 1$ such that $\deg (f^n|_A)\ge 2$ for all the components $A$ of $\AAA^n$. So we have $\deg g\ge 2$. Therefore, $g: U\to V$ is a rational-like map.

Now we want to show that $\deg g<\deg f^p$. Otherwise $\JJJ_f\subset\KKK_g$. But we know that the Julia set of the annular system $f:\, \AAA^1\to\AAA$ is contained in $\JJJ_f$. This is impossible. So $\deg g<\deg f^p$.
It follows that $g$ is a renormalization of $f$.
\end{proof}

From Theorem \ref{st}, Theorem \ref{renorm} and Theorem 2.1 in \cite{TY}, we have:

\begin{corollary}\label{local}
Let $K$ be a component of $\KKK_c$. For each component $W$ of $\cbar\smm K$, its boundary $\partial W$ is locally connected.
\end{corollary}

\subsection{Topology of $\KKK(\Gamma)$}

We prove Theorem \ref{decomposition} in this sub-section. Assume first that $\KKK_c\cup\AAA=\cbar$. This may happen for the folding of polynomials defined in \S8. Then for each component $K$ of $\KKK(\Gamma)$, $K$ is a continuum which eventually maps to a component of $\KKK_c$ under $f$. We have concluded the proof in this case.

Now, assume that $\KKK_c\cup\AAA\neq\cbar$. Then $\cbar\smm\KKK_c$ has infinitely many components. We will construct a puzzle by choosing some components of $f^{-n}(\KKK_c)$ for some integer $n\ge 1$.

Each component of $f^{-n}(\KKK_c)$ for $n\ge 1$ is also a component of $\KKK(\Gamma)$. It is either a component of $\KKK_c$ or a simple type continuum which could be: \\
\indent (a) ({\it disk-type}) a compact set containing exactly one point of $\PPP_f$; or \\
\indent (b) ({\it annular-type}) a compact set disjoint from $\PPP_f$ but having  exactly two complementary components containing points of $\PPP_f$; or \\
\indent (c) ({\it trivial-type}) a compact set disjoint from $\PPP_f$
 and having exactly one complementary components containing points of $\PPP_f$.

Obviously, each component of the pre-image of a trivial-type component is also trivial-type.

\vskip 0.24cm

There exists an integer $n_0\ge 1$ such that each component of $\AAA$ contains essentially an annular-type component of $f^{-n_0}(\KKK_c)$.

For each point $x\in\PPP_f\smm\KKK_c$ (if it exists) either there exists an integer $n_1(x)\ge 1$ such that $x\in f^{-n_1(x)}(\KKK_c)$, or $x\notin\cup_{n\ge 1}f^{-n}(\KKK_c)$. In both cases, there exists an integer $n_2(x)\ge 1$ such that $f^{-n_2(x)}(\KKK_c)$ has an annular-type component separating the point $x$ from other points of $\PPP_f$.  Since $\PPP_f$ is finite, there exists an integer $s\ge 1$ such that
$$
s\ge\left\{\begin{array}{ll}
n_0, &  \\
n_1(x) & \text{ for }x\in (\cup_{n\ge 1}f^{-n}(\KKK_c))\smm\KKK_c, \\
n_2(x) & \text{ for }x\in\PPP_f\smm\KKK_c.
\end{array}\right.
$$

Denote by $\KKK_{e}$ the union of components of $f^{-s}(\KKK_c)$ that are not trivial-type. Then the following statements hold:

(1) Each component of $\AAA$ contains essentially a component of $\KKK_e$.

(2) For each point $x\in\PPP_f\cap(\cup_{n\ge 1}f^{-n}(\KKK_c))$, we have $x\in\KKK_e$.

(3) For each point $x\in\PPP_f\smm(\cup_{n\ge 1}f^{-n}(\KKK_c))$, there exists an annular-type component $K$ of $\KKK_e$ such that $K$ separates the point $x$ from other points of $\PPP_f$.

(4) $f(\KKK_e)\subset\KKK_e$.

\begin{figure}[htbp]
\begin{center}
\includegraphics[width=9cm]{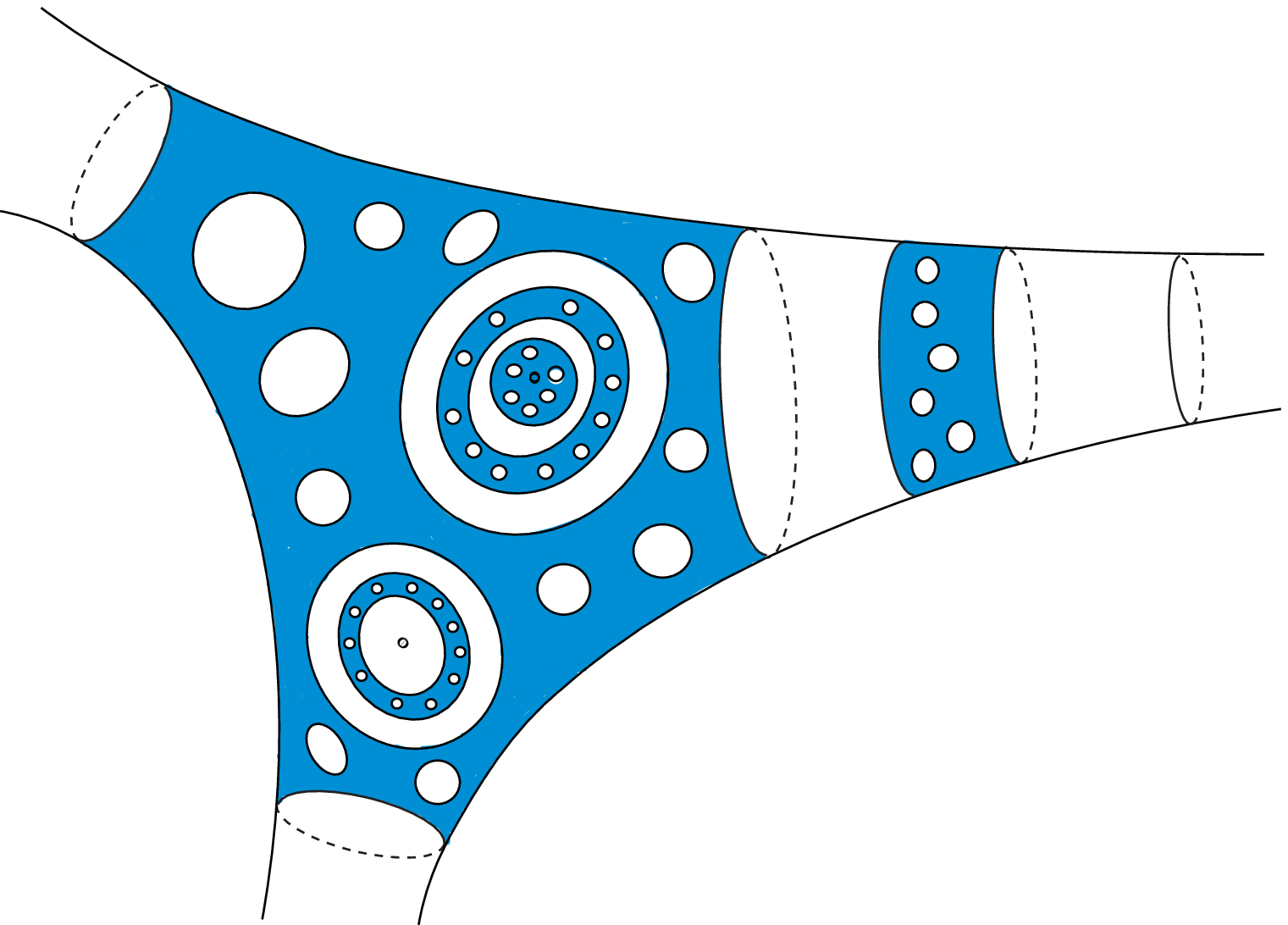}
\end{center}
\begin{center}{\sf Figure 1. The components of $\KKK_e$.}
\end{center}
\end{figure}

Denote $\UUU=\cbar\smm\KKK_e$. Then $\UUU$ has infinitely many components and $f^{-1}(\UUU)\subset\UUU$. Each component of $\UUU$ is either a simply connected domain containing at most one point in $\PPP_f$, or a component of $f^{-s}(\AAA)$ and hence is an annulus.

Denote $\UUU_n=f^{-n}(\UUU)$ for $n\ge 0$. Then $\UUU_{n+1}\subset\UUU_n$. Each component of $\UUU_n$ is either a simply connected domain containing at most one point in $\PPP_f$, or a component of $f^{-n-s}(\AAA)$.

Denote $\KKK_r=\KKK(\Gamma)\smm\cup_{n\ge 0}f^{-n}(\KKK_c)$. For each component $K$ of $\KKK_r$, denote by $U_n(K)$ the component of $\UUU_n$ that contains $K$ for each $n\ge 0$.

\begin{proposition}\label{proper}
For each component $K$ of $\KKK_r$ and each $n\ge 0$, there exists an integer $m>n$ such that $U_m(K)\Subset U_n(K)$.
\end{proposition}

\begin{proof} If $U_n(K)$ is an annulus, then it is a component of $f^{-n-s}(\AAA)$. Pick a point $z\in K$. From Proposition \ref{2-continuum} (2), there exist two components $A_1, A_2$ of $f^{-m-s}(\AAA)$ for some $m>n$ such that both $A_1$ and $A_2$ are contained essentially  in $U_n(K)$and the $2$-connected continuum between $A_1$ and $A_2$, denoted by $E$, contains the point $z$. Note that both $\partial A_1$ and $\partial A_2$ are contained in $f^{-m-s}(\KKK_c)$. Thus, $U_m(K)\subset E$ and hence $U_m(K)\Subset U_n(K)$.

The same argument works when $U_n(K)$ is simply connected.
\end{proof}

\begin{proposition}\label{disk}
Let $K$ be a component of $\KKK_r$ and $n\ge 0$  an integer. If $U_n(K)$ is an annulus, then there exists an integer $m>n$ such that $U_m(K)$ is either simply connected or an annulus but is not contained essentially in $U_n(K)$.
\end{proposition}

\begin{proof}
Otherwise, $\{U_{m}(K)\}$ are all annuli and $U_{m+1}(K)$ is contained essentially  in $U_m(K)$ for all $m\ge n$. Then $\cap_{n\ge 0} U_n(K)$ is a component of $\JJJ(\Gamma)$. Contradiction.
\end{proof}

Let $K$ be a component of $\KKK_r$. By Proposition \ref{disk}, we know that as $n$ is large enough, either $U_n(K)$ is simply connected, or $U_n(K)$ is an annulus and one component $E_n(K)$ of $\cbar\smm U_n(K)$ contains at most one point of $\PPP_f$. Denote by $V_n(K)=U_n(K)\cup E_n(K)$ in the case or $V_n(K)=U_n(K)$ otherwise. Then $V_n(K)$ is a simply connected domain containing at most one point of $\PPP_f$ as $n$ is large enough and $V_{n+1}(K)\subset V_n(K)$ for all $n\ge 0$. Moreover, for each integer $n\ge 0$, there exists an integer $m>n$ such that $V_m(K)\subset U_n(K)$. Thus, $\cap_{n>0}V_n(K)=\cap_{n>0}U_n(K)$ and it is disjoint from $\cup_{n\ge 0}f^{-n}(\KKK_c)$ and $\JJJ(\Gamma)$. Thus, we have
$$
K=\bigcap_{n>0}U_n(K)=\bigcap_{n>0}V_n(K).
$$
Combining with Proposition \ref{proper}, we have:

\begin{corollary}\label{compact}
Each component of $\KKK_r$ is either a single point or a simply connected continuum.
\end{corollary}

\begin{proposition}\label{decomposition2}
Let $K$ be a periodic component of $\KKK_r$. Then $K$ is either a single point or the closure of a quasi-disk, which is a periodic Fatou domain of $f$.
\end{proposition}

\begin{proof} Let $K$ be a periodic component of $\KKK_r$ with period $p\ge 1$. Then as $n$ is large enough, $V_n(K)$ is simply connected and $f^p:\,V_{n+p}(K)\to V_n(K)$ is proper with at most one critical point. Moreover $V_{n+p}(K)\Subset V_n(K)$ by Proposition \ref{proper}. If $K$ contains no super-attracting periodic points of $f$, then
$$
\deg (f^p: \,V_{n+p}(K)\to V_n(K))=1
$$
as $n$ is large enough. Thus, $K$ is a single point. Otherwise $f^p: V_{n+p}(K)\to V_n(K)$ is a polynomial-like map and $K$ is the closure of a quasi-disk, which is a periodic Fatou domain of $f$.
\end{proof}

The open set $\UUU$ can be decomposed into $\UUU=\DDD\sqcup\GGG\sqcup\RRR$ as follows: \\
\indent (1) $\DDD$ is the union of simply connected components of $\UUU$ which are disjoint from $\PPP_f$.  \\
\indent (2) $\GGG$ is the union of simply connected components of $\UUU$ containing exactly one point of $\PPP_f$. \\
\indent (3) $\RRR$ is the union of annular components of $\UUU$ (each of them is a component of $f^{-s}(\AAA)$ which is either contained essentially in $\AAA$ or peripheral around a point $x\in\PPP_f\smm\KKK_c$, i.e. it separates the point $x$ from other points of $\PPP_f$).

Obviously, both $\GGG$ and $\RRR$ have only finitely many components.

\begin{proposition}\label{GGG}
$f^{-1}(\GGG)\subset\DDD\cup\GGG$ and each component of $\GGG$ contains exactly one component of $f^{-1}(\GGG)$.
\end{proposition}

\begin{proof}
Assume that $f^{-1}(\GGG)\cap\RRR\neq\emptyset$. Then $f^{-1}(\GGG)$ has a component $W\subset\RRR$ since $f^{-1}(\UUU)\subset\UUU$. Thus, $f(W)\subset f(\RRR)\subset f^{-s+1}(\AAA)$. But $f(W)$ is a component of $\GGG$ and $f^{-s+1}(\AAA)$ is disjoint from $\PPP_f$. Contradiction.

Let $G$ be a component of $\GGG$. Denote by $x$ be the unique point of $\PPP_f\cap G$. Then $f(x)$ is also contained in a component $G_1$ of $\GGG$. Thus, $G$ contains a component of $f^{-1}(G_1)$.

Let $W_1$ be the component of $\cbar\smm\KKK_c$ that contains $f(x)$. Let $W_0$ be the component of $f^{-1}(W_1)$ that contains the point $x$. Then $G\subset W_0$. Since $W_0$ contains a unique point in $f^{-1}(P_f)$, so does $G$. Therefore, $G$ contains exactly one component of $f^{-1}(\GGG)$.
\end{proof}

\begin{proposition}\label{DDD}
Let $K$ be a wandering component of $\KKK(\Gamma)$. There exists an integer $n>0$ such that $f^n(K)\subset\DDD$.
\end{proposition}

\begin{proof} Assume by contradiction that $f^n(K)\subset\RRR\cup\GGG$ for all $n>0$.

{\it Case 1}. There exists an integer $n_0>0$ such that $f^n(K)\subset\RRR$ for all $n\ge n_0$. Then $f^{n+s}(K)\subset\AAA$ for $n\ge n_0$ since each component of $\RRR$ is a component of $f^{-s}(\AAA)$. On the other hand, there is at least one component of $\KKK_e$ contained essentially in each component of $\AAA$. Thus, each component of $\AAA\smm\KKK_e$ is either a component of $\DDD$ or is contained in $\AAA^1$. Thus, $f^{n+s}(K)\subset\AAA^1$ for $n\ge n_0$ by the assumption. Therefore, $K\subset\JJJ(\Gamma)$. Contradiction.

{\it Case 2}. There exist infinitely many integers $k>1$ such that $f^k(K)\subset\GGG$. Then $f^{k-1}(K)\subset\GGG\cup\DDD$ by Proposition \ref{GGG}. Hence $f^{k-1}(K)\subset\GGG$ by the assumption. Therefore, $f^n(K)\subset\GGG$ for all $n>0$. It follows that $K$ is eventually periodic by Proposition \ref{GGG}. This is a contradiction since $K$ is wandering. The lemma is proved.
\end{proof}

Let $\DDD'$ be the union of components of $\DDD$ intersecting $f^{-1}(\KKK_e)$. Then $\DDD'$ has only finitely many components. Denote $\DDD''=\DDD\smm\DDD'$.
For each component $D$ of $\DDD''$, it is a component of $\cbar\smm f^{-1}(\KKK_e)$. Thus, $f(D)$ is also a component of $\DDD$.

\begin{proposition}\label{DDD'}
Let $K$ be a wandering component of $\KKK(\Gamma)$. There exists an integer $n>0$ such that $f^n(K)\subset\DDD'$.
\end{proposition}

\begin{proof}
By Proposition \ref{DDD}, there exists an integer $k>0$ and a component $D$ of $\DDD$ such that $f^k(K)\subset D$. Assume by contradiction that $f^{n+k}(K)\subset\DDD''$ for all $n\ge 0$. Let $D_n$ be the component of $\DDD''$ that contains $f^{n+k}(K)$. Then $f(D_n)=D_{n+1}$. Thus, they are either eventually periodic or wandering. Since $\partial D_n\subset\JJJ_f$, the former case is impossible since they are disjoint from $\PPP_f$. The latter case is also impossible by Sullivan's no wandering Fatou domain Theorem.
\end{proof}

\begin{proposition}\label{decomposition3}
Each wandering component of $\KKK(\Gamma)$ is a single point.
\end{proposition}

\begin{proof}
For each simply connected domain $D\subset\cbar\smm\PPP_f$, denote
$$
\text{h-diameter}(D)=\sup_{z, w\in D}\ell[\gamma(z,w)],
$$
where $\gamma(z,w)$ is an arc in $D$ connecting the points $z$ and $w$, and $\ell[\gamma(z,w)]$ is the infimum of the length of arcs in $\cbar\smm\PPP_f$ under the orbifold metric over all the arcs in $\cbar\smm\PPP_f$ homotopic to $\gamma(z,w)$ rel $\PPP_f\cup\{z,w\}$.

It is easy to verify that if $\overline D$ is locally connected and disjoint from super-attracting periodic points of $f$, then $\text{h-diameter}(D)<\infty$.

For each component $D$ of $\DDD$, $\overline D$ is locally connected by Corollary \ref{local} and disjoint from super-attracting periodic points of $f$. Thus, $\text{h-diameter}(D)<\infty$. Then there exists a constant $M<\infty$ such that
$$
\text{h-diameter}(D)\le M\text{ for each component $D$ of $\DDD'$}.
$$

Since $\overline{\DDD'}$ is disjoint from super-attracting cycles of $f$, there exists a constant $\lambda>1$ such that $\|f'\|\ge\lambda$ on $f^{-1}(\DDD')$,
the norm being with respect to the orbifold metric.

Let $K$ be a wandering component of $\KKK(\Gamma)$. From Proposition \ref{DDD'}, there exists an infinite increasing sequence $\{n_k\}$ of positive integers such that $f^{n_k}(K)\subset D_k$ and $D_k$ are components of $\DDD'$. Let $W_k$ be the component of $f^{-n_k}(D_k)$ that contains $K$, then $$
\text{h-diameter}(W_k)\le M\lambda^{-k}.
$$
Thus, the spherical diameter of $W_k$ converges to zero as $k\to\infty$. It follows that $K$ is a single point.
\end{proof}

{\noindent\it Proof of Theorem \ref{decomposition}}. Combining Propositions \ref{decomposition1}, \ref{decomposition2}, \ref{decomposition3} and Theorem \ref{renorm}, we obtain Theorem \ref{decomposition}. \qed

\section{Coding the quotient dynamics}

In this section we will prove Theorem \ref{coding}, which was
suggested to us by K. Pilgrim. At first we give some definitions.

\begin{definition} A {\bf dendrite} is a locally connected and uniquely arc-wise connected continuum.

Let $\TTT\subset\cbar$ be a dendrite. A continuous onto map $\tau: \TTT\to\TTT$ is called a {\bf finite dendrite map}
if there exists a finite tree $T_0\subset\TTT$ such that the following statements hold. \\
\indent (1) For each $n\ge 1$, $T_n:=\tau^{-n}(T_0)$ is also a finite tree with $v(T_n)=\tau^{-n}(v(T_0))$, where $v(\cdot)$ denotes the set of vertices of a tree. \\
\indent (2) $T_n\subset T_{n+1}$ and $v(T_n)\subset v(T_{n+1})$. \\
\indent (3) $\tau$ is a homeomorphism on each edge of $T_n$. \\
\indent (4) $\cup_{n\ge 0}T_n$ is dense in $\TTT$. \\
\indent (5) $\deg\tau:=\sup\{\#\tau^{-1}(x),\, x\in\TTT\}<\infty$.
\end{definition}

\subsection{The tower of tree maps}

\begin{definition}
By {\bf a tower of tree maps} we mean an infinite sequence of triples $\{T_n, \iota_n,\tau_n\}_{n\ge 0}$, where $T_n$ are finite trees, $\iota_n:\, T_n\to T_{n+1}$ are inclusions and $\tau_n:\, T_{n+1}\to T_n$ are continuous onto maps such that: \\
\indent (a) $\iota_n(\VVV_n)\subset\VVV_{n+1}=\tau_n^{-1}(\VVV_n)$, where $\VVV_n$ is the set of vertices of $T_n$, \\
\indent (b) $\tau_{n+1}(T_{n+2}\smm\iota_{n+1}(T_{n+1}))\subset T_{n+1}\smm\iota_{n}(T_{n})$, \\
\indent (c) $\tau_n$ is a homeomorphism on each edge of $T_{n+1}$, and \\
\indent (d) the following diagram commutes:
$$
\diagram
\cdots\rto^{\tau_n} & T_n\dto_{\iota_{n}}\rto^{\tau_{n-1}} & T_{n-1}\dto_{\iota_{n-1}}\rto^{\tau_{n-2}} & \cdots \rto^{\tau_{1}} & T_1\dto_{\iota_{1}}\rto^{\tau_0} & T_{0}\dto^{\iota_0} &  \\
\cdots\rto^{\tau_{n+1}} & T_{n+1}\rto^{\tau_{n}} & T_n\rto^{\tau_{n-1}} & \cdots\rto^{\tau_2} & T_2\rto^{\tau_{1}} & T_1\rto^{\tau_0} & T_0
\enddiagram
$$
\end{definition}

The {\bf degree of the tree map} $\tau_n: T_{n+1}\to T_n$ is defined by
$$
\deg\tau_n=\sup\{\#\tau_{n}^{-1}(y),\, y\in T_n\}.
$$
Note that the sequence $\{\deg\tau_n\}$ is increasing. The {\bf degree of the tower} is defined to be the limit of this sequence as $n\to\infty$.

A tower of tree maps $\{T_n, \iota_n, \tau_n\}_{n\ge 0}$ is called {\bf expanding} if there exist a constant $\lambda>1$ and a linear metric on $T_0$ such that $(\tau_0\circ\iota_0): T_0\to T_0$ is $C^1$ under this linear metric and the norm of its derivative is bigger than $\lambda$ on $T_0$.

\begin{theorem}\label{limit}
Let $\{T_n, \iota_n, \tau_n\}_{n\ge 0}$ be an expanding tower of tree maps. Suppose that its degree is bounded. Then there exist an expanding finite dendrite map $\tau: \TTT\to\TTT$ and inclusions $i_n: T_n\to\TTT$ for all $n\ge 0$ such that $i_n=i_{n+1}\circ\iota_n$ and the following diagram commutes:
$$
\diagram
\cdots\rto^{\tau_{n+1}} & T_{n+1}\dto_{i_{n+1}}\rto^{\tau_n} & T_n\dto_{i_n}\rto^{\tau_{n-1}}
& \cdots \rto^{\tau_1} & T_1\dto_{i_1}\rto^{\tau_0} & T_0\dto^{i_0} \\
\cdots\rto^{\tau} & \TTT\rto^{\tau} & \TTT\rto^{\tau} & \cdots\rto^{\tau} & \TTT\rto^{\tau} & \TTT
\enddiagram
$$
Moreover the finite dendrite map $\tau:\,\TTT\to\TTT$ is unique up to topological conjugacy.
\end{theorem}

The finite dendrite map $\tau:\,\TTT\to\TTT$ will be called the {\bf limit of the tower of tree maps} $\{T_n, \iota_n, \tau_n\}_{n\ge 0}$.

\begin{proof}
Let $|\cdot|$ be an expanding linear metric on $T_0$, i.e. there exists a constant $\lambda>1$ such that $|(\tau_0\circ\iota_0)'|>\lambda$ on $T_0$. Define a metric on $\iota_0(T_0)\subset T_1$ such that $\iota_0$ is an isometry. Define a metric on $T_1\smm\iota_0(T_0)$ such that the norm of the derivative of $\tau_0$ is a constant $\lambda_1>\max\{\lambda, d\}$, where $d$ is the degree of the tower. Then we get a metric on $T_1$ such that $\iota_0$ is an isometry and $\lambda_1\ge |\tau'_0|\ge\lambda$ on $T_1$. The total length of $T_1$ satisfies that $|T_1|\le |T_0|+(d/\lambda_1)|T_0|$.

Inductively, one may define a metric on $T_n$ for $n\ge 2$ such that $\iota_{n-1}$ is an isometry and $\lambda_1\ge |\tau'_{n-1}|\ge\lambda$ on $T_{n}$. By the condition (b), the total length of $T_n$ satisfies that
$$
|T_n|\le |T_{n-1}|+\dfrac{d}{\lambda_1}(|T_{n-1}-|T_{n-2}|).
$$
It deduces that
$$
|T_n|\le |T_0|+\dfrac{d}{\lambda_1}|T_0|+\left(\dfrac{d}{\lambda_1}\right)^2|T_0|+\cdots+\left(\dfrac{d}{\lambda_1}\right)^n|T_0|
\le \dfrac{\lambda_1}{\lambda_1-d}|T_0|.
$$

Denote by $\wt\SSS$ the space consisting of left infinite sequences $ (\cdots, t_1, t_0)$ with $t_n\in T_{n+k}$ for some integer $k\ge 0$, such that for any $n\ge 0$, if $t_n\in T_{n+k}$, then $t_{n+1}=\iota_{n+k}(t_n)$. Define an equivalent relation on $\wt\SSS$ by
$$
(\cdots, s_1, s_0, )\sim (\cdots, t_1, t_0)
$$
if there exists an integer $k$ such that $s_n=t_{n+k}$ whenever $n, n+k\ge 0$. Let $\SSS$ be the quotient space $\wt\SSS/\!\!\sim$. For each point $(\cdots, t_1, t_0)$ in $\wt\SSS$, denote by $[\cdots, t_1, t_0]\in\SSS$ the equivalent class of $(\cdots, t_1, t_0)$. Since $\iota_n$ is an isometry, there exists a metric $\rho$ on $\SSS$ such that the inclusion $i_n: T_n\to\SSS$ defined by:
$$
i_n(t)=[ \cdots, \iota_{n+1}\circ\iota_n(t),\iota_n(t), t]
$$
is an isometry. Clearly, $i_{n+1}\circ\iota_n=i_n$ on $T_n$.

Since $\tau_n\circ\iota_n=\iota_{n-1}\circ\tau_{n-1}$ on $T_n$, there exists a continuous onto map $\tau: \SSS\to \SSS$ such that $\tau\circ i_{n}=i_{n-1}\circ\tau_{n-1}$ on $T_{n}$. Moreover $\lambda_1\ge |\tau'|\ge\lambda$.

Since $\lambda_1>d$, the total length of $\SSS$ is bounded. Let $\TTT$ be the completion of $\SSS$. Then it is a dendrite. The map $\tau$ can be extended to a continuous onto map on $\TTT$ since $\tau$ is uniformly continuous.

One may easily check that all the conditions in the definition of finite dendrite maps are satisfied. The uniqueness of $\tau:\,\TTT\to\TTT$ comes from the condition (4).
\end{proof}

\subsection{From the decomposition to a dendrite map}

Let $f$ be a post-critically finite rational map with a stable Cantor multicurve $\G$. We want to define a tower of tree maps from $f$ and construct a semi-conjugacy from $f$ to the limit of the tower.

Denote by $\Gamma_n$ the collection of all the curves in $f^{-n}(\Gamma)$ for $n\ge 0$. Then each curve in $\Gamma_n$ is essential in $\cbar\smm f^{-n}(\PPP_f)$ and no two of them are homotopic in $\cbar\smm f^{-n}(\PPP_f)$.

For each curve $\gamma\in\Gamma_n$, denote by $\Gamma_{n+1}(\gamma)$ the curves in $\Gamma_{n+1}$ homotopic to $\gamma$ rel $f^{-n}(\PPP_f)$. Since $\Gamma$ is pre-stable and stable, the following results are easy to check:

\begin{proposition}\label{succeed}
Suppose that $\g\in\G_{n+2}$ is not homotopic to any curve in $\G_{n+1}$ rel $f^{-n-1}(\PPP_f)$. Then $f(\g)$ is not homotopic to any curve in $\G_{n}$ rel $f^{-n}(\PPP_f)$.
\end{proposition}

\begin{proposition}\label{adjacent}
For each curve $\gamma\in\Gamma_n$, we have $\Gamma_{n+1}(\gamma)\neq\emptyset$ and no curve in $\Gamma_{n+1}\smm\Gamma_{n+1}(\gamma)$ separates curves in $\Gamma_{n+1}(\gamma)$.
Moreover, for any two curves $\gamma_1$ and $\gamma_2$ in $\Gamma_n$, if no curve in $\Gamma_n$ separates $\gamma_1$ from $\gamma_2$, then no curve in $\Gamma_{n+1}\smm(\Gamma_{n+1}(\gamma_1)\cup\Gamma_{n+1}(\gamma_2))$ separates curves in $\Gamma_{n+1}(\gamma_1)\cup\Gamma_{n+1}(\gamma_2)$.
\end{proposition}

{\bf Dual trees}. For any $n\ge 0$, let $T_n$ be the dual tree of $\Gamma_n$ defined by the following: There is a bijection between vertices of $T_n$ and components of $\cbar\smm\Gamma_n$. Two vertices are connected by an edge if their corresponding components of $\cbar\smm\Gamma_n$ have a common boundary component, which is a curve in $\Gamma_n$. Thus, there is a bijection between edges of $T_n$ and curves in $\Gamma_n$. Denote by $e_{\gamma}$ the edge of $T_n$ corresponding to the curve $\gamma\in\Gamma_n$.

\vskip 0.24cm

{\bf Inclusion maps}.
The homotopy rel $f^{-n}(\PPP_f)$ induces an inclusion $\iota_n: T_n\to T_{n+1}$ by the following: For each curve
$\gamma\in\Gamma_n$, define
$$
\iota_n: e_{\gamma}\to\bigcup_{\beta\in\Gamma_{n+1}(\gamma)} e_{\beta}\cup\{\text{common endpoints of }e_{\beta}\}
$$
to be a homeomorphism preserving the orientation induced by a choice of orientations on these curves. The continuity of $\iota_n$ at vertices comes from Proposition \ref{adjacent}. The injectivity comes from the fact that no two curves in $\Gamma_n$ are homotopic rel $f^{-n}(\PPP_f)$.

\vskip 0.24cm

{\bf Induced tree maps}. Given any $n\ge 0$. A continuous map $\tau: T_{n+1}\to T_n$ is called an induced tree map if for each edge $e_{\gamma}$ of $T_{n+1}$ corresponding to a curve $\gamma\in\Gamma_{n+1}$, $\tau: e_{\gamma}\to e_{f(\gamma)}$ is a homeomorphism such that it preserves the orientation induced by the map $f:\gamma\to f(\gamma)$. It is easy to check that induced tree maps always exist.

\begin{proposition}\label{tree}
There exists a linear metric $\rho_1$ on the tree $T_1$ and an induced tree map $\tau_0: T_1\to T_0$ such that $\iota_0\circ\tau_0$ is linear on each edge of $T_1$ and $|(\iota_0\circ\tau_0)'|\ge\lambda$ for some constant $\lambda>1$.
\end{proposition}

\begin{proof} Let $\{e_1, \cdots, e_n\}$ be the edges of $T_0$. Let $M=(b_{ij})$ be the reduced transition matrix of $\Gamma$ as defined in \S2. Then its leading eigenvalue verifies $\lambda_0>1$ by Lemma \ref{eigenvalue} since $\Gamma$ is a Cantor multicurve. Thus, there exists a constant $\lambda\in (1, \lambda_0)$ and a positive eigenvector ${\bf v}=(v(e_i))$ such that $M{\bf v}>\lambda{\bf v}$ by Lemma A.1 in \cite{CT}. Define a linear metric $\rho_1$ on $T_1$ such that for each edge $e$ of $T_1$, it has length $v(\tau_0(e))$. Then the length of $\iota_0(e_i)$ is:
$$
|\iota_0(e_i)|=\ds\sum_j b_{ij}\, v(e_j)>\lambda v(e_i).
$$
Define $\tau_0: T_1\to T_0$ to be an induced tree map such that $\iota_0\circ\tau_0$ is linear on each edge of $T_1$. Then $|(\iota_0\circ\tau_0)'|>\lambda$.
\end{proof}

There exists an induced tree map $\tau_1: T_2\to T_1$ such that $\tau_1\circ\iota_1=\iota_0\circ\tau_0$ on $T_1$. Inductively, for each $n\ge 2$, there exists an induced tree map $\tau_{n-1}: T_{n}\to T_{n-1}$ such that $\tau_{n-1}\circ\iota_{n-1}=\iota_{n-2}\circ\tau_{n-2}$ on $T_{n-1}$. By Proposition \ref{succeed}, $\tau_{n+1}(T_{n+2}\smm\iota_{n+1}(T_{n+1}))\subset T_{n+1}\smm\iota_{n}(T_{n})$. Thus $\{T_n,\iota_n, \tau_n\}_{n\ge 0}$ is an expanding tower of tree maps with degree $\deg\tau_n\le\deg f$. We call it the {\bf induced tower of tree maps} of $f$ with respect to the multicurve $\Gamma$.

Denote by $\tau_f:\, \TTT(\Gamma)\to\TTT(\Gamma)$ the limit of the induced tower of tree maps of $f$ with respect to the multicurve $\Gamma$. Then it is an expanding finite dendrite map by Theorem \ref{limit}.

\vskip 0.24cm

{\it Proof of Theorem \ref{coding}}.
Let $\{T_n,\iota_n, \tau_n\}_{n\ge 0}$ be the induced tower of tree maps of $f$ with respect to $\Gamma$. We may identify $T_n$ with $i_n(T_n)\subset\TTT(\Gamma)$ by Theorem \ref{limit}. Then $\tau_n=\tau_f$.
Let $\III\subset T_1$ be the union of open edges of $T_1$ contained in $T_0$. Let $\sigma$ be the restriction of $\tau_f$ on $\III$. Then $\sigma$ is an expanding linear system. So $\JJJ_{\sigma}$ is dense in $\III$. Let $g=f: \AAA^1\to\AAA$ be the exact annular system obtained in Theorem \ref{as}. There exists a bijection from the components of $\AAA^1$ to the components of $\III$ according to the correspondence from $\Gamma_1$ to the edges of $T_1$ and the homotopy rel $f^{-1}(\PPP_f)$.

Define a map $\Theta_0: \JJJ_g\to\JJJ_{\sigma}$ through itineraries as in Proposition \ref{linear}. It is order-preserving. Since its image $\JJJ_{\sigma}$ is dense in $\III$, it can be extended to a continuous onto map $\Theta_0: \cbar\to T_0$ such that each component of $\cbar\smm\AAA$ maps to a vertex of $T_0$. It is easy to check that $\tau_f\circ \Theta_0=\Theta_0\circ f$ on $\AAA^1$.

Pulling back the map $\Theta_0$ by the above equation, we get a continuous onto map $\Theta_1: \cbar\to T_1$ such that each component of $\cbar\smm f^{-1}(\AAA)$ maps to a vertex of $T_1$ and $\tau_f\circ \Theta_1=\Theta_1\circ f$ on $f^{-1}(\AAA)$.

Inductively, we get a sequence of continuous maps $\Theta_n: \cbar\to T_n$ such that each component of $\cbar\smm f^{-n}(\AAA)$ maps to a vertex of $T_n$ and $\tau_f\circ \Theta_n=\Theta_n\circ f$ on $f^{-n}(\AAA)$. It is easy to check that $\Theta_n$ converges uniformly to a continuous onto map $\Theta:\,\cbar\to \TTT(\Gamma)$ as $n\to\infty$ and the map $\Theta$ satisfies all the conditions.
\qed

\section{Wandering continua}

We will prove Theorem \ref{wc} here. A continuum $E\subset\cbar\smm\PPP_f$ is called {\bf essential}\ if there are exactly two components of $\cbar\smm E$ containing points of $\PPP_f$ and each of them contains at least two points of $\PPP_f$.

\begin{lemma}\label{essential}
Let $f$ be a post-critically finite rational map. Suppose that $K\subset\JJJ_f$ is a wandering continuum. Then either $f^n(K)$ is simply connected for all $n\ge 0$, or there exists an integer $N\ge 0$ such that $f^n(K)$ is essential for $n\ge N$.
\end{lemma}

\begin{proof}
Set $K_n=f^n(K)$ for $n\ge 0$. Since $\#\PPP_f<\infty$ and $K$ is wandering, we have $K_n\cap \PPP_f=\emptyset$ for all $n\geq 0$. Thus, if $K_n$ is simply connected, then $K_m$ is also simply connected for $m\le n$.

Suppose that there is an integer $n_0\ge 1$ such that $K_{n_0}$ is not simply connected, then $K_n$ is not simply connected for all $n\ge n_0$. Let $p(K_n)\ge 1$ be the number of components of $\cbar\smm K_n$ containing points of $\PPP_f$. Since $K_n$ are pairwise disjoint, there are at most $(\#\PPP_f-2)$ continua $K_n$ such that $p(K_n)\ge 3$. Thus, there is an integer $n_1\ge n_0$ such that $p(K_n)\le 2$ for all $n\ge n_1$.

If $p(K_n)\equiv 1$ for all $n\ge n_1$, let $\widehat{K}_n$ be the union of $K_n$ together with the components of $\cbar\smm K_n$ disjoint from $\PPP_f$. Then $f:\widehat{K}_n\to \widehat{K}_{n+1}$ is a homeomorphism for $n\ge n_1$. Since $K_{n_1}$ is not simply connected, $\wh K_{n_1}\smm K_{n_1}$ is non-empty. Let $U$ be a component of $\wh K_{n_1}\smm K_{n_1}$. Then $U\cap \PPP_f=\emptyset$. If $U\cap\JJJ_f\neq\emptyset$, then $f^{m}(U)\supset\JJJ_f$ for some
$m\ge 1$. But $f^m(U)$ is a component of $\cbar\smm K_{n_1+m}$ which is a contradiction. So $U\cap \JJJ_f=\emptyset$. Noticing that $\partial U\subset K_{n_1}\subset \JJJ_f$, the simply connected domain $U$ is exactly a Fatou domain. But $\partial U$ is wandering. This is a contradiction since there is no wandering Fatou domain by Sullivan's theorem (cf.\  \cite{Mi}). Therefore, there is an integer $n_2\ge n_1$ such that $p(K_{n_2})=2$.

We claim that $p(K_n)\equiv 2$ for all $n\ge n_2$. Otherwise, assume that there is an integer $m>n_2$ such that $p(K_m)=1$. Then there is a disk $D$ containing $K_m$ for which $D\cap\PPP_f=\emptyset$. Let $D_n$ be the component of $f^{n-m}(D)$ containing $K_n$ for $n_2\le n\le m$. Then $D_n$ is disjoint from $\PPP_f$. So $p(K_n)=1$ for $n_2\le n\le m$. This contradicts $p(K_{n_2})=2$.

We may assume $\#\PPP_f\ge 3$ (otherwise $f$ is conjugate to the map $z\to z^{\pm d}$ and hence has no wandering continuum). Then $f$ has at most one exceptional point. If there is an integer $m\ge n_2$ such that $\cbar\smm K_m$ has a component containing exactly one $\PPP_f$ point, then there is a disk $D\supset K_m$ such that $D$ contains exactly one $\PPP_f$ point. Let $D_n$ be the component of $f^{n-m}(D)$ containing $K_n$ for $n_2\le n\le m$. Then $D_n$ is simply connected and contains at most one point of $\PPP_f$. Thus, $\cbar\smm K_n$ has a component containing exactly one $\PPP_f$ point for $n_2\le n\le m$. Therefore, either there exists an integer $N\ge n_2$ such that for $n\ge N$, $f^n(K)$ is essential, or $\cbar\smm f^n(K)$ has a component containing exactly one $\PPP_f$ point for all $n\ge n_2$.

In the latter case, denote by $U$ the component of $\cbar\smm K_{n_2}$ containing exactly one $\PPP_f$ point. If $U\cap\JJJ_f\neq\emptyset$, then there is an integer $k>0$ such that $\cbar\smm f^k(U)$ contains at most one point (an exceptional point). On the other hand, there is a disk $D\supset K_{n_2+k}$ such that $D$ contains exactly one $\PPP_f$ point. Let $D_{n_2}$ be the component of $f^{-k}(D)$ containing $K_{n_2}$. Then $D_n$ is simply connected and contains at most one point of $\PPP_f$. Thus, $U\subset D_{n_2}$. Therefore, $f^k(U)\subset D$ and hence $\cbar\smm D\subset\cbar\smm f^k(U)$ contains at most one point. This contradicts $\#\PPP_f\ge 3$. So $U$ is disjoint from $\JJJ_f$ and hence is a simply connected Fatou domain. This again contradicts Sullivan's no wandering Fatou domain theorem.
\end{proof}

\begin{lemma}\label{multicurve}
Suppose that $K\subset\JJJ_f$ is a non-simply connected wandering continuum. There is a multicurve $\G_K$ such that:

(1) for each curve $\g$ in $\G_K$, there are infinitely many continua $f^n(K)$ homotopic to $\g$ rel $\PPP_f$, and

(2) there is an integer $N_1\ge 0$ such that for $n\ge N_1$, $f^n(K)$ is essential and homotopic rel $\PPP_f$ to a curve in $\G_K$.
\end{lemma}

\begin{proof} By Lemma \ref{essential}, there is an integer $N\ge 0$ such that $f^n(K)$ is essential for $n\ge N$. Given any integer $m\ge N$, one may choose an essential Jordan curve $\be_n$ in $\cbar\smm\PPP_f$ for $N\le n\le m$ such that $\beta_n$ is homotopic to $f^n(K)$ rel $\PPP_f$ and all these curves are pairwise disjoint since $f^n(K)$ are pairwise disjoint. Let $\wt\G_m$ be the collection of all these curves.

Let $\G_m\subset\wt\G_m$ be a multicurve such that each curve in $\wt\G_m$ is homotopic to a curve in $\G_m$. Then each curve in $\G_m$ is homotopic to a curve in $\G_{m+1}$. This implies that $\#\G_m$ is increasing and hence there is an integer $m_0\ge N$ such that $\#\G_m$ is a constant for $m\ge m_0$ since any multicurve contains at most $\#\PPP_f-3$ curves. Therefore each curve in $\G_{m+1}$ is homotopic to a curve in $\G_m$ for $m\ge m_0$. This shows that the multicurves $\G_m$ are homotopic to each other for all $m\ge m_0$.

Let $\G_K\subset\G_{m_0}$ be the sub-collection consisting of curves $\g\in\G_{m_0}$ such that there are infinitely many $f^n(K)$ homotopic to $\g$ rel $\PPP_f$. Then it is non-empty and hence is a multicurve. Obviously, $\G_K$ is uniquely determined by $K$ and there is an integer $N_1\ge 0$ such that for $n\ge N_1$, $f^n(K)$ is essential and homotopic rel $\PPP_f$ to a curve in $\G_K$.
\end{proof}

\begin{lemma}\label{irreducible}
$\G_K$ is an irreducible Cantor multicurve.
\end{lemma}

\begin{proof}
By Lemma \ref{multicurve}, there exists an integer $N_1\ge 0$ such that $f^n(K)$ for every $n\ge N_1$ is homotopic to a curve in $\G_K$ rel $\PPP_f$. Thus, $\Gamma_K$ is pre-stable. For any pair $(\g,\al)\in\G_K\times\G_K$, there are integers $k_2>k_1\ge N_1$ such that $f^{k_1}(K)$ is homotopic to $\g$ and $f^{k_2}(K)$ is homotopic to $\al$ rel $\PPP_f$. Thus, $f^{k_1-k_2}(\al)$ has a component $\de$ homotopic to $\g$ rel $\PPP_f$. So for $1<i<k_2-k_1$ the curve $f^i(\de)$ is homotopic rel $\PPP_f$ to $f^{k_1+i}(K)$ and hence to a curve in $\G_K$ rel $\PPP_f$. This shows that $\G_K$ is irreducible.

Now we want to prove that $\G_K$ is a Cantor multicurve. We may apply Lemma \ref{Cantor} and assume by contradiction that for each $\g\in\G_K$, $f^{-1}(\g)$ has exactly one component homotopic rel $\PPP_f$ to a curve in $\G_K$.

Assume $N_1=0$ for simplicity. Denote by $\G_K=\{\g_0,\cdots,\g_{p-1}\}$ such that $\g_0$ is homotopic to $K$ and $\g_n$ is homotopic to a component of $f^{-1}(\g_{n+1})$ for $0\le n<p$ (set $\g_p=\g_0$). It makes sense since for each $\g\in\G_K$, $f^{-1}(\gamma)$ has exactly one component homotopic rel $\PPP_f$ to a curve in $\G_K$. Then $f^n(K)$ is homotopic to $\g_k$ if $n\equiv k(\m p)$.

For $n\ge 0$ and $k\ge 1$ denote by $A(n,n+kp)$ the unique annular component of $\cbar\smm(f^n(K)\cup f^{n+kp}(K))$. Then $f^m: A(n,n+kp)\to A(n+m, n+kp+m)$ is proper for any $m\ge 1$. This is because $A(n+m, n+kp+m)$ is disjoint from $\PPP_f$ and homotopic to $f^{n+m}(K)$, so $f^{-m}(A(n+m, n+kp+m))$ has a unique component homotopic to $f^n(K)$, which must be $A(n,n+kp)$.

One may choose three distinct positive integers $k_1, k_2, k_3$ such that $A(k_1p, k_2p)$ contains $f^{k_3p}(K)$. Hence $A(k_1p, k_2p)$ contains points of $\JJJ_f$. On the other hand, for all $m\ge 1$, $f^m$ is proper on $A(k_1p, k_2p)$, whose image is disjoint from $\PPP_f$. Thus $\{f^m\}$ is a normal family in $A(k_1p, k_2p)$ since $\PPP_f$ contains at least three points. It is a contradiction.
\end{proof}

{\noindent\it Proof of Theorem \ref{wc}}. Suppose that $K\subset\JJJ_f$ is a wandering continuum and is not simply connected. Then $\G_K$ is an irreducible Cantor multicurve by Lemma \ref{irreducible}. By Lemma \ref{multicurve}, there exists an integer $N_1\ge 0$ such that $f^n(K)$ for every $n\ge N_1$ is homotopic to a curve in $\G_K$ rel $\PPP_f$. We assume $N_1=0$ for simplicity.

Let $\EEE$ be the collection of the essential components $E$ of $f^{-m}(f^n(K))$ for $n,m \ge 0$ such that $f^i(E)$ is homotopic to a curve in $\G_K$ for $0\le i<m$. Then $f(E)\in\EEE$ for any element $E\in\EEE$, and any two elements in $\EEE$ are either disjoint or one contains the other as subsets of $\cbar$.

For each $\g\in\G_K$, let $\EEE(\g)$ be the sub-collection of continua in $\EEE$   homotopic to $\g$ rel $\PPP_f$. We claim that for any continuum $E\in\EEE(\g)$, there are two disjoint continua $E_1,E_2\in\EEE(\g)$ such that $E\subset A(E_1, E_2)$, where $A(E_1,E_2)$ denotes the unique annular component of $\cbar\smm(E_1\cup E_2)$.

Consider $\{f^n(E)\}$ for $0\le n\le 2\cdot \#\G_K+1$. There is a curve $\be\in\G_K$ such that at least three of them are contained in $\EEE(\be)$.
Let us eumerate them by $f^{n_i}(E)$ $(i=1,2,3)$
such that $f^{n_3}(E)\subset A(f^{n_1}(E), f^{n_2}(E))$. Let $A$ be the component of $f^{-n_3}(A(f^{n_1}(E), f^{n_2}(E)))$ that contains
$E$. Then $A=A(E_1, E_2)$ where $E_i$ ($i=1,2$) is a component of $f^{-n_3}(f^{n_i}(E))$. The claim is proved.

Denote $A(\g)=\cup A(E,E')$ for all disjoint pairs $E,E'\in\EEE(\g)$. Then $A(\g)$ is an annulus in $\cbar\smm\PPP_f$
homotopic to $\g$ rel $\PPP_f$, and $A(\g)\cap A(\be)=\emptyset$ for distinct curves $\be,\g\in\G_K$.

Denote by $\AAA=\cup_{\gamma\in\Gamma_K} A(\g)$ and $\AAA^1$ the union of components of $f^{-1}(\AAA)$ homotopic to curves in $\G_K$. Then $\AAA^1\subset\AAA$ and $\partial\AAA\subset\partial\AAA^1$ by the claim and the definition of $\EEE$. So $g=f|_{\AAA^1}:\AAA^1\to\AAA$ is an exact annular system. In particular, $f^n(K)\subset\AAA$ and hence $f^n(K)\subset\AAA^1$ for all $n\ge 0$. So $K\subset\JJJ_g$. Since $K$ is connected, it must be contained in a component of $\JJJ_g$ which is a Jordan curve by  Theorem \ref{curve}. But $K$ is essential. Therefore, $K$ coincides with the Jordan curve.
\qed

\section{Foldings of polynomials}

In this section, we introduce a topological surgery to produce branched coverings with Cantor multicurves from  polynomials. We give two criteria for these maps to be equivalent to rational maps. This section is self-contained and can be read independently to the previous sections.

\subsection{Folding maps}

Let $F$ be a post-critically finite branched covering of $\cbar$ and $\be$ an essential Jordan curve in $\cbar\smm\PPP_F$. The pair $(F,\be)$ is called a {\bf folding map}\ if $F^{-1}(\be)$ contains at least two curves and each of them is essential and homotopic to $\be$ rel $\PPP_F$.

Let $(F,\be)$ be a folding map. Denote by $U, V$ the two components of $\cbar\smm\be$. Denote by $U_1, V_1$ the two disk components of $\cbar\smm F^{-1}(\be)$ such that $U_1$ is homotopic to $U$ (i.e. there is an isotopy $\theta$ of $\cbar$ rel $\PPP_F$ such that $U_1=\theta(U)$). Then $V_1$ is homotopic to $V$. There are three possibilities:

\vskip 0.24cm

\indent type A: $F(U_1)=U$ and $F(V_1)=U$, \\
\indent type B: $F(U_1)=U$ and $F(V_1)=V$,\text{ and} \\
\indent type C: $F(U_1)=V$ and $F(V_1)=U$.

\vskip 0.24cm

Define
$$
m(F,\be)=\#\{\text{components of }F^{-1}(\be)\}, \text{ and}
$$
$$
d(F,\be)=\left\{\begin{array}{ll} \deg (F|_{U_1}), & \text{ in type A}, \\
\min\{\deg (F|_{U_1}), \deg (F|_{V_1})\}, \mystrut & \text{ in type B}, \\
\sqrt{\deg (F|_{U_1})\deg (F|_{V_1})}, & \text{ in type C}.
\end{array}\right.
$$

The following facts are easy to check: \\
\indent $\bullet$ $(F,\be)$ is of type A if and only if $m(F,\be)$ is an even number. \\
\indent $\bullet$ $(F^n,\be)$ is also a folding map with $m(F^n,\be)=m(F,\beta)^n$ and $d(F^n,\be)=d(F,\be)^n$ for $n\ge 1$. \\
\indent $\bullet$ If $(F,\be)$ is of type A (resp. type B), then $(F^n,\be)$ is also of type A (resp. type B) for $n\ge 1$; If $(F,\be)$ is of type C, then $(F^{2k-1},\be)$ is of type C and $(F^{2k},\be)$ is of type B for $k\ge 1$.

It is obvious that $F$ has no Thurston obstruction if $F^2$ has no Thurston obstruction. We will consider only types A and B in this section.

\vskip 0.24cm

Let $(F,\be)$ be a folding map of type A and $g$ be a polynomial. We say that $(F,\be)$ is a {\bf folding of the polynomial} $g$ if $F$ is Thurston equivalent to another map $G$ through a pair of homeomorphisms $(\phi,\psi)$ of $\cbar$ such that $G^{-1}(U)$ has a disk component $U_1\Subset U$ and $G|_{U_1}=g$, where $U$ is a component of $\cbar\smm\phi(\beta)$. It is obvious that the polynomial $g$ is also post-critically finite.

Let $(F,\be)$ be a folding map of type B and $(g_1, g_2)$ be a pair of polynomials. We say that $(F,\be)$ is a {\bf folding of the pair of polynomials} $(g_1, g_2)$ if $F$ is Thurston equivalent to another map $G$ through a pair of homeomorphisms $(\phi,\psi)$ of $\cbar$ such that there are disjoint Jordan domains $U$ and $V$ in $\cbar$ with both $\partial U$ and $\partial V$ homotopic to $\phi(\beta)$ rel $\PPP_G$, both $G^{-1}(U)$ and $G^{-1}(V)$ have a disk component $U_1\Subset U$ and $V_1\Subset V$, $G|_{U_1}=g_1$ and $G|_{V_1}=g_2$. Obviously, both $g_1$ and $g_2$ are also post-critically finite.

\vskip 0.24cm

Note that the multicurve consisting of the single Jordan curve $\beta$ is a Cantor multicurve. Denote by $d_i$ ($1\le i\le m(F,\beta)$) the degrees of $F$ on the components of $F^{-1}(\be)$. The leading eigenvalue of its transition matrix is
$$
\lambda_{\be}=\dfrac1{d_1}+\cdots +\dfrac1{d_{m(F,\beta)}}.
$$
The next result relates a folding map to a folding of polynomials.

\begin{proposition}\label{polynomial} Let $(F,\be)$ be a folding map of type A (or type B) with $\lambda_{\be}<1$. Then $(F,\be)$ is the folding of a polynomial $g$ (or a pair of polynomials $(g_1, g_2)$) if and only if any stable multicurve disjoint from $\be$ is not a Thurston obstruction. Moreover, the polynomial $g$ (or the pair of polynomials $(g_1,g_2)$) is unique up to holomorphic conjugation.
\end{proposition}

\begin{proof}
Suppose that $(F,\be)$ is a folding of a polynomial $g$ with $\lambda_\be<1$. By the definition, there is a branched covering $G$ of $\cbar$ which is Thurston equivalent to $F$ through a pair of homeomorphisms $(\phi,\psi)$ of $\cbar$, such that $G^{-1}(U)$ has a disk component $U_1\Subset U$ and $G|_{U_1}=g$, where $U$ is a component of $\cbar\smm\phi(\beta)$.

Let $\G$ be an irreducible multicurve of $G$ disjoint from $\alpha:=\phi(\beta)$. If there is one curve $\g\in\G$ homotopic to $\al$ rel $\PPP_G$, then $\G=\{\g\}$ and hence $\lambda_\G<1$. Now we assume that for any $\g\in\G$, $\g$ is not homotopic to $\al$ rel $\PPP_G$. Then $\g\subset U$. Otherwise $\g$ is contained in the other component $V$ of $\cbar\smm\{\al\}$. Thus each curve in $F^{-1}(\g)$ is contained in the annulus $F^{-1}(V)$ and hence is either non-essential or homotopic to $\al$ rel $\PPP_G$.

Set $\PPP=(\PPP_G\cap U)\cup\{\infty\}$. Then $\PPP_g\subset\PPP$, $g(\PPP)\subset\PPP$ and $\G$ is a multicurve of the marked polynomial $(g,\PPP)$. Thus $\lambda_\G<1$ by \cite[Theorem 3.3]{CT}. Applying Lemma \ref{Mc}, we see that $\lambda_\Gamma<1$ for any stable multicurve $\Gamma$ of $G$ disjoint from $\al$.

Conversely, suppose that any stable multicurve of $F$ disjoint from $\be$ is not a Thurston obstruction. Let $W$ be the component of $\cbar\smm\beta$ and $W_1$ be the component of $\cbar\smm F^{-1}(\be)$ homotopic to $W$. Then $F(W_1)=W$ and there is an isotopy $\theta$ of $\cbar$ rel $\PPP_F$ such that $\theta(W_1)\Subset W$. Set $G_1=F\circ\theta^{-1}$. Then $G_1$ is Thurston equivalent to $F$.

Denote $\PPP=\PPP_F\cap W$. Then $G_1(\PPP)\subset\PPP$ and $G_1: (\overline{\theta(W_1)}, \PPP)\to (\overline{W}, \PPP)$ is a marked repelling system (ref to \cite{CT}). Applying Lemma 2.1 and Theorem 3.5 in \cite{CT}, there exist two Jordan domains $V_1\Subset V$ in $\cbar$, a polynomial-like map $g_1: V_1\to V$ and a pair of homeomorphisms $(\phi, \psi)$ from $\overline{W}$ to $\overline{V}$ such that $\psi$ is isotopic to $\phi$ rel $\PPP\cup\partial W$,
$\psi(\theta(W_1))=V_1$ and $\phi\circ G_1\circ\psi^{-1}=g_1$ on $\overline{V_1}$. Extend $(\phi,\psi)$ to homeomorphisms of $\cbar$ such that they coincide with each other outside of $W$. Let $G_2=\phi\circ G_1\circ\psi^{-1}$. Then $G_2$ is Thurston equivalent to $F$ and $G_2|_{V_1}=g_1$ is a polynomial-like map. By the Straightening Theorem \cite{DH2}, there is a quasiconformal map $h$ of $\cbar$ such that $(h\circ G_2\circ h^{-1})|_{h(V_1)}$ is the restriction of a polynomial $g$. Therefore, $(F,\be)$ is a folding of the polynomial $g$. The uniqueness of $g$ comes from Thurston's Theorem.

This argument also works for type B. We omit its proof.
\end{proof}

\begin{proposition}\label{construction1}
Let $g$ be a post-critically finite polynomial with $\#\PPP_g\ge 3$. Let $m\ge 2$ be an even number and $\{d_1, \cdots, d_m\}$ be a sequence of integers such that $d_1=\deg g$ and
$$
\dfrac1{d_1}+\cdots+\dfrac1{d_{m}}<1.
$$
Then there exists a folding $(F,\beta)$ of the polynomial $g$ such that $\lambda_\be<1$, $m(F,\beta)=m$ and the degrees of $F$ on all the components of $F^{-1}(\beta)$ are $\{d_1, \cdots, d_m\}$.
\end{proposition}

\begin{proof}
Let $\beta\subset\cbar$ be the unit circle. For $1\le i\le m$, let $\delta_i\subset\cbar$ be a round circle with center zero and radius $r_i$ such that $1=r_1<\cdots<r_m<\infty$.  Define
$$
F(z)=\left(\dfrac{z}{r_i}\right)^{(-1)^{i+1}d_i}\text{ for }z\in\delta_i.
$$
Then $F:\delta_i\to\beta$ is a covering of degree $d_i$.

Set $\rho(z)=z/(1-|z|)$. It is a homeomorphism from the unit disk $\D$ to $\C$. Define $G_0=\rho^{-1}\circ g\circ\rho$ on $\D$.

Let $D_m$ be the unbounded component of $\cbar\smm\delta_m$. Define $G_m: D_m\to\D$ to be a branched covering such that its boundary value coincides with the definition of $F$ on $\delta_m$ and its critical values are eventually periodic points of $G_0$.

For each $1\le i<m$, let $A_i$ be the annulus bounded by $\delta_{i}$ and $\delta_{i+1}$. Define $G_i: A_i\to\cbar\smm\overline{\D}$ to be a branched covering such that its boundary value coincides with the definition of $F$ on $\delta_{i}\cup\delta_{i+1}$ and its critical values land on the pre-images under $G_m$ of eventually periodic points of $G_0$. Define
$$
F=\left\{\begin{array}{lll}
G_0: & \overline\D\to\overline\D, \\
G_i: & \overline{A_i}\to\cbar\smm\D\text{ for }1\le i< m, \\
G_m: & \overline{D_m}\to\overline{\D}\ .
\end{array}\right.
$$
Then $(F, \beta)$ is a folding of the polynomial $g$ with $\lambda_\be<1$.
\end{proof}

\begin{proposition}\label{construction2}
Let $(g_1, g_2)$ be a pair of post-critically finite polynomials such that $\#\PPP_{g_1}\ge 3$, $\#\PPP_{g_2}\ge 3$ and $\deg g_1+\deg g_2\ge 5$. Let $m\ge 3$ be an odd number and $\{d_1, \cdots, d_m\}$ be a sequence of integers such that $d_1=\deg g_1$, $d_m=\deg g_2$ and
$$
\dfrac1{d_1}+\cdots+\dfrac1{d_{m}}<1.
$$
Then there exists a folding $(F,\beta)$ of the pair $(g_1, g_2)$ such that $\lambda_\be<1$, $m(F,\beta)=m$ and the degrees of $F$ on all the components of $F^{-1}(\beta)$ are $\{d_1, \cdots, d_m\}$.
\end{proposition}

\begin{proof}
Let $\beta\subset\cbar$ be the unit circle. For $1\le i\le m$, let $\delta_i\subset\cbar$ be a round circle with center zero and radius $r_i$ such that $1=r_1<\cdots<r_m<2$.  Define
$$
F(z)=\left(\dfrac{z}{r_i}\right)^{(-1)^{i+1}d_i}\text{ for }z\in\delta_i.
$$
Then $F:\delta_i\to\beta$ is a covering of degree $d_i$.

Define $G_0=\rho^{-1}\circ g_1\circ\rho$ on $\D$, where $\rho(z)=z/(1-|z|)$.

Let $\alpha=\{z\in\C:\, |z|=2\}$. Let $\Delta$ be the unbounded component of $\cbar\smm\al$. Set $\rho_1(z)=|z|/(z(1-|z|))$. It is a homeomorphism from $\Delta$ to $\C$. Define $G_m=\rho_1^{-1}\circ g_2\circ\rho_1$ on $\Delta$.

Let $D_m$ be the unbounded component of $\cbar\smm\delta_m$. Then $\Delta\Subset D_m$. The map $G_m$ can be extended to a branched covering from $D_m$ to $\cbar\smm\overline{\D}$ with the same degree such that its boundary value coincides with the definition of $F$ on $\delta_m$.

For each $1\le i<m$, let $A_i$ be the annulus bounded by $\delta_{i}$ and $\delta_{i+1}$. When $i$ is an odd number, define $G_i: A_i\to D_m$ to be an branched covering such that its boundary value coincides with the definition of $F$ on $\delta_{i}\cup\delta_{i+1}$ and its critical values are eventually periodic points of $G_m$ in $\Delta$. When $i$ is an even number, define $G_i: A_i\to\D$ to be an branched covering such that its boundary value coincides with the definition of $F$ on $\delta_{i}\cup\delta_{i+1}$ and its critical values are eventually periodic points of $G_0$. Define
$$
F=\left\{\begin{array}{lll}
G_0: & \overline\D\to\overline\D, \\
G_i: & \overline{A_i}\to\overline\D\text{ for even number }1<i<m, \\
G_i: & \overline{A_i}\to\overline{D_m}\text{ for odd number }1\le i<m, \\
G_m: & \overline{D_m}\to\overline{D_m}\ .
\end{array}\right.
$$
Then $(F, \beta)$ is a folding of the pair of polynomials $(g_1, g_2)$ with $\lambda_\be<1$.
\end{proof}

\begin{theorem}\label{no1}
Let $(F,\be)$ be a folding of a polynomial (or a pair of polynomials) with $\lambda_\be<1$. Suppose that $d(F,\be)< m(F,\be)$. Then $F$ has no Thurston obstructions.
\end{theorem}

Let $(F,\be)$ be a folding of a polynomial (or a pair of polynomials) and $T\subset\cbar\smm\beta$ be a finite tree whose endpoints are contained in $\PPP_F$. We call it an {\bf injective tree} if there exists an integer $p\ge 1$ such that $F^p(T)$ is isotopically contained in $T$ (i.e. there exists a homeomorphism $\theta$ of $\cbar$ isotopic to the identity rel $\PPP_F$ such that $F^p(T)\subset\theta(T)$) and $F^p$ is injective on $T$.

\begin{theorem}\label{no2}
Let $(F,\be)$ be a folding of a polynomial (or a pair of polynomials) with $\lambda_\be<1$. Suppose that there exist an injective tree $T\subset\cbar\smm\beta$ and an integer $k\ge 1$ such that $F^{-k}(T)$ has a component homotopic to $\beta$ rel $\PPP_F$. Then $F$ has no Thurston obstructions.
\end{theorem}

{\bf Remark}. (1) We will give an example in \S8.4 a folding $(F,\beta)$ of a polynomial with $\lambda_\beta<1$ such that it has Thurston obstructions. We will show that the condition in each of the above two theorems is not necessary.

(3) Let $(G, \alpha)$ be a folding of a polynomial $g$ such that $G^{-1}(U)$ has a disk component $U_1\Subset U$ and $G|_{U_1}=g$, where $U$ is a component of $\cbar\smm\al$. Suppose that $T$ is an injective tree of $(G, \alpha)$ with period $p\ge 1$. Then there exists a homeomorphism $\theta$ of $\cbar$ isotopic to the identity rel $\PPP_G$ such that $\theta(T)$ consists of finitely many internal rays in periodic Fatou domains of $g$ together with their endpoints. Moreover, $\theta(T)=g^p(\theta(T))$. This shows that a folding of a polynomial have an injective tree if and only if $g$ has non-empty bounded Fatou domains.

\vskip 0.24cm

Apply Theorems \ref{no1} and \ref{no2}, we have following results.

\begin{theorem}\label{apply1}
(1) Let $g$ be a post-critically finite polynomial with $\#\PPP_g\ge 3$ such that $g$ has non-empty bounded Fatou domains. Then there exist a post-critically finite rational map $f$ and an essential curve $\beta$ in $\cbar\smm\PPP_f$ such that $(f,\beta)$ is a folding of $g$. The degree of $f$ can be taken to be any integer $d$ with
$$
d\ge\max\{\deg g+2, 5\}.
$$

(2) Let $(g_1, g_2)$ be a pair of post-critically finite polynomials with $\#\PPP_{g_1}\ge 3$ and $\#\PPP_{g_2}\ge 3$ such that $\deg g_1+\deg g_2\ge 5$ and $g_1$ or $g_2$ has non-empty bounded Fatou domains. Then there exist a post-critically finite rational map $f$ and an essential curve $\beta$ in $\cbar\smm\PPP_f$ such that $(f,\beta)$ is a folding of the pair $(g_1, g_2)$. The degree of $f$ can be taken to be any integer $d$ with
$$
d>\deg g_1+\deg g_2+\dfrac{\deg g_1\deg g_2}{\deg g_1\deg g_2-\deg g_1-\deg g_2}.
$$
\end{theorem}

\begin{proof} (1) Choose $m=2$ and $d_1=\deg g$. For any integer $d\ge\max\{\deg g+2, 5\}$, let $d_2=d-d_1$. Then
$$
\dfrac{1}{d_1}+\dfrac{1}{d_2}<1.
$$
By proposition \ref{construction1}, there exists a folding $(F, \beta)$ of $g$ with $\lambda_\beta<1$ such that $\deg F=d$.

Since $g$ has non-empty bounded Fatou domains, there exists an injective tree $T$ for $F$. Let $A$ be the unique annular component of $\cbar\smm F^{-1}(\beta)$. We further require that $F$ has exactly two critical points in $A$, which map by $F$ to the endpoints of a component of $F^{-1}(T)$. Let $E$ be the component of $F^{-2}(T)$ containing the two critical points of $F$ in $A$. Then $E$ is homotopic to $\beta$ rel $\PPP_F$. Applying Theorem \ref{no2}, we conclude that $F$ is Thurston equivalent to a rational map.

(2) Choose $m=3$, $d_1=\deg g_1$ and $d_3=\deg g_2$. For any integer $d$ satisfying the condition (2) of the theorem, let $d_2=d-d_1-d_3$. Then
$$
d_2>\dfrac{d_1d_3}{d_1d_3-d_1-d_3}.
$$
Thus $\sum_{i=1}^m (1/d_i)<1$. By proposition \ref{construction2}, there exists a folding $(F, \beta)$ of the pair $(g_1, g_2)$ with $\lambda_\beta<1$ such that $\deg F=d$.

Since $g_1$ or $g_2$ has non-empty bounded Fatou domains, there exists an injective tree $T$ for $F$. As above, we further require that $F$ has exactly two critical points in an annular component $A$ of $\cbar\smm F^{-1}(\beta)$, which map by $F$ to the endpoints of $T$. Let $E$ be the component of $F^{-1}(T)$ containing the two critical points of $F$ in $A$. Then $E$ is homotopic to $\beta$ rel $\PPP_F$. Applying Theorem \ref{no2}, we conclude that $F$ is Thurston equivalent to a rational map.
\end{proof}

\vskip 0.24cm

{\bf Remark}. The folding map we have constructed in (1) satisfies the condition of Theorem \ref{no2} but not Theorem \ref{no1}.

\begin{theorem}\label{apply2}
(1) Let $g$ be a post-critically finite polynomial with $\#\PPP_g\ge 3$. Then there exist a post-critically finite rational map $f$ and an essential curve $\beta$ in $\cbar\smm\PPP_f$ such that $(f,\beta)$ is a folding of $g$. The degree of $f$ can be taken to be any integer $d$ with
$$
d\ge\deg g+(\deg g+1)(\deg g+5).
$$

(2) Let $(g_1, g_2)$ be a pair of post-critically finite polynomials with $\#\PPP_{g_1}\ge 3$ and $\#\PPP_{g_2}\ge 3$ such that $\deg g_1\le\deg g_2$ and $\deg g_1+\deg g_2\ge 5$ . Then there exist a post-critically finite rational map $f$ and an essential curve $\beta$ in $\cbar\smm\PPP_f$ such that $(f,\beta)$ is a folding of the pair $(g_1, g_2)$. The degree of $f$ can be taken to be any integer $d$ with
$$
d\ge\deg g_1+\deg g_2+\deg g_1(\deg g_2+7).
$$
\end{theorem}

\begin{proof}
(1) Choose $m=\deg g+1$ when $\deg g+1$ is an even number or $m=\deg g+2$ when $\deg g+2$ is an even number. Choose $d_1=\deg g$ and $d_i=\deg g+5$ for $1<i<m$. For any integer $d$ with $d\ge d_1+(d_1+1)(d_1+5)$, let $d_m=d-\sum_{i=1}^{m-1}d_i$. Then $d_m\ge d_1+5$. Thus
$$
\sum_{i=1}^m\dfrac{1}{d_i}\le\dfrac{1}{d_1}+\dfrac{m-1}{d_1+5}\le \dfrac{1}{d_1}+\dfrac{d_1+1}{d_1+5}<1.
$$
By proposition \ref{construction1}, there exists a folding $(F, \beta)$ of the pair $(g_1, g_2)$ with $\lambda_\beta<1$ such that $\deg F=d$. Applying Theorem \ref{no1}, we conclude that $F$ is Thurston equivalent to a rational map.

(2) Choose $m=\deg g_1+1$ when $\deg g_1+1$ is an odd number or $m=\deg g_1+2$ when $\deg g_1+2$ is an odd number. Choose $d_1=\deg g_1$, $d_m=\deg g_2$ and $d_i=d_m+7$ for $1<i\le m-2$. For any integer $d$ with $d\ge d_1+d_m+d_1(d_m+7)$, let $d_{m-1}=d-d_m-\sum_{i=1}^{m-2}d_i$. Then $d_{m-1}\ge d_m+7$. Thus
$$
\sum_{i=1}^m\dfrac{1}{d_i}\le\dfrac{1}{d_1}+\dfrac{1}{d_m}+\dfrac{m-2}{d_m+7}.
$$

{\it Case 1. $d_1=d_m$}. Then $d_1\ge 3$. Thus
$$
\sum_{i=1}^m\dfrac{1}{d_i}\le\dfrac{2}{d_1}+\dfrac{d_1}{d_1+7}<1.
$$

{\it Case 2. $2=d_1<d_m$}. Then $m=3$. Thus
$$
\sum_{i=1}^m\dfrac{1}{d_i}\le\dfrac{1}{2}+\dfrac{1}{d_m}+\dfrac{1}{d_m+7}\le\dfrac{1}{2}+\dfrac{1}{3}+\dfrac{1}{10}<1.
$$

{\it Case 3. $3\le d_1<d_m$}. Then
$$
\sum_{i=1}^m\dfrac{1}{d_i}\le\dfrac{1}{d_1}+\dfrac{1}{d_m}+\dfrac{d_1}{d_m+7}\le\left(\dfrac{2}{d_m}+\dfrac{d_m}{d_m+7}\right)+
\left(\dfrac{d_m-d_1}{d_1d_m}-\dfrac{d_m-d_1}{d_m+7}\right).
$$
Since $d_1\ge 3$ and $d_m\ge 4$, we have $d_1d_m>d_m+7$. Thus the second part of the last term of the above inequality is negative. The first part of the last term is less than $1$. Thus $\sum_{i=1}^m (1/d_i)<1$.

By proposition \ref{construction2}, there exists a folding $(F, \beta)$ of the pair $(g_1, g_2)$ with $\lambda_\beta<1$ such that $\deg F=d$. Applying Theorem \ref{no1}, we conclude that $F$ is Thurston equivalent to a rational map.
\end{proof}

\vskip 0.24cm

{\it Proof of Theorem \ref{folding}}. This is a direct consequence of Theorem \ref{apply2} (1).

\subsection{Proof of Theorems \ref{no1} and \ref{no2}}

Let $(F,\be)$ be a folding of a polynomial (or a pair of polynomials). For any two essential Jordan curves $\g$ and $\al$ in $\cbar\smm\PPP_F$, set $k(\g,\al)$ to be their geometric intersection number. It is defined by
$$
k(\g,\al)=\min\{\#(\de\cap\al)\},
$$
where the minimum is taken over all the choices of $\de$ in the homotopy class of $\g$. By definition $k(\g,\al)=0$ if $\g$ is homotopic to $\al$ rel $\PPP_F$.

\begin{lemma}\label{intersection}
Let $\G$ be an irreducible multicurve of $F$. Then either $k(\g,\be)\neq 0$ for all $\g\in\G$ or $k(\g,\be)=0$ for all $\g\in\G$.
\end{lemma}

\begin{proof}
Suppose that $k(\g,\be)=0$ for some $\g\in\G$. For any curve $\al\in\G$, since $\G$ is irreducible, $\al$ is homotopic to a component of $F^{-n}(\g)$ rel $\PPP_F$ for some $n\ge 0$. Let $\de$ be a Jordan curve in $\cbar\smm\PPP_F$ homotopic to $\g$ rel $\PPP_F$ such that it is disjoint from $\be$, then $\al$ is homotopic to a component of $F^{-n}(\de)$ rel $\PPP_F$, which is disjoint from $F^{-n}(\be)$. Thus, $k(\al,\be)=0$ since $F^{-n}(\be)$ contains a curve homotopic to $\be$ rel $\PPP_F$.
\end{proof}

\begin{lemma}\label{Deg}
Let $\g$ and $\al$ be essential Jordan curves in $\cbar\smm\PPP_F$ such that $k(\g,\be)\neq 0$. Suppose that
$F^{-1}(\g)$ has a component $\de$ homotopic to $\al$ rel $\PPP_F$. Then
$$
\deg(F:\de\to\g)\ge\dfrac{m\cdot k(\al, \be)}{k(\g,\be)},
$$
where $m=m(F,\be)$.
\end{lemma}

\begin{proof} Denote by $d(\delta)=\deg(F:\de\to\g)$. We may assume that $\#(\g\cap\be)=k(\g, \be)$. Denote by $\al_1, \ldots,\al_m$ the components of $F^{-1}(\be)$. All of them are homotopic to $\beta$ rel $\PPP_F$. Thus,
$$
\#\Big(\de\cap\bigcup_{p=1}^{m}\al_p\Big)=d(\delta)\cdot\#(\g\cap \be)= d(\delta)\cdot k(\g,\be)\ .
$$
On the other hand,
$$
\#\Big(\de\cap\bigcup_{p=1}^{m} \al_p\Big)= \sum_{p=1}^{m}\#(\de\cap \al_p) \ge
\sum_{p=1}^{m} k(\al, \be) =m\cdot k(\al,\be)\ .
$$
Combining the above two inequalities we get the lemma.
\end{proof}

\begin{lemma}\label{entry}
Suppose that $\G=\{\g_1,\cdots, \g_n\}$ is an irreducible multicurve of $F$ such that $k_i=k(\g_i, \be)\neq 0$. Let $M_{\G}=(a_{ij})$ be the transition matrix of $\G$. Then
$$
a_{ij}\le \frac{d_0 k_j^2}{m k_i^2}\ ,
$$
where $d_0=d(F,\be)$ and $m=m(F,\be)$.
\end{lemma}

\begin{proof}
We may assume that $\#(\g_i\cap\be)=k(\g_i, \be)=k_i$ for any $\g_i\in\G$. Fix a pair $(i,j)$. If $F^{-1}(\g_j)$ has no component homotopic to $\g_i$, then $a_{ij}=0$. Now suppose that $F^{-1}(\g_j)$ has $n>0$ components homotopic to $\g_i$. Denote them by $\{\de_s, s=1,\cdots,n\}$. We claim that $n\le d_0k_j/k_i$.

Denote by $U, V$ the two components of $\cbar\smm\be$. Denote by $U_1$, $V_1$ the two disc components of $\cbar\smm F^{-1}(\be)$ such that $U_1, V_1$ are homotopic to $U$ and $V$, respectively. Denote $d_1=\deg (F|_{U_1})$ and $d_2=\deg (F|_{V_1})$. Then both $U\cap\g_j$ and $V\cap\g_j$ have exactly ${k_j}/2$ components (notice that $k_j=\#(\g_j\cap\be)$ is an even number). It follows that $U_1\cap F^{-1}(\g_j)$ has exactly $d_1k_j/2$ components and $V_1\cap F^{-1}(\g_j)$ has exactly $d_2k_j/2$ components.

On the other hand, since both $\partial U_1$ and $\partial V_1$ are homotopic to $\be$ rel $\PPP_F$, both $U_1\cap\de_s$ and $V_1\cap\de_s$ have at least $k_i/2$ components for $s=1,\cdots, n$. It follows that
$$
\begin{array}{rl}
\dfrac{nk_i}2 & \le\#\Big\{\text{components of }U_1\cap(\cup\de_s )\Big\} \\
& \le \#\Big\{\text{components of }U_1\cap F^{-1}(\g_j)\Big\}=\dfrac{d_1k_j}2\ .
\end{array}
$$
So $n\le d_1k_j/k_i$. We also have $n\le d_2k_j/k_i$ if we replace $U_1$ by $V_1$ in the above inequality. Hence
$n\le\min\{d_1,d_2\}k_j/k_i\le d_0k_j/k_i$.

Now applying Lemma \ref{Deg}, we have
$$
a_{ij}=\sum_{s=1}^n \frac1{\deg(F:\de_s\to\g_j)}\le \sum_{s=1}^n\frac{k_j}{m k_i}= n\frac{k_j}{m k_i} \le
d_0\frac{k_j}{k_i}\frac{k_j}{m k_i}=\frac{d_0k_j^2}{m k_i^2}\ .
$$
This proves the lemma.
\end{proof}

\noindent {\it Proof of Theorem \ref{no1}}. Denote $m=m(F,\be)$, $d_0=d(F,\be)$ and $p=\#\PPP_F$. By the condition $d_0<m$, there is an integer $N\ge 1$ such that $(p-3)d_0^N<m^N$.

Consider now the folding map $(F^N,\be)$. Note that $\#\PPP_{F^N}=p$, $m(F^N,\be)=m^N$ and $d(F^N,\be)=d_0^N$. Let
$\G=\{\g_1,\cdots, \g_n\}$ be an irreducible multicurve of $F^N$. If $k(\g_i,\be)=0$ for some $\g_i\in\G$, then $k(\g_j,\be)=0$ for all $\g_j\in\G$ by Lemma \ref{intersection}. So $\G$ is homotopic rel $\PPP_F$ to a multicurve disjoint from $\be$. Thus, it is not a Thurston obstruction by Theorem \ref{polynomial}.

Now, assume that $k_i=k(\g_i,\be)\neq 0$ for all $\g_i\in\G$. Set the vector ${\bf v}=(v_i)$ with $v_i=1/k_i^2$, then by Lemma \ref{entry}
$$
(M_{\G}{\bf v})_i=\sum_{j=1}^n a_{ij} v_j\le \sum_{j=1}^n \frac{d_0^Nk_j^2}{m^Nk_i^2}\cdot \frac1{k_j^2}= \frac{n\,d_0^N}{m^Nk_i^2}.
$$
Notice that $n=|\G|\le \#\PPP_{F^N}-3=p-3$. Therefore,
$$
(M_{\G}{\bf v})_i\le
\frac{(p-3)d_0^N}{m^Nk_i^2}< \frac1{k_i^2} = v_i\ .
$$
It follows that $M_\G{\bf v}<{\bf v}$ and hence $\la_\G<1$ \cite[Lemma A.1]{CT}. Thus $F^N$ has no Thurston obstruction by Lemma \ref{Mc}. This implies that $F$ has no Thurston obstruction. \qed

\vskip 0.24cm

\noindent {\it Proof of Theorem \ref{no2}}. Denote $m=m(F,\be)$. Let $\G=\{\g_1,\cdots, \g_n\}$ be an irreducible multicurve of $F$. If $k(\g_i,\be)=0$ for some $\g_i\in\G$, then $k(\g_j,\be)=0$ for all $\g_j\in\G$ by Lemma \ref{intersection}. So $\G$ is homotopic rel $\PPP_F$ to a multicurve disjoint from $\be$. Thus it is not a Thurston obstruction by Theorem \ref{polynomial}.

Now we assume that $k(\g_i,\be)\neq 0$ for all $\g_i\in\G$. Denote by $k(\g_i, T)$ the geometric intersection number.

{\it Case 1. $k(\g_i, T)=0$ for some curve $\g_i\in\G$}. Assume that $\g_i$ is disjoint from $T$. By the condition, there exists an integer $k\ge 1$ such that $F^{-k}(T)$ has a component homotopic to $\beta$ rel $\PPP_F$. Thus, $k(\de, \beta)=0$ for all the components $\de$ of $F^{-k}(\g_i)$. But $\G$ is irreducible, so $k(\g_j, \beta)=0$ for some $\g_j\in\G$. This is a contradiction.

{\it Case 2. $k(\g_i, T)\neq 0$ for all curves $\g_i\in\G$}. We assume $p=1$ and $F(T)\subset T$ for the simplicity. We may assume that $\#(\g_i\cap T)=k(\g_i, T)$ for $\g_i\in\G$. Let $\de_s$ ($s=1,\cdots, n$) be all the components of $F^{-1}(\g_j)$ homotopic to a curve in $\G$, then $F: (\cup \de_s)\cap T\to\g_j\cap T$ is injective since $F(T)\subset T$ and $F|_T$ is injective. So
$$
\ds\sum_{s=1}^n\#(\de_s\cap T)\le\#(\g_j\cap T)=k(\g_j, T).
$$
Therefore, if $\g_i$ is homotopic to a curve in $F^{-1}(\g_j)$, then $k(\g_i, T)\le k(\g_j, T)$. Since $\G$ is irreducible, for each pair $(i,j)$ with $1\le i,j\le n$, there exist integers $i_0=i, i_1, \cdots, i_{p}=j$ in $[1,n]$ such that $\g_{i_s}$ is homotopic rel $\PPP_F$ to a curve in $F^{-1}(\g_{i_{s+1}})$ for $0\le s<p$ and $\g_{i_p}$ is homotopic rel $\PPP_F$ to a curve in $F^{-1}(\g_{i_{0}})$. Thus $k(\g_i, T)\le k(\g_j, T)$. This implies that $k(\g_i, T)$ is a constant for all $\g_i\in\G$. Moreover $F^{-1}(\g_j)$ has exactly one component homotopic to a curve in $\G$. Relabel indices such that $F^{-1}(\g_{j+1})$ has a curve $\de_j$ homotopic to $\g_j$ for $j=1,\cdots, n$ (set $\g_{n+1}=\g_1$). Let $M_{\G}=(a_{ij})$ be the transition matrix of $\G$. Then
$$
a_{j,j+1}=\dfrac{1}{\deg( F:\de_j\to\g_{j+1})}
$$
for $j=1,\cdots, n$ (set $a_{n,n+1}=a_{n,1}$), and $a_{ij}=0$ otherwise. By Lemma \ref{Deg}, we have
$$
\deg(F:\de_{j}\to\g_{j+1})\ge \dfrac{mk_{j}}{k_{j+1}}.
$$
So $a_{j, j+1}\le k_{j+1}/(mk_{j})$. Define the vector ${\bf v}=(1/k_i)_{1\leq i\leq n}$, then $M_{\G}\,{\bf v}\le (1/m) {\bf v}$ and hence $\la_\G\le 1/m<1$ \cite[Lemma A.1]{CT}. Applying Lemma \ref{Mc} again we conclude that $F$ has no Thurston obstructions. \qed

\subsection{Sierpinski maps}

A connected compact subset $E\subset\cbar$ is called a {\bf Sierpinski carpet}\ if it is locally connected, nowhere dense and its complementary components are Jordan domains with pairwise disjoint closures.

A post-critically finite and hyperbolic rational map is called a {\bf Sierpinski map}\/ if its Julia set is a Sierpinski carpet. A post-critically finite and hyperbolic polynomial is called of {\bf Sierpinski type}\ if it has at least two bounded Fatou domains and any two of them have disjoint closures.

\begin{theorem}\label{Sierpinski}
Let $f$ be a post-critically finite and hyperbolic rational map and $\beta\subset\cbar\smm\PPP_f$ be an essential Jordan curve. Suppose that $(f,\be)$ is a folding of a Sierpinski type polynomial (or a pair of Sierpinsky type polynomials). Then $f$ is a Sierpinski map.
\end{theorem}

\begin{proof} We will prove the theorem for type A. The case of type B is similar.

Note that $\{\be\}$ is a Cantor multicurve of $f$. Applying Theorem \ref{as}, there exists an annulus $A\subset\cbar\smm\PPP_f$ homotopic to $\be$ rel $\PPP_f$, such that $f: f^{-1}(A)\to A$ is an exact annular system.

Denote by $B_1, B_2$ the two components of $\cbar\smm A$. They are {\em a fortiori} components of $f^{-1}(B_1\cup B_2)$. Rearranging the indices if necessary, we may assume that $f(B_1)=f(B_2)=B_1$. There exists a Jordan curve $\beta_0$ contained essentially in $A$ such that $U_1\Subset U_0$, where $U_0$ is the domain enclosed by $\beta_0$ and containing $B_1$ and $U_1$ is the pre-image of $f^{-1}(U_0)$ containing $B_1$. Since $\be_0$ is homotopic to $\be$ rel $\PPP_f$, the polynomial-like map $g_1=f|_{U_1}: U_1\to U_0$ is quasiconformally conjugate to a restriction of the polynomial $g$.

Each periodic Fatou domain of $f$ is contained in $B_1$. Thus it is a periodic Fatou domain of the polynomial-like map $g_1$. Since the polynomial $g$ is of Sierpinski type, every periodic Fatou domain of $f$ is a Jordan domain. Since $f$ is hyperbolic, every Fatou domain of $f$ is a Jordan domain.

Any two Fatou domains of $f$ are either contained in the same component of $f^{-n}(B_1)$ for some integer $n\ge 0$ and hence they have disjoint closures since $g$ is of Sierpinski type, or separated by a component of $f^{-m}(A)$ for some integer $m\ge 0$ and hence have disjoint closures. So $f$ is a Sierpinski map.
\end{proof}

\vskip 0.24cm

{\bf Remark}. Start with a Sierpinsky type polynomial $g$. Using the ideas in the proof of Theorem \ref{apply2}, one may construct a folding $(G,\al)$ of $g$ such that it satisfies the condition of Theorem \ref{no1} and each cycle in $\PPP_G$ is super-attracting. Thus $G$ is Thurston equivalent to a Sierpinsky map. We will see that the folding map $(G,\al)$ does not satisfy the condition of Theorem \ref{no2}.

Through a Thurston equivalence, we may assume that $G^{-1}(U)$ has a disc component $U_1\Subset U$ and $G|_{U_1}=g$, where $U$ is a component of $\cbar\smm\beta$. For any injective tree $T$ of $(G, \al)$, there exists a homeomorphism $\theta$ of $\cbar$ isotopic to the identity rel $\PPP_G$ such that $\theta(T)$ consists of internal rays in periodic Fatou domains of $g$. Therefore $\theta(T)$ is contained in one Fatou domain of $g$ except endpoints since no two bounded Fatou domains of $g$ touch together. In particular, $\theta(T)$ contains exactly one attracting periodic point. Consequently, every component of $G^{-k}(T)$ for each $k\ge 1$ is not homotopic to $\al$ rel $\PPP_G$. This shows that the folding map $(G, \al)$ does not satisfy the condition of Theorem \ref{no2}.

\subsection{An example}\label{ex}

We will give an example of the folding of polynomials which has Thurston obstructions.

Let $Q_c(z)=z^2+c$ be the airplane quadratic polynomial, i.e. the parameter $c$ is chosen to be the unique real solution of the equation $(c^2+c)^2+c=0$. The critical point $z=0$ forms a super-attracting cycle with period $3$. Let $x_0<0$ be the unique fixed point of $Q_c$ on the negative real axis (it is called the $\al$-fixed point). Then $Q_c^{-2}(x_0)$ has
four points $ x_{-1}, x_0,  x_1, x_2$, located in $\R$ relative to the super-attracting cycle as follows:
$$
c<x_{-1}<x_0<0<x_1<c^2+c<x_2.
$$
Denote by $R(\theta)$ the external ray of $Q_c$ of angle $2\pi\theta$. It is also a ray of $g$. Both $R(1/3)$ and $R(2/3)$ land on the $\al$-fixed point $x_0$. Denote $L_0=R(1/3)\cup\{x_0\}\cup R(2/3)$. Then $Q_c(L_0)=L_0$ . Now $Q_c^{-2}(L_0)$ has 4 arcs $L_{-1},L_0, L_1, L_2$ with $x_i\in L_i$.

\vskip 0.24cm

Set $\rho(z)=z/(1-|z|)$. It is a homeomorphism from the unit disk $\D$ to $\C$. Define $G_1(z)=\rho^{-1}\circ Q_c^2\circ\rho(z): \D\to\D$. It can be extended continuously to the boundary with $G_1(z)=z^{4}$ on $\beta:=\partial\D$.

Set $\sigma(z)=2/z$. It is a homeomorphism from $V=\{z:\, |z|>2\}\cup\{\infty\}$ to $\D$. Define $G_2=G_1\circ\sigma(z): \overline V\to\overline\D$. Then $G_2(z)=(2/z)^4$ on $\be_*=\partial V$.

Set $A=\{z:\, 1<|z|<2\}$. Define $G_3: A\to A\cup\overline{V}$ to be a branched covering such that its boundary value coincides with $G_1$ on $\be$, and $G_2$ on $\be_*$, and its critical values are contained in $V$. The precise definition of $G_3$ requires some care in order to control the obstructions of $F$.

Denote $S_k:=\rho^{-1}(L_k)$ for $k=-1,0,1,2$. The two ends of $S_k$ are $(e^{2\pi i\theta_k}, e^{2\pi i\varphi_k})$, with $\theta_{-1}=5/12$, $\theta_0=1/3$, $\theta_1=1/6$, $\theta_2=1/12$ and $\varphi_k=-\theta_k$. As $G_1=\rho^{-1}\circ g\circ\rho$, we have $\ds G_1^{-1}(S_0)=\cup_{k=-1}^2 S_k$.

Denote $E_k=\sigma^{-1}(S_k)$.  As $G_2=G_1\circ\sigma$, we have $G_2^{-1}(S_0)=\cup_{k=-1}^2 E_k$. As $\sigma(re^{2\pi i \eta})=\rho(2/(re^{2\pi i \eta}))$ and $\varphi_k=-\theta_k$  we know that the two ends of the arc $E_k$ are $(2e^{2\pi i\theta_k}, 2e^{2\pi i\varphi_k})$. Denote
$$
I_k=\Big\{re^{2\pi i\theta_k}, r\in [1,2]\Big\}\cup \Big\{re^{2\pi i\varphi_k}, r\in [1,2]\Big\}.
$$
Then $\g_k:=S_k\cup E_k\cup I_k $ is a Jordan curve.

\begin{figure}[h]\centering
\includegraphics[scale=0.7]{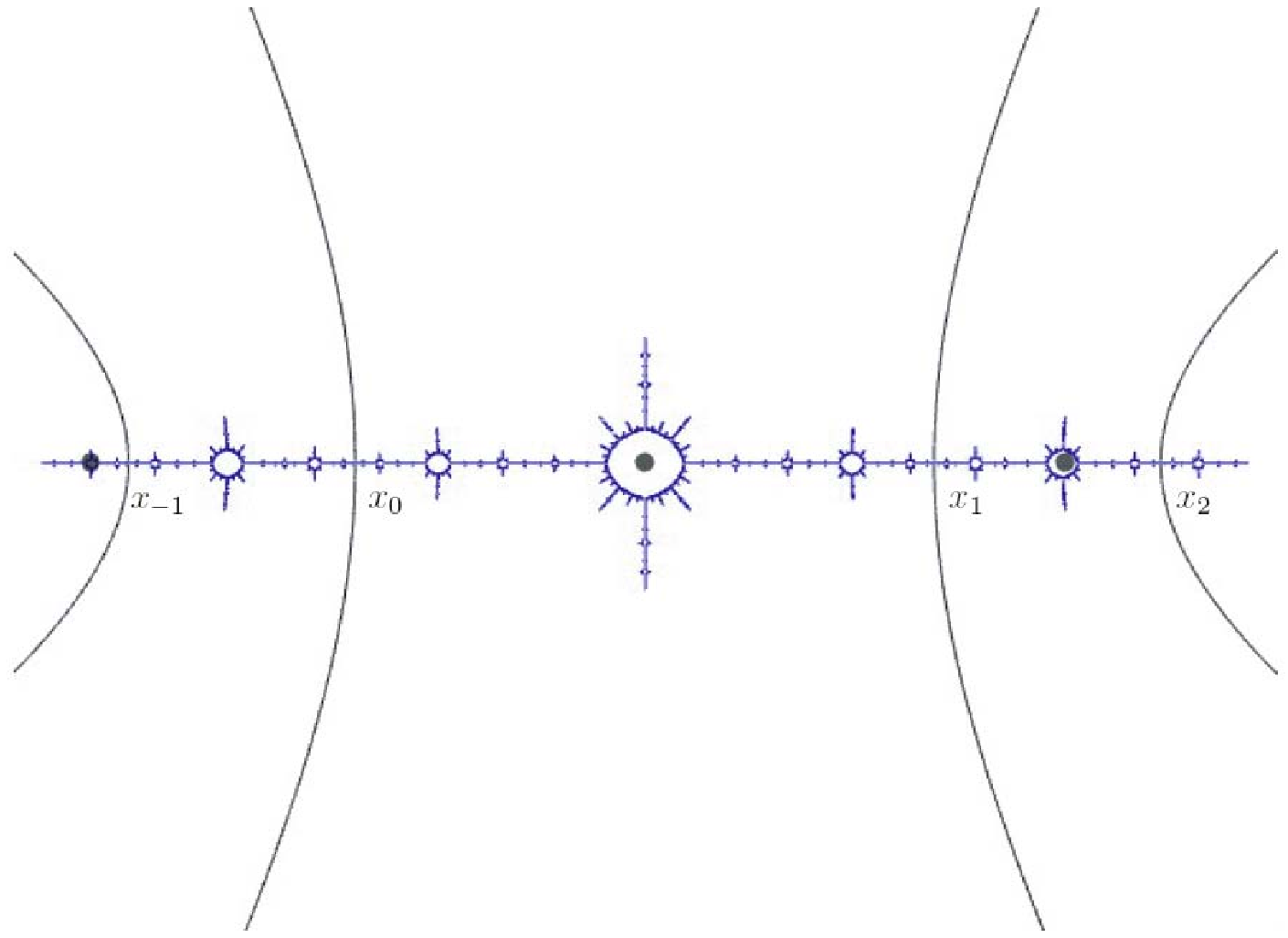}
\begin{center}{\sf Figure 2. The four arcs $L_{-1}, L_0, L_1, L_2$ (from left to right)}\end{center}
\end{figure}

Define $G_3$ on each of the $8$ arcs in $A$ such that it maps the arc homeomorphically onto $ I_0\cup E_0$, and maps $(e^{2\pi i t}, 2e^{2\pi i t})$ to $(e^{2\pi i (4t)}, e^{-2\pi i (4t)})$.

Extend $G_3$ continuously in each of the $8$ quadrilaterals of $A\smm(\cup_k I_k)$ as a orientation preserving branched covering of degree two. The image must be one of the two components of $(A\cup\overline{V})\smm(I_0\cup E_0)$. We also assure that the unique critical point in each quadrilateral is mapped to either $y_c=\sigma^{-1}(\rho^{-1}(c))$ or $y_0=\sigma^{-1}(0)$. Define
$$
F=\left\{\begin{array}{lll}
G_1: & \overline\D\to \overline\D, \\
G_2: & \overline V\to \overline\D, \\
G_3: & A\to A\cup\overline{V}.
\end{array}\right.
$$
Then $(F,\beta)$ is a folding of the polynomial $Q_c^2$ with $\lambda_\beta=1/2$.

The post-critical set $\PPP_{F}=\rho^{-1}(\PPP_{Q_c})\cup\{y_0, y_c\}$. The curve $F^{-1}(\g_0)$ has $4$ components $\g_k\ni x_k$ ($k=-1,0,1,2$) with $\deg(F_1:\g_k\to\g_0)=2$. The curve $\g_2$ is null-homotopic and $\g_{1}$ is peripheral. Both $\g_0$ and $\g_{-1}$ are essential and homotopic to each other rel $\PPP_{F}$. Set $\G=\{\g_0\}$. Then $\G$ is a stable multicurve with $\lambda_\G=1/2+ 1/2=1$. So $\G$ is a Thurston obstruction.

\subsection{Perspectives}

Here are some remarks and several problems related to this work.

(1) If a renormalization in our decomposition admits again a Cantor multicurve, one can join it to the original one and then eventually get a maximal Cantor multicurve under inclusion. In this way we can declare that the renormalizations are post-critically finite rational maps without Cantor multicurves. Maximal Cantor multicurves might not be unique. Therefore, our decomposition need not be canonical. This phenomenon occurs also in realizing rational maps as matings of polynomials.

(2) The existence of simply connected wandering continua is a very interesting problem and remains largely open. It is known that a Latt\`{e}s map has a simply connected wandering continuum if and only if it is flexible \cite{Cui-Gao}. We conjecture that every post-critically finite rational maps have no simply connected wandering continuum except for flexible Latt\`{e}s maps. One possible attempt to address the problem is to study the mating of two polynomials at first.

(3) It is conjectured that a hyperbolic component of rational maps whose Julia sets are Sierpinski carpets has compact closure in the parameter space \cite{Mc3}. Many such rational maps were constructed in \cite{HP}. One may study this conjecture just for a Sierpinsky rational map $f$ which is a folding of a polynomial. Applying Theorem \ref{as} for $f$, we get a multi-annulus consisting of one annulus. This annulus is well-defined for any rational map in the hyperbolic component containing the map $f$. We can prove that the modulus of this annulus is bounded from above in this hyperbolic component. We do not know if it is bounded from below, which would be sufficient to solve the conjecture for this hyperbolic component.

(4) Expanding Thurston type branched coverings have attracted much attention in recent years. We hope that some of the techniques developed in this work can be adapted to these expanding maps.

\appendix

\section{Rees-Shishikura's semi-conjugacy}

Let $F$ be a formal mating of two post-critically finite polynomials. Suppose that $F$ is Thurston equivalent to a rational map $f$. There is a semi-conjugacy from $F$ to $f$. This result was obtained by M. Rees if both polynomials are hyperbolic \cite{Rees}, and proved by Shishikura in general \cite{Shi}. The same result is still true for general post-critically finite branched coverings. Here we provide a statement in a general form.

\begin{theorem}\label{shi}
Let $F:\cbar\to\cbar$ be a post-critically finite branched covering which is Thurston equivalent to a rational map $f$ through a pair of homeomorphisms $(\psi_0, \psi_1)$ of $\cbar$. Suppose that $F$ is holomorphic in a neighborhood of the critical cycles of $F$. Then there exist a neighborhood $U$ of the critical cycles of $F$ and a sequence of homeomorphisms $\{\phi_n\}$ $(n\ge 0)$ of $\cbar$ homotopic to $\psi_0$ rel $\PPP_F$, such that $\phi_n|_U$ is holomorphic, $\phi_n|_U=\phi_0|_U$ and $f\circ\phi_{n+1}=\phi_n\circ F$. The sequence $\{\phi_n\}$ converges uniformly to a continuous map $h:\, \cbar\to\cbar$.
Moreover, the following statements hold: \\
\indent (1) $h\circ F=f\circ h$. \\
\indent (2) $h$ is surjective. \\
\indent (3) $h^{-1}(w)$ is a single point for $w\in \FFF_f$ and $h^{-1}(w)$ is either a single point or a full continuum for $w\in \JJJ_f$. \\
\indent (4) For points $x,y\in\cbar$ with $f(x)=y$, $h^{-1}(x)$ is a component of $F^{-1}(h^{-1}(y))$.
Moreover, $\deg F|_{h^{-1}(x)}=\deg_x f$. \\
\indent (5) $h^{-1}(E)$ is a continuum if $E\subset\cbar$ is a continuum. \\
\indent (6) $h(F^{-1}(E))=f^{-1}(h(E))$ for any $E\subset\cbar$. \\
\indent (7) $F^{-1}(\wh E)=\wh{F^{-1}(E)}$ for any $E\subset\cbar$, where $\wh E=h^{-1}(h(E))$.
\end{theorem}

One may also look at \cite{CPT} for a detailed account in a generalized form. The crucial part of the theorem is the construction of the homotopy $(\phi_0, \phi_1)$ rel a neighborhood $U$ of critical cycles and the convergence of the sequence $\{\phi_n\}$. The other statements are directly deduced. The statements (5)-(7) is used in this paper, so we add them in the theorem and provide a proof here.

\vskip 0.24cm

{\noindent\it Proof of (5)-(7)}.

(5) Suppose that $E\subset\cbar$ is a connected closed subset. The fact that $h^{-1}(E)$ is closed, is easy to see. Now suppose that $h^{-1}(E)$ is not connected, i.e., there are open sets $U_1$, $U_2$ in $\cbar$ such that $h^{-1}(E)\subset U_1\cup U_2$, $U_1\cap U_2=\emptyset$ and both $U_{1}$ and $U_2$ intersect $h^{-1}(E)$. Then $K:=h(\cbar\smm (U_1\cup U_2))$ is a compact set disjoint from $E$. Since $E$ is connected, there is a connected
neighborhood $V$ of $E$ such that $\overline{V}\cap K=\emptyset$. Since $\{\phi_n\}$ converges uniformly to $h$, there exists an integer $n>0$ such that
$$
\text{d}(h,\phi_n)=\sup_{z\in\cbar}\text{d}\Big(h(z),\phi_n(z)\Big)<\min\Big\{\text{d}(E,\partial
V), \text{d} (\overline{V},K)\Big\},
$$
where $d(\cdot, \cdot)$ denotes the spherical distance. It follows that $\phi_n(\cbar\smm (U_1\cup U_2))\cap\overline{V}=\emptyset$, hence $\phi_n^{-1}(V)\subset U_{1}\cup U_{2}$. On the other hand, since $V\supset E$, both $U_1$ and $U_2$ intersect $\phi_n^{-1}(V)$. This contradicts the fact that $\phi_n^{-1}(V)$ is connected.

(6) From $f\circ h(F^{-1}(E))=h\circ F(F^{-1}(E))=h(E)$, we have $h(F^{-1}(E))\subset f^{-1}(h(E))$. Conversely, for any point $w\in f^{-1}(h(E))$, $f(w)\in h(E)$. So there is a point $z_0\in E$ such that $f(w)=h(z_0)$. By (5), the map
$$
F: h^{-1}(w)\to h^{-1}\Big(f(w)\Big)
$$
is surjective. Noticing that $z_0\in h^{-1}(f(w))$, there is a point $z_1\in h^{-1}(w)$ such that $F(z_1)=z_0$. So $w=h(z_1)\in h(F^{-1}(z_0))\subset h(F^{-1}(E))$. Therefore, $f^{-1}(h(E))\subset h(F^{-1}(E))$.

(7) $F^{-1}(\wh E)=F^{-1}\Big(h^{-1}\big(h(E)\big)\Big)=h^{-1}\Big(f^{-1}\big(h(E)\big)\Big)$. From (6), we obtain
$$
F^{-1}(\wh E)=h^{-1}\Big(h\big(F^{-1}(E)\big)\Big)=\wh{F^{-1}(E)}.
$$
\qed

\section{Buff's example}

{\bf Example}. Denote by $U=\{z: 1<|z|<r_0\}$ with $r_0>2$. Define a spiral in $U$ by:
$$
\delta=\Big\{\rho e^{i\theta}: \dfrac{r_0}2\le\rho<r_0, \,
\theta=\dfrac{1}{r_0-\rho}\Big\}.
$$

Set $A_1=U\smm\de$. Then $A_1$ is an annulus with modulus $\m(A_1)<\log r_0/(2\pi)$. Pick an integer $d>2$ such that $d\,\m(A_1)>\log r_0/\pi$. Let $r_1>1$ be such that $\log r_1/(2\pi)=d\,\m(A_1)$. Then $r_1>r_0^2$.
Set $A=\{z: 1<|z|<r_1\}$ and $A_2=h(A_1)$ with $h(z)=r_1/z$. Then $A_2$ is disjoint from $A_1$ and there is a holomorphic covering $g$ of degree $d$ from $A_i$ ($i=1,2$) to $A$ such that $g$ fixes the two components of $\partial A$. Then $g: A_1\cup A_2\to A$ is an exact annular system .

\begin{figure}[htbp]\centering
\includegraphics[width=9cm]{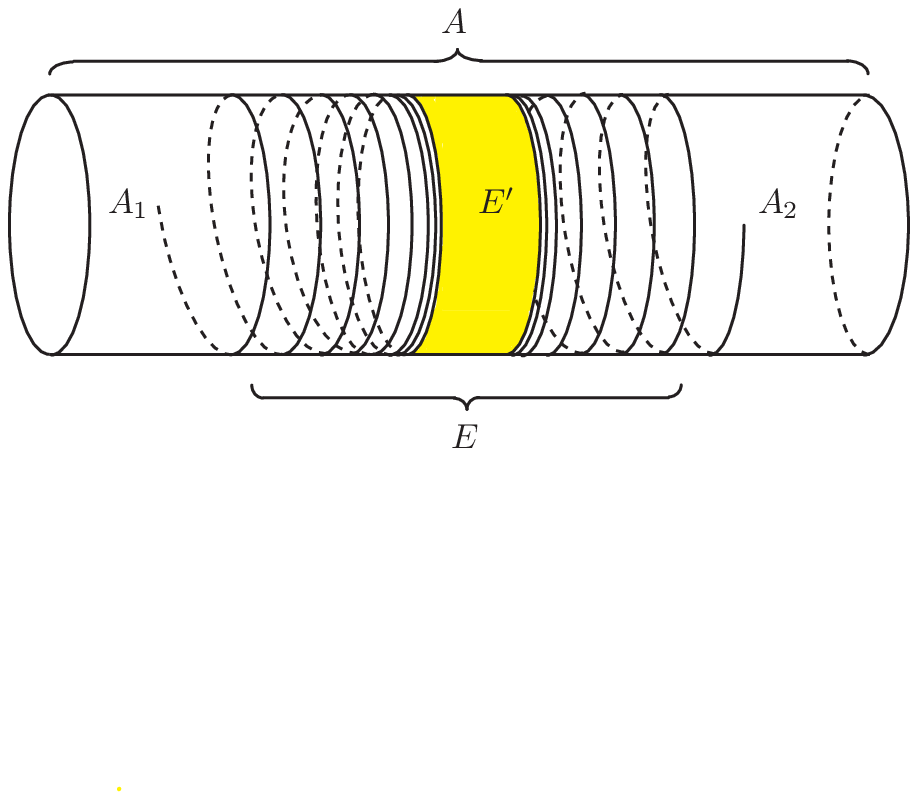}\vspace{-2mm}
\begin{center}{\sf Figure 3. An annular system with spirals}\end{center}
\end{figure}

\begin{theorem}{\label{Buff}}
Let $\mathscr{J}$ be the collection of the components of the Julia set of the exact annular system $g: A_1\cup A_2\to A$. With the topology induced by the corresponding linear system, $\mathscr{J}$ has a dense subset whose elements are not locally connected.
\end{theorem}

Set $B=\{\zeta, 0<\text{Im }\zeta<\log r_1\}$. Then $\pi(\zeta)=\exp(\zeta):\, B\to A$ is a universal covering. Denote $E_0=A\smm(A_1\cup A_2)$. For each component $E_n$ of $g^{-n}(E_0)$ ($n\ge 0$) and  any point $z\in A\smm E_n$, denote $$
\text{h-dist}(z, \partial A; E_n)=\inf\Big\{\text{diam}(\pi^{-1}(\gamma))\Big\},
$$
where the infimum is taken over all the open arc $\gamma$ in $A\smm E_n$ connecting $z$ with $\partial A$, and
$\text{diam}(\pi^{-1}(\gamma))$ is the Euclidean diameter of a component of $\pi^{-1}(\gamma)$.

Denote $E'_0=A\smm(U\cup h(U))=E_0\smm(\de\cup h(\de))$. It is easy to check that for any constant $M<\infty$, there exists a constant $\varepsilon>0$ such that for any point $z\in A\smm E_0$, if the Euclidean distance $\text{dist}(z,E'_0)<\varepsilon$, then $\text{h-dist}(z, \partial A; E_0)>M$.

For each component $E_n$ of $g^{-n}(E_0)$ and any $k>n\ge 0$, denote by $V_k(E_n)$ the union of the two components of $g^{-k}(A)$ whose closures meet $E_n$. Then $\wt V_k(E_n):=V_k(E_n)\cup E_n$ is an annulus with $\wt V_{k+1}(E_n)\subset\wt V_k(E_n)$ and $\cap_{k\ge n}\wt V_k(E_n)=E_n$.

By the above argument, we see that for any constant $M<\infty$, there is an integer $k(M)\ge 1$ such that for any component $K$ of $\JJJ_g\cap V_{k(M)}(E_0)$, there exists a point $z\in K$ such that $\text{h-dist}(z, \partial A; E_0)>M$. By the Ko\"{e}be distortion theorem, we may prove the next lemma.

\begin{lemma}{\label{Buff1}}
For any component $E_n$ of $g^{-n}(E_0)$ ($n\ge 0$) and any constant $M<\infty$, there is an integer $k(M,E_n)> n$ such that for any component $K$ of $\JJJ_g\cap V_{k(M,E_n)}(E_n)$, there exists a point $z\in K$ such that  $\text{h-dist}(z, \partial A; E_n)>M$.
\end{lemma}

\noindent{\it Proof of Theorem \ref{Buff}}. For each integer $m>0$ and any component $E_n$ of $g^{-n}(E_0)$ ($n\ge 1$), define $\mathscr{N}(m, E_n)$ to be the sub-collection of $\mathscr{J}$ such that $K\in \mathscr{N}(m, E_n)$ if $K\subset V_{k(m, E_n)}(E_n)$. Then $\mathscr{N}(m, E_n)$ is an open set in $\mathscr{J}$.
Set $\mathscr{N}(m)$ to be the union of $\mathscr{N}(m, E_n)$ for all $n\ge 1$ and all the components $E_n$ of $g^{-n}(E_0)$. Then it is an open dense subset of $\mathscr{J}$. Thus, $\mathscr{N}=\bigcap_{m\ge 1}\mathscr{N}(m)$ is a dense subset of $\mathscr{J}$  (in the sense of Baire category).

For each $K\in\mathscr{N}$, we want to show that $K$ is not locally connected. Otherwise $K$ is a Jordan curve and hence there is an constant $M<\infty$ such that for any point $z\in K$, there are open arcs $\de_+(z)$ and $\de_-(z)$ in $A\smm K$ connecting $z$ with the two components of $\partial A$, respectively, such that $\text{diam}(\pi^{-1}(\de_{\pm}(z)))<M$.

Fix an integer $m>M$. Since $K\in\mathscr{N}(m)$, there exist an integer $n\ge 0$ and a component $E_n$ of $g^{-n}(E_0)$ such that $K\in\mathscr{N}(m, E_n)$. Thus, there is a point $z\in K$ such that $\text{h-diam}(z, \partial A; E_n)>m>M$, contradicting the fact that $\text{diam}(\pi^{-1}(\de_{\pm}(z)))<M$.
\qed

\vskip 0.24cm

{\it Acknowledgements}. The first and second authors are supported by the NSFC grant no. 11125106 and 11101402. The third author is supported by ANR LAMBDA and Geanpyl Pays de la Loire. The authors are grateful to the referees and H.H. Rugh for helpful suggestions and comments.

\noindent
Guizhen Cui and Wenjuan Peng \\
Academy of Mathematics and Systems Science, \\
Chinese Academy of Sciences, Beijing 100190, \\
P. R. China. \\
gzcui@math.ac.cn and wenjpeng@amss.ac.cn

\vskip 0.24cm

\noindent
TAN Lei, \\
LAREMA, UMR 6093 CNRS. \\
Universit\'e d'Angers, \\
2 bd Lavoisier, Angers, 49045, \\
France. \\
Lei.Tan@univ-angers.fr

\end{document}